\documentclass{amsart}

\RequirePackage{doi}
\usepackage{hyperref}
\usepackage{cite}

\usepackage{amsmath}
\usepackage{tikz-cd}
\usepackage{amsthm}
\usepackage{hyperref}
\usepackage{epsfig}
\usepackage{flafter}
\usepackage{mathtools}
\usepackage{comment}
\usepackage{stmaryrd}

\usepackage{amssymb} 
\def\acts{\curvearrowright}

\hypersetup{
    colorlinks=true,    
    linkcolor=blue,          
    citecolor=blue,      
    filecolor=blue,      
    urlcolor=blue           
}
\usepackage{tikz}
\usetikzlibrary{graphs,positioning,arrows,shapes.misc,decorations.pathmorphing}

\tikzset{
    >=stealth,
    every picture/.style={thick},
    graphs/every graph/.style={empty nodes},
}

\tikzstyle{vertex}=[
    draw,
    circle,
    fill=black,
    inner sep=1pt,
    minimum width=5pt,
]
\usepackage[position=top]{subfig}
\usepackage{amssymb}
\usepackage{color}

\calclayout

\usetikzlibrary{decorations.pathmorphing}
\tikzstyle{printersafe}=[decoration={snake,amplitude=0pt}]

\newcommand{\codim}{\operatorname{codim}}

\newcommand{\supp}{\operatorname{supp}}

\newcommand{\pp}{\mathbb{P}}

\newcommand{\qq}{\mathbb{Q}}
\newcommand{\zz}{\mathbb{Z}}

\newcommand{\kk}{\mathbb{K}}

\def\O#1.{\mathcal {O}_{#1}}			
\def\pr #1.{\mathbb P^{#1}}				
\def\af #1.{\mathbb A^{#1}}			
\def\ses#1.#2.#3.{0\to #1\to #2\to #3 \to 0}	
\def\xrar#1.{\xrightarrow{#1}}			
\def\K#1.{K_{#1}}						
\def\bA#1.{\mathbf{A}_{#1}}			
\def\bM#1.{\mathbf{M}_{#1}}				
\def\bL#1.{\mathbf{L}_{#1}}				
\def\bB#1.{\mathbf{B}_{#1}}				
\def\bK#1.{\mathbf{K}_{#1}}			
\def\subs#1.{_{#1}}					
\def\sups#1.{^{#1}}

\usepackage{tikz}
\usetikzlibrary{matrix,arrows,decorations.pathmorphing}

  \newtheorem{introthm}{Theorem}

  \newtheorem{introcor}{Corollary}

  \newtheorem{theorem}{Theorem}[section]
  \newtheorem{lemma}[theorem]{Lemma}
  \newtheorem{proposition}[theorem]{Proposition}
  \newtheorem{corollary}[theorem]{Corollary}

  \newtheorem{notation}[theorem]{Notation}
  \newtheorem{definition}[theorem]{Definition}
  \newtheorem{example}[theorem]{Example}

  \newtheorem{question}[theorem]{Question}

\newtheorem{remark}[theorem]{Remark}

\theoremstyle{remark}

\numberwithin{equation}{section}

\usepackage[all]{xy}

\begin{document}

\title[On a toroidalization for klt singularities]
{On a toroidalization for klt singularities}

\author[J.~Moraga]{Joaqu\'in Moraga}
\address{Department of Mathematics, Princeton University, Fine Hall, Washington Road, Princeton, NJ 08544-1000, USA
}
\email{jmoraga@princeton.edu}

\maketitle

\begin{abstract}
In this article, we prove a toroidalization principle for
finite actions on klt singularities.
As an application, we prove that the Jordan property for the regional fundamental group of klt singularities can be realized geometrically: by extracting a toric singularity over the klt germ.
In the course of the proof, 
we will prove statements about finite actions on dual complexes and
almost fixed points in the fibers of equivariant Fano type morphisms.
Furthermore, we will prove that the rank of a fundamental group of the klt singularity is bounded above by its regularity.
\end{abstract}

\setcounter{tocdepth}{1} 
\tableofcontents

\section{Introduction}

Toroidalization principles play a fundamental role in the solution of many important problems 
in algebraic geometry.
In Hironaka's resolution of singularities~\cite{Hir64},
the local monomialization of 
ideal sheaves is one of the main steps. 
In Abramovich-Karu's proof of weak semistable reduction~\cite{AK00}, a crucial step is a reduction to weak toroidal embeddings.
In Wlodarczyk's proof of the weak factorization theorem~\cite{Wol06},
the toroidal version of Morse theory plays a fundamental role.
In Birkar's theorem on boundedness of Fano varieties~\cite{Bir21}, the reduction to the boundedness of toric Fano varieties is a main step of the proof. 
Furthermore,
ideas from toric geometry also appear
in semistable reduction~\cite{ALT20} and
embedded resolutions via weighted blow-ups~\cite{ATW19,McQ20}.
The aim of this article is to prove a toroidalization principle for Kawamata log terminal singularities from a topological perspective.
We expect this to be a starting point for other toroidalization principles for klt singularities.

Kawamata log terminal singularities are the main class of singularities that appear in the minimal model program~\cite{Kol13}.
In dimension two, klt singularities are quotient singularities~\cite{Tsu83}.
In dimension three, terminal singularities are classified as deformation of compound DuVal singularities~\cite{Rei87}.
In higher dimensions, the geometry of klt singularities remains a mystery (see~\cite{Kol11} for some interesting examples).
In~\cite{Bir19}, Birkar proves the existence of bounded complements for $n$-dimensional klt singularities. 
A complement is a boundary $B$ so that $(X,B)$ is log canonical and the index of $K_X+B$ at $x$ is $O_n(1)$.
The existence of bounded complements has been used to understand {\em exceptional klt singularities} in the works~\cite{Mor18a,HLS19}.
An exceptional singularity is a klt singularity that admits a unique divisorial valuation that can compute a log canonical place.
Exceptional singularities are higher-dimensional analogues of the $E_6,E_7$ and $E_8$ singularities.
By~\cite{HLM20}, in a fixed dimension, these singularities deform to cones over exceptional Fano varieties.
By~\cite{Bir19}, exceptional Fano varieties form a bounded family when their dimension is fixed. 
Hence, exceptional singularities, of a fixed dimension, are bounded up to deformation.
From the topological perspective,
exceptional singularities are rather simple, the fundamental group of the link of a $n$-dimensional exceptional singularity is finite and has order bounded by $O_n(1)$.

On the other hand, the topology of klt singularities is well-understood. 
By~\cite{Xu14}, the pro-finite completion of the local fundamental group of a klt singularity is finite.
By~\cite{Bra20}, the regional fundamental group of a klt singularity is finite. 
In~\cite{BFMS20}, the authors prove that the regional fundamental group of a $n$-dimensional klt singularity admits a normal abelian subgroup of rank at most $n$ and index at most $O_n(1)$.
This result is known as the Jordan property for klt singularities.
It means that topologically, klt singularities resemble quotient singularities. 
In~\cite{Mor20a,Mor20b}, the author studies $n$-dimensional klt singularities whose regional fundamental group contains $\zz_m^n$ for $m$ large enough.
It is proved that these singularities deform to cones over $\zz_m^n$-equivariant equivalent log crepant models of $(\pp^{n-1},H_1+\dots+H_n)$.
These singularities are called {\em klt singularities of full rank} since their fundamental groups are as large as possible.
In particular, the rank of the abelian part of the fundamental group is as large as possible.
Kawamata log terminal singularities of full rank are, in many ways, the opposite of exceptional singularities.
The former have complements with $(n-1)$-dimensional dual complex, while the latter have complements with $0$-dimensional dual complex.
The former have fundamental groups of rank $n$, while the latter have fundamental groups of rank $0$.
Recently, in~\cite{BM21}, the authors proved that the iteration of Cox rings of a Fano type variety is controlled by the topology of a klt singularity.
See~\cite{LLM19}, for related results about $\mathbb{T}$-varieties.

The previous results point that the topology of klt singularities, even if it is simple, often reflects in the geometry of the singularity.
In many of the aforementioned results, the proof goes through finding toroidal models for the singularities (or varieties).
The philosophy introduced in the articles~\cite{Mor20a,Mor20b}, is that large finite actions on Fano type varieties (or singularities) resemble toric actions;
and that, in many cases, that torus action can be realized by passing to a different birational model.
The main results of~\cite{Mor20a} and~\cite{Mor20b} are projective and mostly concern with the case of full rank, with some applications to klt singularities (that require deformations). 
In this article, we prove that finite actions on klt singularities can be toroidalized up to a birational transformation.
Many of the results stated in the introduction will hold for generalized pairs. 
However, throughout the introduction, we will give the simplest statement possible.

\subsection{Toroidalization of finite actions on klt singularities}

In this subsection, we introduce the main theorem of this article.
We start with the motivation for our first theorem.

Let $G$ be a finite group acting on a smooth $n$-dimensional variety $X$ over $\kk$.
Let $x\in X$ be a closed point which is fixed by $G$.
We have an embedding 
$G\leqslant {\rm Aut}(T_x X) \simeq 
{\rm GL}_n(\kk)$. By the Jordan property for finite subgroups of ${\rm GL}_n(\kk)$, we can find a normal abelian subgroup $A\leqslant G$ of rank at most $n$ and index at most $c_0(n)$. Say that $A$ has rank $r\leq n$. Then, we can find $r$ hyperplanes $H_1,\dots,H_r\subset X$ through $x$ so that $(X,H_1+\dots+H_r;x)$ is log smooth, 
$A$ acts as the identity on the intersection $Z:=H_1\cap \dots \cap H_r$, and $A$ acts as the multiplication by roots of unity on the log smooth germ
$(X, H_1+\dots+H_r; \eta_Z)$.
In particular, we have a monomorphism
$A<\mathbb{G}_m^r\leqslant {\rm Aut}(X,H_1+\dots+H_r;\eta_Z)$.
This means that any finite action on a $n$-dimensional smooth point,
up to pass to a subgroup of index $c_0(n)$, admits a toroidalization.
In general klt singularities may not admit torus actions, so the above monomorphism can only be attained after passing to a higher model.
Our first theorem states that any finite action on a $n$-dimensional klt singularity, up to pass to a subgroup of index $c(n)$, admits a birational toroidalization.

\begin{introthm}\label{introthm:group-toroidal}
Let $n$ be a positive integer.
There exists a constant $c(n)$, only depending on $n$, satisfying the following.
Let $(X,x)$ be a $n$-dimensional klt singularity
and $G\leqslant {\rm Aut}(X,x)$ be a finite group.
Then, there exists:
\begin{enumerate} 
\item[(i)] A normal abelian subgroup $A\triangleleft G$ of index at most $c(n)$ and rank $r$,
\item[(ii)] A $G$-equivariant boundary $B$ so that $(X,B;x)$ is log canonical, and 
\item[(iii)] An $A$-equivariant projective birational morphism $\pi\colon Y\rightarrow X$,
\end{enumerate}
so that the following statements hold:
\begin{enumerate}
    \item $\pi$ extracts $r$ prime divisors $E_1,\dots,E_r$, 
    \item each $E_i$ is $A$-invariant and a log canonical place of $(X,B)$,
    \item the intersection $E_1\cap\dots\cap E_r$ is non-trivial,
    \item $A$ fixes a component $Z$ of $E_1\cap\dots\cap E_r$,
    \item $(Y,E_1+\dots+E_r)$ is formally toric at the generic point $\eta_Z$ of $Z$, 
    \item $A$ acts as the identity on $Z$, and 
    \item $A<\mathbb{G}_m^r \leqslant {\rm Aut}(Y,E_1+\dots+E_r,\eta_Z)$.
\end{enumerate}
In particular, the action of $A$ on $Y$ is toroidal at the generic point of $Z$.
\end{introthm}

We note that Theorem~\ref{introthm:group-toroidal}
contains implicitly many results about klt singularities.
First, the action of $A$ on $Y$ fixes each divisor $E_i$ and acts as the identity on a component of their intersection.
This means that the action of $A$ on the dual complex $\mathcal{D}(X,B;x)$ fixes a $r$-dimensional simplex.
Hence, Theorem~\ref{introthm:group-toroidal} is related to the existence of almost fixed simplices for the action of finite groups on dual complexes. 
Here, almost fixed means that the action is fixed for a subgroup whose index is bounded by some constant on the dimension.
We will discuss this result in a more general setting in Subsection~\ref{subsec:intro-finite-act-dual-comp}.
Secondly, we not only have an almost fixed point on the dual complex.
Indeed, since $A$ acts as the identity on $Z$, we have an almost fixed point for the $G$-action in the fiber $\pi^{-1}(x)$.
We will prove that given a $G$-equivariant Fano type morphism 
$\pi\colon Y\rightarrow X$, for every fixed point in the base, there is an almost fixed point in the fiber over the fixed point. This result and some application will be explained in Subsection~\ref{subsec:intro-almost-fix-fiber}.
Thirdly, in the statement of Theorem~\ref{introthm:group-toroidal}, we can see that existence of a large finite group acting on a klt germ implies the existence of enough log canonical places.
More precisely, the action of a large abelian group of rank $r$ on the klt singularity $(X;x)$ implies that we can find a boundary $B$ so that $(X,B)$ is log canonical around $x$, and the dual complex $\mathcal{D}(X,B;x)$ has dimension at least $r-1$.
This says that the rank of the acting group on the singularity is at most the regularity of the singularity, i.e., 
the largest dimension of a dual complex of a lc singularity supported on $(X;x)$.
We will state a general statement about the rank and regularity of log Calabi-Yau pairs of Fano type in Subsection~\ref{subsec:intro-rank-vs-reg}.
In Subsection~\ref{subsec:intro-p1}, we will show some implications of Theorem~\ref{introthm:group-toroidal} 
to the study of fundamental groups of Fano type varieties and klt singularities.
In what follows, we proceed to elaborate more on these results and introduce some global results as well.

\subsection{Finite actions on dual complexes}\label{subsec:intro-finite-act-dual-comp}
In this subsection, we collect some results regarding finite actions on dual complexes.
The dual complex of a log canonical singularity is defined in~\cite{dFKX17}.
This object retrieves the combinatorial structure of the log canonical places of a dlt modification.
In~\cite{dFKX17}, it is proved that the PL-homeomorphism type of the dual complex does not depend on the chosen dlt modification.
In~\cite{KX16}, the authors study the dual complex of log Calabi-Yau pairs $(X,\Delta)$, its fundamental group, and connections with the fundamental group of the smooth locus of $X$.
In~\cite[Question 4]{KX16} it is asked whether the dual complex of a $n$-dimensional log Calabi-Yau pair is PL-homeomorphic to the quotient of a sphere $\mathbb{S}^{k-1}$ by a finite group $G< O_k(\mathbb{R})$ where $k\leq n$. This question naturally motivates to study which finite actions on dual complexes can be realized.
For instance, for an arbitrary $m$, is it possible to obtain $\zz_m \acts \mathcal{D}(X,\Delta)\simeq \mathbb{S}^1$
where $\zz_m$ acts as the rotation by $\frac{2 \pi}{m}$.
Note that in the case of rotations of the circle,
the action is not detected by the quotient.
Our following theorem states that the action is almost trivial with respect to the dimension of the log pair.
In particular, once we fix the dimension of $X$, 
the possible realizable rotations on the dual complex of a log Calabi-Yau pair $(X,B)$ of regularity one will attain only finitely many possible angles.

\begin{introthm}\label{introthm:almost-fixed-global}
Let $n$ be a positive integer.
There exists a constant $c(n)$,
only depending on $n$,
satisfying the following.
Let $(X,B)$ be a log Calabi-Yau pair 
of dimension $n$
and $G\leqslant {\rm Aut}(X,B)$ be a finite subgroup.
Then, there exists a normal subgroup
$A\triangleleft G$ of index at most $c(n)$ so that
the action 
of $A$ on $\mathcal{D}(X,B)$ is trivial.
\end{introthm}

The above theorem is motivated by toric geometry.
Let $(T,B_T)$ be a projective log Calabi-Yau toric pair of dimension $n$.
Let $G$ be a finite group acting on the log pair $(T,B_T)$.
By~\cite[Corollary 1]{Mor20c}, we know that there exists a normal abelian subgroup $A\triangleleft G$ of index bounded by a constant in $n$, so that 
$A<\mathbb{G}_m^n\leqslant {\rm Aut}(T,B_T)$.
Since $\mathbb{G}_m^n$ fixes all the toric strata, we conclude that the same happens for $A$.
Thus, $A$ acts as the identity on the dual complex $\mathcal{D}(T,B_T)$.
The above argument also goes through if we consider toric singularities.
Let $(T,B_T; t)$ be the germ of a toric singularity so that $B_T$ is the reduced torus invariant boundary. 
Let $G<{\rm Aut}(T,B_T; t)$ be a finite subgroup.
By~\cite[Theorem 2]{Mor20c}, we know that there exists a normal abelian subgroup $A\triangleleft G$, of index at most $c(n)$, so that 
$A<\mathbb{G}_m^n \leqslant {\rm Aut}(T,B_T;t)$.
In particular, both $A$ and the torus $\mathbb{G}_m^n$ acts as the identity on the dual complex $\mathcal{D}(T,B_T;t)$.
Analogously, we have a similar statement for log canonical germs supported on klt singularities.

\begin{introthm}\label{introthm:almost-fixed-local} 
Let $n$ be a positive integer.
There exists a constant $c(n)$,
only depending on $n$,
satisfying the following.
Let $(X,B;x)$ be a $n$-dimensional log canonical 
singularity.
Assume that $X$ is klt.
Let $G\leqslant {\rm Aut}(X,B)$ be a finite subgroup.
Then, there exists a normal abelian subgroup
$A\leqslant G$ of index at most $c(n)$ so that
the action 
of $A$ on $\mathcal{D}(X,B;x)$ is trivial.
\end{introthm}

In general, we need to take a non-trivial subgroup to have a fixed
point on the dual complex.
For instance, we can consider the symmetric group $S_n$ acting on $\pp^n$.
This action has no fixed points on the dual complex of $\pp^n$ with its reduced toric boundary.
Furthermore, no normal subgroup will have such a fixed point.
The above example is one of the worst examples that one could expect among toric varieties.

In Theorem~\ref{introthm:almost-fixed-local}, we need the singularities to be log canonical.
Moreover, we need the germ to support a klt singularity.
Otherwise, the dual complex of a lc singularity can be rather arbitrary (see, e.g.,~\cite{Kol13}).
Thus, we do not expect the existence of an almost fixed point on the dual complex for finite actions on lc singularities (see Example~\ref{ex:elliptic}).

The strategy for the proof of both Theorem~\ref{introthm:almost-fixed-global} and Theorem~\ref{introthm:almost-fixed-local} is similar. 
We try to construct a subgroup $A\leqslant G$ and a $A$-invariant log canonical place of $(X,B)$.
Then, we extract such log canonical place $E$ in a higher model and apply the adjunction formula to it.
Then, by induction on the dimension, possibly passing to a normal subgroup of $A$, we may assume that $A$ fixes every single log canonical center of $(X,B)$ which intersects $E$.
In particular, $A$ will fix a simplex of maximal dimension in $\mathcal{D}(X,B)$.
Then, an argument using $\pp^1$-links on the simplicial complex associated to $\mathcal{D}(X,B)$, will imply that $A$ fixes $\mathcal{D}(X,B)$ pointwise.

As an application of Theorem~\ref{introthm:almost-fixed-global}, we obtain that dual complexes of log Calabi-Yau pairs of a fixed dimension can have only finitely many possible fundamental groups.

\begin{introcor}\label{introcor-finiteness-pi1-dual}
Let $n$ be a positive integer.
There exists a finite set $\mathcal{G}_n$ of finite groups satisfying the following.
Let $(X,B)$ be a log Calabi-Yau pair of dimension $n$, then $\pi_1(\mathcal{D}(X,B))\simeq G$
for some $G\in \mathcal{G}_n$.
\end{introcor}

The following corollary shows that
for $G$-equivariant $n$-dimensional log pairs $(X,B)$ with $K_X+B\sim_\qq 0$, we can still find almost fixed points on $\mathcal{D}(X,B)$ regardless of the singularities of the pair.

\begin{introcor}\label{introcor:almost-fixed-point-not-lc}
There exists a constant $c(n)$, only depending on $n$,
satisfying the following. 
Let $(X,B)$ be a projective log pair.
Let $G\leqslant {\rm Aut}(X,B)$ be a finite group.
Assume that $K_X+B\sim_\qq 0$.
Then, there exists a normal subgroup $A\triangleleft G$ so that every element of $A$ acts on $\mathcal{D}(X,B)$ with a fixed point.
\end{introcor}

\subsection{Almost fixed points in fibers}\label{subsec:intro-almost-fix-fiber}
In this subsection, we discuss the existence of almost fixed points in fibers.
The aim is a simple one: 
given a $G$-equivariant morphism $Y\rightarrow X$, and a fixed point $x\in X$, try to find an almost fixed point in the fiber over $x$.
This question is already interesting for projective varieties, i.e.,
for the structure morphism $Y\rightarrow {\rm Spec}(\mathbb{K})$.
In the case of elliptic curves or curves of general type, we can find arbitrarily large finite group actions so that no subgroup fixes a point.
These examples can be realized relatively by taking the cone
over the projective variety.
For instance, we can consider the cone over an elliptic curve $C$ and the blow-up at the vertex $X\rightarrow C$.
For each $n$-torsion point of the elliptic curve, we get the action of a group of order $n$ that makes $X\rightarrow C$ an equivariant birational morphism.
However, even if the vertex is fixed, no point of the fiber is fixed. 
Similar examples can be constructed by considering cones over curves of general type (see Example~\ref{ex:elliptic}). 
The above examples suggest that, in order to obtain an almost fixed point in the fiber, 
we need to control the singularities of the fixed point on the base. 

On the other hand, 
we know that the action of a finite group $G$ on a rationally connected variety $X$ of dimension $n$ admits an almost fixed point with respect to $n$~\cite{PS14,PS16}.
More precisely, given a positive integer $n$, there exists a constant $c(n)$ satisfying the following. If $X$ is a $n$-dimensional rationally connected variety and $G\leqslant {\rm Aut}(X)$ is a finite subgroup, then there exists a normal abelian subgroup $A\leqslant G$ of index at most $c(n)$ so that $A$ has a fixed point on $X$. 
The above projective results also apply to Fano type varieties because these are rationally connected~\cite{Zha06}.
On the other hand, the 
fibers of Fano type morphisms are rationally chain connected~\cite{HMc07}.
These three results together suggest that the fibers over fixed points of $G$-equivariant Fano type morphisms contain almost fixed points.
The next theorem shows the existence of almost fixed points in the fibers for Fano type morphisms.

\begin{introthm}\label{introthn:almost-fixed-point-morphism}
Let $n$ be a positive integer.
There exists a constant $c(n)$, only depending on $n$, satisfying the following.
Let $G$ be a finite group.
Let $\phi\colon X\rightarrow Z$ be a $G$-equivariant Fano type morphism where $X$ has dimension $n$.
Let $z\in Z$ be a fixed point for the action on the base.
Then, there exists a normal abelian subgroup $A\triangleleft G$ of index at most $c(n)$, so that $A$ has a fixed point in $\phi^{-1}(z)$.
\end{introthm} 

\subsection{Rank and regularity}
\label{subsec:intro-rank-vs-reg}
In this subsection, we discuss the results regarding the rank and the regularity of a klt singularity, or analogously, of a Fano type variety.
A $n$-dimensional toric singularity naturally admits the action of $\zz_m^n$ for every $m$ by the multiplication of roots of unity on the big torus.
On the other hand, if we consider the dual complex of the torus invariant $1$-complement, we obtain a simplicial complex structure on a compact slice of the defining rational polyhedral cone. 
In particular, this dual complex has dimension $n-1$.
Similar examples can be achieved with singularities admitting a torus action, i.e., 
for a $n$-dimensional singularity $(X;x)$ admitting a $k$-dimensional torus action, we can find a boundary $B$ so that $\mathcal{D}(X,B;x)$ is at least $(k-1)$-dimensional.
This boundary can be chosen to be torus invariant.
The following result states that this phenomenon also happens if, instead of having a torus action, we have a large finite abelian action.
In this setting, the dimension of the acting torus is replaced by the rank of the finite abelian group. 

\begin{introthm}\label{introthm:rank-vs-reg-local}
Let $n$ and $r$ be two positive integers. 
There exists a constant $c(n)$, only depending on $n$, satisfying the following. 
Let $(X;x)$ be a $n$-dimensional klt singularity.
Assume that 
$\zz_m^r \leqslant {\rm Aut}(X;x)$ for some $m \geq c(n)$.
Then, there exists a $\zz_m^r$-invariant boundary $B$ so that
$(X,B;x)$ is log canonical and has regularity at least $r-1$, i.e., 
the dual complex $\mathcal{D}(X,\Delta+B;x)$ is a CW complex 
of dimension at least $r-1$.
\end{introthm}

The above theorem answers a conjecture of Shokurov regarding the rank and regularity of klt singularities.
The regularity of a klt singularity $(X;x)$ is defined as the maximum among the dimension of the dual complexes of log canonical pairs $(X,B;x)$.
Hence, the previous result can be stated as follows: if $(X;x)$ admits the action of $\zz_m^r$ for $m$ large enough, compared with a constant in terms of the dimension, then the regularity of $(X;x)$ is at least $r-1$.
In view of Theorem~\ref{introthn:almost-fixed-point-morphism}, the log canonical places provided by $\mathcal{D}(X,\Delta+B;x)$ in Theorem~\ref{introthm:rank-vs-reg-local}, will give us toroidal coordinates for the action of $\zz_m^r$ on the dlt modification.
The previous theorem follows from a projective statement for Fano type varieties.

\begin{introthm}\label{introthm:rank-vs-reg-global}
Let $n,N$ and $r$ be three positive integers. 
There exists a constant $c(n,N)$, only depending on $n$ and $N$, satisfying the following. 
Let $(X,B)$ be a $n$-dimensional log canonical log Calabi-Yau pair.
Assume that $X$ is of Fano type.
Assume that 
$\zz_m^r \leqslant {\rm Aut}(X;B)$ for some $m \geq c(n,N)$ and $N(K_X+B)\sim 0$.
Then, the dual complex of $(X,B)$ has dimension at least $r-1$.
\end{introthm}

Note that the condition on the rational connectedness of the variety is indeed necessary.
Otherwise, we can consider an elliptic curve, whose regularity is $-\infty$, and admits finite actions of arbitrarily large size.
Analogously in the case of singularities, the cone over an elliptic curve admits arbitrarily large cyclic actions, although its regularity is zero (see Example~\ref{ex:elliptic} for more details).
In Example~\ref{ex:no-control-B}, 
we show that the control of the index of $(X,B)$ is necessary in the previous theorem.

\subsection{Fundamental groups of klt singularities} 
\label{subsec:intro-p1}

It is known that the regional fundamental group of a klt singularity is finite (see, e.g.,~\cite{Xu14,Bra20}).
Hence, we can pass to the universal cover and work with finite actions on klt singularities.
Thus, as a corollary of our main theorem, 
we obtain a toroidalization for klt singularities that almost preserves the regional fundamental group.
This means that the Jordan property for the regional fundamental group~\cite[Theorem 1.2]{BFMS20}
can be achieved by a toroidalization. 

\begin{introcor}\label{introcor:p1-toroidal}
Let $n$ be a positive integer. 
There exists a constant $c(n)$, 
only depending on $n$,
satisfying the following.
Let $(X,\Delta;x)$ be a $n$-dimensional klt singularity.
Then, there exists:
\begin{enumerate}
    \item[(i)] A projective birational morphism $\phi \colon Y\rightarrow X$, 
    \item[(ii)] an effective divisor $B_Y$ on $Y$, and
    \item[(iii)] a klt toric singularity $(Y,B_Y;y)$,
\end{enumerate}
so that the induced homomorphism
\[
\phi_* \colon
\pi_1^{\rm reg}(Y,B_Y;y) \rightarrow 
\pi_1^{\rm reg}(X,\Delta;x)
\]
has cokernel of order at most $c(n)$.
In particular, $\pi_1^{\rm reg}(X,\Delta;x)$ admits a normal abelian subgroup of rank at most ${\rm codim}(Y,y)$ and of index at most $c(n)$.
\end{introcor} 

In~\cite[Theorem 2]{BFMS20}, we prove that the rank of the abelian part of the regional fundamental group of a $n$-dimensional klt singularity is bounded above by $n$.
In Theorem~\ref{introthm:group-toroidal}, the rank of the abelian part is bounded by the regularity of the singularity. 
Then, we obtain the following enhancement of the main theorem of~\cite{BFMS20}.

\begin{introthm}
\label{introthm:Jordan-reg}
Let $n$ and $r$ be a positive integer
with $r\leq n-1$.
Then, there exists a constant $c=c(n)$, only depending on $n$,
which satisfies the following property. Let $x\in (X,\Delta)$ be a $n$-dimensional klt singularity
of regularity $r$.
Then, there is an exact sequence
\begin{equation}
\nonumber 
\xymatrix{
1\ar[r] &
A\ar[r] &
\pi_1^{\rm reg}(X,\Delta,x) \ar[r]&
N\ar[r] &
1,
}
\end{equation}
where $A$ is a finite abelian group of rank at most $r+1$ and index at most $c(n)$.
Furthermore, since the regional fundamental group surjects onto 
the local fundamental group of the singularity, 
we obtain an exact sequence 
\begin{equation}
\nonumber
\xymatrix{
1\ar[r]&
A'\ar[r]&
\pi_1^{\rm loc}(X,x) \ar[r]&
N'\ar[r]&
1,
}
\end{equation}
where $A'$ is a finite abelian group of rank at most $r+1$ and index at most $c(n)$.
\end{introthm}

\subsection*{Acknowledgements} 
The author would like to thanks Stefano Filipazzi and Mirko Mauri for many useful comments. 

\section{Preliminaries} 

In this section, we recall the preliminaries 
that will be used in this article:
the singularities of the minimal model program,
Fano type varieties, $G$-equivariant complements,
the canonical bundle formula,
and dual complexes. 
Throughout this article, we work over an algebraically closed field $\kk$ of characteristic zero.
$G$ will denote a finite group unless otherwise stated.
The rank of a finite group is the least number of generators.

\subsection{Singularities} 

In this subsection, we recall the various classes of singularities that will be considered in this article.
We will prove some preliminary lemmas regarding special kinds of extractions and complements for these singularities.

\begin{definition}
{\em
A { \em contraction}
$f\colon X\rightarrow Z$ is a
projective morphism of
normal algebraic varieties so that $f_*\mathcal{O}_X=\mathcal{O}_Z$.
A {\em fibration} is a contraction
with positive dimensional general fiber.
}
\end{definition}

\begin{definition}
{\em 
Let $X$ be a normal quasi-projective variety with
the action of a finite group $G$.
A {\em $G$-prime divisor} is the $G$-orbit of a prime divisor on $X$.
We say that $X$ is {\em $G\qq$-factorial} if every $G$-invariant Weil divisor is $\qq$-Cartier. 
}
\end{definition}

\begin{definition}
{\em
Let $G$ be a finite group acting on a normal quasi-projective variety $X$.
We consider the set of all 
$G$-equivariant projective
birational morphisms 
$\pi\colon X_\pi\rightarrow X$, where $X_\pi$ is a normal variety.
This is a partially ordered set with
$\pi'\geq\pi$ when $\pi'$ factors through $\pi$.
A {\em $G$-invariant Weil b-divisor} on $X$ is an element of the inverse limit 
\[
\text{Div}_{G,b}(X):=
\varprojlim_\pi \mathrm{Div}_G(X_\pi).
\]
A {\em $\qq$-Weil b-divisor} is an element of $\text{Div}_{G,b}(X)_\qq$.
From now on, whenever we say $G$-invariant b-divisor, we mean a $G$-invariant $\qq$-Weil b-divisor.
Analogously, a $G$-equivariant b-divisor $M$ can be written as 
$\sum_{i\in I} b_iV_i$, where $I$ is a countable set and the $V_i$'s are $G$-invariant divisorial valautions of $\kk(X)$ satisfying the following condition. For each $G$-equivariant normal variety $X'$  (birational to $X$) only finitely many of the $V_i$'s are realized by divisors on $X'$.

Let $M$ be a $G$-equivariant b-divisor on $X$.
Let $Y$ be a $G$-equivariant normal variety birational to $X$.
The {\em trace} of $M$ on $Y$ is defined as
\[
M_Y:= \sum_{i\in I} b_iD_i,
\]
where the sum runs over all valuations $V_i$ so that 
$c_{Y}(V_i)=D_i$ with
$\codim_Y(D_i)=1$.
Then, the trace $M_Y$ is a $G$-invariant divisor on $Y$.

We say that a $G$-invariant b-divisor $M$ is {\em b-$\qq$-Cartier} if there exists a $G$-equivariant birational model $X'$ of $X$ and a $G$-equivariant $\qq$-Cartier divisor $D_{X'}$ on $X'$ satisfying the following condition.
For every $G$-equivariant projective morphism $\pi \colon X'' \rightarrow X'$, we have that
\[
M_{X''} = \pi^* D_{X'}.
\]
If $D_{X'}$ is nef, then we will say that $M$ is {\em b-nef}
or a {\em nef b-divisor}.
}
\end{definition}

\begin{proposition}\label{prop:G-inv-b-nef}
Let $G$ be a finite group.
Let $X$ be a $G$-invariant normal quasi-projective variety.
Let $M$ be a $G$-invariant b-nef divisor on $X$.
Let $\phi\colon X\rightarrow Y$ be the quotient by $G$.
Then, there exists a b-nef divisor $M_0$ on $Y$ satisfying the following. Let $X'$ be a $G$-equivariant birational model of $X$, then
\begin{equation}
\label{eq:pullback-b-nef}
M_{X'} = {\phi'}^*(M_{0,Y'}),
\end{equation}
where $\phi'\colon X'\rightarrow Y'$ is the quotient by $G$.
\end{proposition}

\begin{proof}
Let $X''$ be the model where the $G$-equivariant b-nef divisor $M$ descends.
Let $\phi''\colon X''\rightarrow Y''$ be the quotient by $G$.
Without loss of generality, we may assume that $\pi_X\colon X''\rightarrow X'$ is a $G$-equivariant projective birational morphism.
Then, we have an induced
projective birational morphism 
$\pi_Y \colon Y \rightarrow Y'$,
making the following diagram commutative: 
\[
\xymatrix{
X'' \ar[r]^-{\phi''}\ar[d]_-{\pi_X} & Y'' \ar[d]^-{\pi_Y} \\
X'\ar[r]^-{\phi'} & Y'.
}
\]
Let $D_X$ be a $\qq$-Cartier $G$-invariant nef divisor on $X''$ defining $M$.
There exists a $\qq$-Cartier $G$-invariant nef divisor $D_Y$ on $Y''$ so that ${\phi''}^*D_Y=D_X$.
We define $M_0$ to be the b-nef divisor on $Y$ defined by $D_Y$.
Then, we have that 
\[
M_{X'} =
{\pi_X}_*(D)=
{\pi_X}_*({\phi''}^*D_Y)=
{\phi'}^*({\pi_Y}_*D_Y)=
{\phi'}^*M_{0,Y'}.
\]
This concludes the statement.
\end{proof}

\begin{remark}
{\em 
With the notation of the proof of Proposition~\ref{prop:G-inv-b-nef}.
We refer to $X''$ as the model where the $G$-invariant b-nef divisor descends.
The model $Y''$ will be called the {\em quotient where the b-divisor descends}.
The previous proposition suggests that the theory of b-divisors should be developed for generically finite morphisms instead of projective birational maps.
In Example~\ref{ex:no-control-M}, we show the importance of controlling the Cartier index of the nef part in the quotient where it descends. 
}
\end{remark}

\begin{definition}
{\em 
Let $G$ be a finite group.
A $G$-equivariant generalized pair $(X,B+M)$ consists of the data:
\begin{enumerate}
    \item A normal quasi-projective variety $X$ with $G\leqslant {\rm Aut}(X)$,
    \item an effective $G$-invariant divisor $B$, and 
    \item $M$ is a nef $G$-invariant b-divisor,
\end{enumerate}
so that $K_X+B+M$ is $\qq$-Cartier.
In particular, the divisor
$K_X+B+M$ is $G$-invariant and $\qq$-Cartier.
By Proposition~\ref{prop:G-inv-b-nef}, any nef $G$-invariant b-divisor is the pullback from a nef b-divisor in the quotient of $X$.
We call $B$ the {\em boundary part} of the generalized pair and
$M$ the {\em nef part} of the generalized pair.

If we drop condition $(2)$, then we say that $(X,B+M)$ 
is a $G$-equivariant {\em generalized sub-pair}.
}
\end{definition}

\begin{definition}
{\em 
Let $(X,B+M)$ be a $G$-equivariant generalized pair. 
Let $\pi\colon Y\rightarrow X$ be a $G$-equivariant projective birational morphism.
As usual, we fix $K_Y$ so that 
$K_X=\pi_*K_Y$.
Then, there exists a uniquely determined $G$-invariant effective divisor $B_Y$ so that
\[
K_Y+B_Y+M_Y =\pi^*(K_X+B+M),
\]
where $M_Y$ is the trace of $M$ on $Y$.
Let $E$ be a $G$-prime divisor on $Y$.
We define the {\em generalized log discrepancy} of $(X,B+M)$ at $E$, to be
\[
a_E(X,B+M)= 1 - \text{coeff}_E(B_Y).
\]
We say that $(X,B+M)$ has:
\begin{enumerate} 
\item {\em Generalized log canonical singularities} if $a_E(X,B+M)\geq 0$ for every $G$-prime divisor $E$ over $X$.
We write generalized log canonical or glc for short.
\item {\em Generalized Kawamata log terminal singularities} if
$a_E(X,B+M)>0$ for every $G$-prime divisor $E$ over $X$.
We write generalized klt or gklt for short.
\item {\em Generalized canonical singularities} if $a_E(X,B+M)\geq 1$ for every $G$-prime divisor $E$ over $X$.
\end{enumerate} 
In the case that the b-nef divisor is trivial, i.e., $M_X=\mathcal{O}_X$ for every $X$, then the above concepts are just the singularities of $G$-pairs (see, e.g.,~\cite{Mor20a}).
In this case, we drop the word generalized from the definitions.
Furthermore, if $G$ is the trivial group, the above concepts are just the singularities of pairs (see, e.g.,~\cite{KM98}).
As usual, the previous definitions can be checked on a $G$-equivariant log resolution of the pair where $M$ descends. 
}
\end{definition}

\begin{definition}
{\em 
We say that a germ $(X;x)$ is of
{\em generalized klt type} (resp. {\em generalized lc type})
if there exists a boundary $B$ and a b-nef divisor $M$ on $X$ so
that $(X,B+M;x)$ has generalized klt (resp. generalized lc)
singularities.
In the case that $M=0$, then we say that
$X$ is of {\em klt type} (resp. {\em lc type}).
}
\end{definition}

\begin{notation}
{\em 
By abuse of notation, we may call the induced divisor $K_X+B+M$ a generalized pair (referring to $(X,B+M)$), when the boundary and nef parts are clear from the context.
Given a $G$-equivariant generalized pair $(X,B+M)$ and a $G$-equivariant projective birational morphism 
$\pi\colon Y \rightarrow Y$, we can write $\pi^*(K_X+B+M)=K_Y+B_Y+M_Y$.
We will call $(Y,B_Y+M_Y)$ the {\em log pullback} of $(X,B+M)$ to $Y$.
We may also say that $K_Y+B_Y+M_Y$ is the log pullback of $K_X+B+M$ to $Y$,
to stress the generalized pair structure of the pullback divisor.
}
\end{notation}

\begin{definition}{\em 
Let $(X,B+M)$ be a $G$-equivariant generalized pair.
The {\em generalized non-klt locus}
of the pair is the closed propert subet:
\[
\bigcup_E c_X(E) \subset X,
\]
that runs over all the $G$-prime divisors $E$ over $X$ for which 
$a_E(X,B+M)\leq 0$.
A {\em generalized log canonical place} of $(X,B+M)$ is a $G$-prime divisor $E$ for which $a_E(X,B+M)=0$.
A {\em generalized log canonical center} of $(X,B+M)$ is the center on $X$ of a generalized log canonical place.
For a $G$-equivariant generalized log canonical pair $(X,B+M)$,
the generalized non-klt locus equals the union of all the generalized log canonical centers. This set is denoted by ${\rm glcc}(X,B+M)$.
A generalized log canonical center is said to be {\em minimal}
if it is minimal with respect to the inclusion of glc centers.
A generalized {\em divisorial log canonical center} of $(X,B+M)$ is a generalized log canonical center which is a divisor on $X$, i.e., a prime component of $\lfloor B\rfloor$.

Observe that $G$ acts on the set of generalized log canonical centers of $(X,B+M)$.
We say that a generalized log canonical center $W$ of $(X,B+M)$ is {\em $G$-invariant} if $gW=W$ for every $g\in G$.
We say that a generalized log canonical place of $(X,B+M)$ is {\em $G$-invariant}
if the corresponding $G$-prime divisor is a prime divisor.
}
\end{definition}

\begin{proposition}\label{prop:quot-gen-pair}
Let $(X,B+M)$ be a $G$-equivariant generalized pair.
Let $\phi\colon X\rightarrow Y$ be the quotient by $G$.
Then, there exists a generalized pair $(Y,B_Y+M_Y)$ so that 
\[
K_X+B+M = \phi^*(K_Y+B_Y+M_Y).
\]
Furthermore, for every $G$-prime divisor $E$ over $X$, we have that
\[
a_E(X,B+M) = r a_{E/G}(Y,B_Y+M_Y),
\]
where $r$ is the ramification index of the field extension 
$\kk(X)\supseteq \kk(Y)$ at the $G$-prime divisor $E$.
In particular, $(X,B+M)$ is gklt (resp. glc) if and only if $(Y,B_Y+M_Y)$
is gklt (resp glc).
Furthermore, in the glc case, we have a natural bijection between $G$-equivariant log canonical places (resp. log canonical centers) of $(X,B+M)$ 
and the log canonical places
(resp. log canonical centers) of $(Y,B_Y+M_Y)$.
\end{proposition}

\begin{proof}
The existence of a log pair $(Y,B_Y)$ so that
$K_X+B=\phi^*(K_Y+B_Y)$ is proved in~\cite[Proposition 2.18]{Mor20a}.
On the other hand, by Proposition~\ref{prop:G-inv-b-nef}, we have that 
$M$ induces a b-nef divisor on $Y$.
By abuse of notation, we denote such b-divisor by $M$ as well.
By the equality~\eqref{eq:pullback-b-nef}, we conclude that 
\[
\phi^*(M_Y)=M_X.
\]
Hence, $(Y,B_Y+M_Y)$ is a generalized pair so that its log pullback to
$X$ equals $(X,B+M)$, i.e., 
\[
K_X+B+M =\phi^*(K_Y+B_Y+M_Y).
\]
Hence, it suffices to prove the statement about the generalized log discrepancies.
Let $X'\rightarrow X$ be the $G$-equivariant model where $M$ descends.
We may assume that $\pi\colon X'\rightarrow X$ is a log resolution of the generalized
pair $(X,B+M)$.
Let $\phi'\colon X'\rightarrow Y'$ be the corresponding quotient by $G$.
Thus, we have an induced birational contraction $\pi_Y \colon Y'\rightarrow Y$.
Then, we have a commutative diagram as follows:
\begin{equation}
\label{eq:comm-quot}
\xymatrix{
X' \ar[r]^-{\phi'}\ar[d]_-{\pi_X} & Y' \ar[d]^-{\pi_Y} \\
X\ar[r]^-{\phi} & Y.
}
\end{equation}
Let $E$ be a $G$-prime divisor on $X'$ and $E/G$ its quotient on $Y'$, which is a prime divisor of $Y'$.
We consider the divisors 
\[
K_{Y'}+B_{Y'}+M_{Y'} = {\pi_Y}^*(K_Y+B_Y+M_Y),
\]
and
\[
K_{X'}+B_{X'}+M_{X'} = \pi_X^*(K_X+B+M).
\]
Since ${\phi^*}(M_{Y'})=M_{X'}$, we conclude that 
\[
K_{X'}+B_{X'} = {\phi'}^*(K_{Y'}+B_{Y'}).
\]
By the Hurwitz formula, we deduce that
\[
1-{\rm coeff}_{E}(B_{X'})=r(1-{\rm coeff}_{E/G}(B_{Y'})).
\]
Thus, we have that $a_E(X,B+M) = r a_{E/G}(Y,B_Y+M_Y)$ as claimed.
The last statement of the theorem follows from the previous equality and the commutative diagram~\eqref{eq:comm-quot}.
\end{proof}

\begin{definition}
{\em 
Let $G$ be a finite group.
Let $(X,B+M)$ be a $G$-equivariant generalized pair.
The generalized pair $(Y,B_Y+M_Y)$ introduced in Proposition~\ref{prop:quot-gen-pair}
will be referred as the {\em quotient generalized pair} 
or the generalized pair induced on the quotient.
}
\end{definition}

\begin{definition}
{\em 
Let $G$ be a finite group.
Let $X\rightarrow Z$ be a $G$-equivariant contraction
and $G_Z$ be the quotient group acting on the base.
Let $z\in Z$ be a $G_Z$-fixed point.
Let $(X,B+M)$ be a $G$-equivariant generalized pair.
Let $X'\rightarrow Z'$ be the contraction induced by the quotient
and $(X',B'+M')$ be the quotient generalized pair.
A {\em $G$-equivariant $N$-complement} for $(X,B+M)$ over $Z$ is
the log pullback $(X,\Gamma+M)$ to $X$ of a generalized pair
$(X',\Gamma'+M')$ that satisfies:
\begin{enumerate}
    \item $\Gamma'\geq B'\geq 0$,
    \item $(X',\Gamma'+M')$ is generalized log canonical, and
    \item $N(K_{X'}+\Gamma'+M')\sim_{Z'} 0$, and 
    \item $NM'$ is Cartier in the birational model where it descends.
\end{enumerate}
By Proposition~\ref{prop:quot-gen-pair}, we conclude that:
\begin{enumerate}
    \item[(1')] $\Gamma \geq B \geq 0$;
    \item[(2')] $(X,\Gamma+M)$ is generalized log canonical, and
    \item[(3')] $N(K_X+\Gamma+M)\sim_Z 0$.
\end{enumerate}
If instead of condition $(3)$, we have that
\begin{enumerate} 
\item[(3'')] $K_{X'}+\Gamma'+M' \sim_{Z',\qq} 0,$
\end{enumerate}
then we say that $(X,\Gamma+M)$ is a {\em $G$-equivariant $\qq$-complement} of $(X,B+M)$ over $Z$.
Note that condition (4) is equivalent to say that $NM$ is Cartier in the quotient where it descends. 
If $G$ is the trivial group, then we say that $(X,B+M)$ is a $N$-complement or
$\qq$-complement, respectively.
With this notation, a $G$-equivariant $N$-complement is the pullback
of a $N$-complement of the quotient generalized pair.

In the case that $Z={\rm Spec}(\kk)$, then we say that
$(X,\Gamma+M)$ is a $G$-equivariant $N$-complement of $(X,B+M)$.
}
\end{definition}

The following lemma states the existence of bounded
$G$-equivariant generalized $N$-complements for klt singularities.

\begin{lemma}\label{lem:G-equiv-N-comp}
Let $n,p$ be positive integers and $\Lambda$ be a set of rational numbers with rational accumulation points.
There exists a constant $N:=N(n,p,\Lambda)$, only depending on $n,p$ and $\Lambda$, which satisfies the following.
Let $G$ be a finite group. Let $(X,B+M;x)$ be a $G$-equivariant klt singularities so that:
\begin{enumerate}
    \item The coefficients of $B$ belong to $\Lambda$, and 
    \item $pM$ is Cartier in the quotient where it descends.
\end{enumerate}
Then, there exists a $G$-equivariant $N$-complement for $(X,B+M;x)$.
\end{lemma}

\begin{proof}
Let $\rho_X\colon X\rightarrow X'$ be the quotient of $X$ by $G$.
By Proposition~\ref{prop:quot-gen-pair}, we can write
\[
\rho_X^*(K_{X'}+B'+M')=K_X+B+M,
\]
where the coefficients of $B'$ belongs to a set $\Lambda'\subset \qq$ so that 
$\overline{\Lambda}\subset \qq$ and $\Lambda'$ only depends on $\Lambda$.
Indeed, $\Lambda'$ is the set of hyperstandard coefficients over $\Lambda$,
i.e., 
\[
\Lambda' :=  \left\{
1-\frac{1}{m}+\frac{b}{m} \mid 
b\in \Lambda \text{ and } m\in \zz_{>0}
\right\}. 
\] 
By definition, we have that the b-divisor inducing the trace $M'$ has Cartier index at most $p$ where it descends.
Let $x'$ be the image of $x$ in $X'$.
By Proposition~\ref{prop:quot-gen-pair}, we know that the generalized pair $(X',B'+M';x'$) is generalized klt.
By~\cite[Theorem 1.2]{FM20}, there exists a $N$-complement for the generalized pair
$(X',B'+M')$ around $x'\in X'$.
Let $\Gamma' \geq B'\geq 0$ be the effective divisor that gives such complement.
Then, we have that
$(X',\Gamma' + M')$ is generalized log canonical at $x'$
and 
\[
N(K_{X'}+\Gamma'+M')\sim 0.
\]
Observe that $\Gamma'\geq B'$, hence the log pullback of $K_{X'}+\Gamma'+M'$ to $X$ is of the form $K_X+\Gamma+M$ where $\Gamma \geq B$.
Note that $(X,\Gamma+M)$ is generalized log canonical at $x$, is a $G$-equivariant generalized pair, and the linear equivalence 
\[
N(K_X+\Gamma+M)\sim 0,
\]
holds around $x\in X$.
We conclude that $(X,\Gamma+M)$ is a generalized $G$-equivariant $N$-complement. 
Note that $N$ only depends on $\Lambda'$ and $p$, so it only depends on $\Lambda$ and $p$.
\end{proof}

\begin{definition}
{\em
Let $(X,B+M)$ be a $G$-equivariant generalized pair.
We say that $(X,B+M)$ is {\em generalized divisorially log terminal} (generalized dlt or gdlt for short) if there exists a $G$-equivariant open set $U\subset X$ which satisfies the following conditions: 
\begin{enumerate}     
    \item the coefficients of $B$ are less than or equal to one,
    \item $U$ is smooth and $B|_U$ has simple normal crossing support, and
    \item all the generalized non-klt centers of $(X,B+M)$ map into $U$ and consists of strata of $\lfloor B \rfloor$.
\end{enumerate}
If we run a $G$-equivariant MMP for a generalized dlt pair, then it remains generalized dlt (see, e.g.,~\cite[Definition 1.23]{HM20}).
Note that the condition $(2)$ can be restated a saying that $(U,B|_U)$ is a log smooth pair.
We say $(X,B+M)$ is {\em generalized quotient divisorially log terminal} (generalized qdlt or gqdlt for short) if instead of $(2)$, we have that 
\begin{enumerate}
 \item[(2')] $(U,B|_U)$ has formally toric $\qq$-factorial singularities
so that $\lfloor B|_U\rfloor$ corresponds to the toric boundary. 
\end{enumerate} 
The above condition is equivalent to $(U,B|_U)$ being locally the quotient of a log smooth pair by a finite abelian group fixing the hyperplane coordinates.

Let $(X,B+M)$ be a $G$-equivariant generalized dlt pair.
If $(X,B+M)$ admits a unique divisorial generalized log canonical center,
then we say that $(X,B+M)$ is {\em generalized purely log terminal} (or generalized plt for short).
Observe that in a $G$-equivariant generalized plt pair, the divisorial glc center 
must be fixed by the group action.
}
\end{definition}

The following proposition states that a 
$G$-equivariant generalized log canonical pair
can be turned into a $G$-equivariant generalized dlt pair
by extracting certain $G$-invariant log canonical places.
The proof is similar to that of~\cite[Lemma 2.6]{LMT20}
by running a $G$-equivariant minimal model program.

\begin{proposition}\label{prop:dlt-mod}
Let $G$ be a finite group.
Let $(X,B+M)$ be a $G$-equivariant generalized log canonical pair.
Then, there exists a $G$-equivariant projective birational morphism
$\pi\colon Y\rightarrow X$ so that the following conditions are satisfied:
\begin{enumerate}
    \item The divisors extracted by $\pi$ are generalized log canonical places of $(X,B+M)$, and 
    \item the log pullback $(Y,B_Y+M_Y)$ of $(X,B+M)$ to $Y$ has gdlt singularities.
\end{enumerate}
\end{proposition}

In what follows, we proceed to recall some special blow-ups for generalized klt singularities.

\begin{lemma}\label{lem:plt-blow-up-gen}
Let $(X,B+M;x)$ be a generalized klt singularity.
Then, there exists a projective birational morphism
$\pi\colon Y\rightarrow X$ satisfying the following properties:
\begin{enumerate}
    \item $\pi$ extracts a unique prime divisor $E$ that maps to $x$ and $-E$ is anti-ample over $X$, and
    \item the generalized log pair $(Y,B_Y+M_Y+E)$ has generalized plt singularities, where $B_{Y}$ is the strict transform of $B$ and $M_Y$ is the trace of the corresponding b-divisor.
\end{enumerate}
\end{lemma}

\begin{proof}
Let $\Gamma\geq0$ be an effective $\qq$-Cartier divisor throught $x$ so that
$(X,B+\Gamma+M;x)$ is generalized log canonical.
We may assume that $x$ is a generalized log canonical center of $(X,B+\Gamma+M;x)$.
By tie-breaking, we may assume that $x$ is a unique generalized log canonical center of $(X,B+\Gamma+M;x)$ (see, e.g.,~\cite[Proposition 6.6]{Fil18}). Moreover, we may assume that $(X,B+\Gamma+M;x)$ has a unique generalized log canonical place over $x$.
Then, it suffices to take a generalized dlt modification of $(X,B+\Gamma+M;x)$
(see, e.g., Proposition~\ref{prop:dlt-mod}).
\end{proof}

\begin{lemma}\label{lem:gen-inv-plt}
Let $G$ be a finite group.
Let $(X,B+M;x)$ be a $G$-equivariant generalized klt singularity.
Then, there exists a projective $G$-equivariant birational morphism $\pi\colon Y\rightarrow X$ satisfying the following properties: 
\begin{enumerate}
    \item $\pi$ extracts a unique $G$-invariant prime divisor $E$ which maps to $x$, 
    \item  $-E$ is ample over $X$, and 
    \item $(Y,E+B_Y+M_Y)$ has generalized plt singularities, where $B_Y$ is the strict transform of $B$ to $Y$ and $M_Y$ is the trace of the b-divisor.
\end{enumerate}
\end{lemma}

\begin{proof}
Let $X'$ be the quotient of $X$ by $G$.
By Proposition~\ref{prop:quot-gen-pair},
we have a generalized klt pair $(X',B'+M';x')$ which pullbacks to $(X,B+M;x)$.
Let $\pi'\colon Y'\rightarrow X'$ be a generalized plt blow-up of $(X',B'+M';x')$ at $x$ (see Lemma~\ref{lem:plt-blow-up-gen}). 
Let $Y$ be the normal closure of $Y'$ in the function field of $X$.
Let
$\rho_X\colon X\rightarrow X'$ and
$\rho_Y \colon Y\rightarrow Y'$ be the corresponding quotient morphisms.
Then, we have a $G$-equivariant projective birational morphism
$\pi\colon Y \rightarrow X$.
By construction, we have the following properties:
\begin{enumerate}
\item $\pi'$ extracts a unique prime divisor $E'$ that maps to $x'$ and is anti-ample over $X'$, and 
\item 
the generalized log pair $(Y',B_{Y'}+M_{Y'}+E')$ has generalized plt singularities,
where $B_{Y'}$ is the strict transform of $B'$ and $M_{Y'}$ is the trace of the corresponding b-divisor.
\end{enumerate}
Observe that we have the following $\qq$-linear equivalence
\[
\rho_Y^*(K_{Y'}+B_{Y'}+M_{Y'}+E')\sim_\qq
K_Y+B_Y+M_Y+E,
\]
where $E=\rho_Y^*(E')$.
Hence, we conclude that
$(Y,E+B_Y+M_Y)$ is plt and anti-ample over $X$.
By the connectedness of log canonical centers for generalized pairs (see, e.g.~\cite[Lemma 2.14]{Bir19}), we conclude that $E$ is a prime divisor.
We deduce that 
$E$ is a $G$-invariant prime divisor, which maps to $x$ and is anti-ample over $X$.
This finishes the proof.
\end{proof}

\begin{lemma}\label{lem:existence-g-inv-lcp}
Let $(X,\Delta)$ be a $G$-invariant log canonical pair.
Assume that $(X,(1-\epsilon)\Delta)$ is klt.
Let $W\subset X$ be a $G$-invariant log canonical center of $(X,\Delta)$.
Then, $(X,\Delta)$ admits a $G$-invariant log canonical place over $W$.
\end{lemma}

\begin{proof}
Let $(X',\Delta')$ be the log pair obtained by quotienting by $G$.
Then, we have that $(X',(1-\epsilon')\Delta')$ is a klt pair for $\epsilon'$ small enough.
Let $W'\subset X'$ be the image of $W$ on $X'$.
Proceeding as in the proof of~\cite[Lemma 1]{Xu14}, we can find a projective birational morphism
$Y'\rightarrow X'$ which extracts a unique prime divisor $E$ that satisfies:
\begin{enumerate}
    \item $E'$ maps to $W'$,
    \item $-E'$ is ample over $X'$, and 
    \item $(Y',\Delta_{Y'}+E')$ has plt singularities.
\end{enumerate} 
Here, $\Delta_{Y'}$ is the strict transform of $\Delta'$ in $Y'$.
Then the rest of the proof proceeds as the one of Lemma~\ref{lem:gen-inv-plt}.
By taking the normal closure of $Y'$ on the function field of $X$, we obtain a $G$-equivariant projective birational morphism $Y\rightarrow X$.
Let $E$ be the divisorial exceptional locus of this projective birational map.
Then, $E$ maps onto $W$.
By construction,
the pair $(Y,E+\Delta_Y)$ is $G$-equivariant, plt and,
the divisor 
$K_Y+E+\Delta_Y$ is anti-ample over $X$.
By the connectedness for generalized log canonical centers~\cite[Lemma 2.14]{Bir19}, we conclude that $E$ is prime.
Hence, $E$ is a $G$-invariant log canonical place over $W$.
\end{proof}

\subsection{Fano type varieties}

In this subsection, we will recall the main class of projective varieties that we consider in this article.
We recall the theory of complements on Fano varieties and prove some preliminary lemmas.

\begin{definition}
{\em 
A {\em Fano type morphism}
is a projective morphism $X\rightarrow Z$ so that
there exists a boundary $\Delta$ in $X$ satisfying:
\begin{enumerate}
    \item $(X,\Delta)$ has klt singularities, and 
    \item $-(K_X+\Delta)$ is big and nef over $Z$.
\end{enumerate}
If $Z={\rm Spec}(\kk)$, then we say that $X$ is a {\em Fano type variety}
or a variety of Fano type.
It is well-known that Fano type morphisms are relative Mori dream spaces~\cite[Corollary 1.3.1]{BCHM10}.
In particular, any minimal model program for a divisor on $X$ over $Z$
must terminate with a good minimal model or a Mori fiber space over $Z$.
}
\end{definition}

\begin{definition}
{\em 
Let $X\rightarrow Z$ be a contraction.
Let $(X,B+M)$ be a generalized pair structure on $X$.
We say that $(X,B+M)$ is generalized log Calabi-Yau over $Z$ if the following conditions are satisfied:
\begin{enumerate}
    \item $(X,B+M)$ has generalized log canonical singularities, and 
    \item $K_X+B+M\sim_{Z,\qq} 0$.
\end{enumerate}
In the case that $Z={\rm Spec}(\kk)$, then we say that $(X,B+M)$ is a generalized log Calabi-Yau pair.
}
\end{definition}

\begin{notation}
{\em 
Let $(X,B+M)$ be a generalized log Calabi-Yau pair.
Let $\phi\colon X\rightarrow Y$ be a projective birational contraction.
We can induce a generalized log Calabi-Yau pair on $Y$ as follows:
let $B_Y$ be the push-forward of $B_X$ on $Y$ and $M_Y$ be the trace
of the b-divisor $M$ on $Y$.
Then, the generalized pair $(Y,B_Y+M_Y)$ is generalized log Calabi-Yau.
More generally, given the following data:
\begin{enumerate}
    \item a birational map $X\dashrightarrow Y$, and 
    \item a generalized log Calabi-Yau pair $(X,B+M)$,
\end{enumerate}
we can pullback $(X,B+M)$ to a birational
model which admits a projective contraction to $Y$ and then push-forward to $Y$.
We call the obtained generalized sub-pair $(Y,B_Y+M_Y)$ the push-forward of $(X,B+M)$ to $Y$.
We say that $(X,B+M)$ and $(Y,B_Y+M_Y)$ are {\em log crepant} generalized sub-pairs.
Observe that in general $B_Y$ may have negative coefficients, but whenever $X\dashrightarrow Y$ is a birational contraction, thee divisor $B_Y$ 
will be effective.
}
\end{notation}

The following lemma will often be used throughout the article.
It allows us to prove that certain extractions over Fano type
varieties are still Fano type varieties.

\begin{lemma}\label{lem:extraction-FT}
Let $X$ be a Fano type variety.
Let $(X,B+M)$ be a generalized log Calabi-Yau pair.
Let $\pi\colon Y\rightarrow X$ be a projective birational morphism
which only extracts divisors with generalized log discrepancy in $[0,1)$ with respect to $(X,B+M)$. Then $Y$ is a Fano type variety.
\end{lemma}

\begin{proof}
Since $X$ is of Fano type,
then it is a Mori dream space~\cite[Corollary 1.3.1]{BCHM10}.
Passing to a small $\qq$-factorialization, we may assume that $X$ has $\qq$-factorial singularities.
The diminished base locus of $M$ has codimension at least two,
being it the push-forward of a nef divisor on a higher model.
We run a $M$-MMP which terminates with a good minimal model.
Let $X\dashrightarrow X_0$ be such a minimal model program.
Since the diminished base locus of $M$ has codimension at least two, then 
the birational map $X\dashrightarrow X_0$ is small.
Let $M_0$ be the trace of $M$ on $X_0$.
We have that $M_0$ is a semiample divisor.
Let $B_0$ be the push-forward of $B$ to $X_0$.
Since $(X_0,B_0+M_0)$ is generalized log canonical, then $(X_0,B_0)$ is log canonical.
We can find $M_0\sim_\qq \Delta_0 \geq 0$ so that $(X_0,B_0+\Delta_0)$ is log canonical.
Observe that $(X_0,B_0+\Delta_0)$ is log Calabi-Yau.
Let $\Delta$ be the push-forward of $\Delta_0$ to $X_0$.
Then, the log pair $(X,B+\Delta)$ is log Calabi-Yau.

Since $X$ is of Fano type,
we can find a boundary $\Gamma\geq 0$ so that
$(X,\Gamma)$ has klt singularities
and $-(K_X+\Gamma)$ is big and nef.
For $\epsilon \in (0,1)$, we define the boundary
\[
\Omega_\epsilon = (1-\epsilon)\Gamma + \epsilon(B+\Delta).
\]
For $\epsilon>0$ small enough, 
all the divisors extracted on $Y$ 
have log discrepancy 
in $(0,1)$ with respect to $(X,\Omega_\epsilon)$.
Observe that $(X,\Omega_\epsilon)$ has klt singularities.
Furthermore, we have that
\[
-(K_X+\Omega_\epsilon) =
-(1-\epsilon)(K_X+\Gamma) -
\epsilon(K_X+B+\Delta) \sim_\qq 
-(1-\epsilon)(K_X+\Gamma)
\]
is big and nef.
Let $(Y,\Omega_{Y,\epsilon})$ be the log pull-back of $(X,\Omega_\epsilon)$ to $Y$.
For $\epsilon>0$ small enough the following conditions are satisfied:
\begin{enumerate}
    \item $(Y,\Omega_{Y,\epsilon})$ is a log pair,
    \item $(Y,\Omega_{Y,\epsilon})$ has klt singularities, and 
    \item $-(K_Y+\Omega_{Y,\epsilon})$ is big and nef.
\end{enumerate}
Hence, $Y$ is a Fano type variety.
\end{proof}

Now, we turn to prove that $G$-equivariant Fano type varieties which admit a $G$-equivariant $\qq$-complements also admits $G$-equivariant $N$-complements.
Here, $N$ only depends on the dimension and numerical invariants of the boundary and nef parts.

\begin{lemma}\label{lem:Q-com-N-com}
Let $n$ and $p$ be positive integers and let
$\Lambda$ be a set of rational numbers with rational accumulation points.
There exists a constant $N:=N(n,p,\Lambda)$, only depending on $n,p$ and $\Lambda$, satisfying the following.
Let $G$ be a finite group.
Let $X$ be a $n$-dimensional Fano type variety.
Let $(X,B+M)$ be a $G$-equivariant generalized log canonical pair satisfying the following conditions:
\begin{enumerate}
    \item The coefficients of $B$ belong to $\Lambda$, and 
    \item $pM$ is Cartier in the quotient where it descends.
\end{enumerate}
Assume that $(X,B+M)$ admits a $\qq$-complement.
Then, $(X,B+M)$ admits a $G$-equivariant $N$-complement.
\end{lemma}

\begin{proof}
Let $(X',B'+M')$ be the quotient generalized pair.
Then, as in the proof of Lemma~\ref{lem:G-equiv-N-comp}, the coefficients of $B'$ belong to a set $\Lambda'\subset \qq$ with $\overline{\Lambda'}\subset \qq$, which only depends on $\Lambda$.
Furthermore, the b-divisor $M'$ has Cartier index a divisor of $p$ in the model where it descends.
We claim that the generalized pair $(X',B'+M')$ is $\qq$-complemented by some boundary $\Gamma' \geq 0$.
Inded, by the proof of~\cite[Proposition 2.17]{Mor20a}, the $G$-equivariant generalized pair $(X,B+M)$ admits a $G$-equivariant $\qq$-complement, so the claim follows as in~\cite[Proposition 2.18]{Mor20a}.
Hence, we have that the generalized pair
$(X',B'+\Gamma'+M')$ has generalized log canonical singularities.
By~\cite[Proposition 2.17]{Mor20a}, we know that $X'$ is a Fano type variety, in particular, it is a Mori dream space.
We run a $-(K_{X'}+B'+M')$-minimal model program.
Since $-(K_{X'}+B'+M')$ is pseudo-effective, this minimal model program $X'\dashrightarrow X''$ terminates with a good minimal model
for $-(K_{X'}+B'+M')$.
By~\cite[Proposition 6.1.(3)]{Bir19}, in order to find a $N$-complement for $(X',B'+M')$ it suffices to produce a $N$-complement for $(X'',B''+M'')$.
The existence of a $N$-complement for $(X'',B''+M'')$ is proved in~\cite[Theorem 1.2]{FM20}.
Hence, we conclude that $(X',B'+M')$ admits a $N$-complement.
By pulling-back to $X$, we obtain a generalized $G$-equivariant $N$-complement for $(X,B+M)$.
\end{proof}

The following lemma will be used to find 
horizontal log canonical centers in Fano type fibrations, which are fixed for a finite action.

\begin{lemma}\label{lem:fix-lcc}
Let $n$ be a positive integer.
There exists a constant $c(n)$, only depending on $n$, satisfying the following.
Let $F$ be a $n$-dimensional canonical Fano variety.
Let $(F,B_F)$ be a log Calabi-Yau pair.
Then, for each $i\in \{0,\dots, \dim F-1\}$ the number of $i$-dimensional log canonical places of $(F,B_F)$ is bounded above by $c(n)$.
\end{lemma}

\begin{proof}
Let $(F',B')$ be a dlt modification of $(F,B_F)$.
The variety $F'$ is a $n$-dimensional Fano type variety.
Let $E':=\lfloor B'\rfloor$.
Then, the pair $(F',E')$ is log canonical and $\qq$-complemented.
By Lemma~\ref{lem:G-equiv-N-comp}, we conclude that $(F',E')$ is $N$-complemented, where $N$ only depends on the dimension $n$.
Let $\Gamma'$ be such $N$-complement.
Then, we have that $N(K_{F'}+E'+\Gamma')\sim 0$
and $(F',E'+\Gamma')$ has log canonical singularities.
Let $(F,\Gamma)$ be the pair obtained by pushing-forward $(F,E'+\Gamma')$ to $F$.
Then, since the coefficients of $\Gamma$ belong to a finite set which only depends on $n$, we conclude that the log pairs $(F,\Gamma)$ belong to a log bounded family (see, e.g.,~\cite[Theorem 3.3]{FM20}).
Thus, there exists a constant $c(n)$, only depending on $n$, so that for each $i\in \{0,\dots,\dim F-1\}$, the number of $i$-dimensional log canonical centers of $(F,\Gamma)$ are bounded above by $c(n)$.
Since every log canonical center of $(F,B_F)$ is a log canonical center of $(F,\Gamma)$, the result follows.
\end{proof}

\subsection{Dual complexes} 
In this subsection, we recall some basics about dual complex associated to 
gdlt modifications of
generalized log canonical pairs.
In~\cite{FS20,Bir20}, the authors prove that such dual complexes are connected provided
that they are positive dimensional and the underlying glc pair is log Calabi-Yau.
We refer to~\cite{KX16,dFKX17} for the classics about dual complexes of singularities
and log Calabi-Yau pairs.

\begin{definition}
\label{def:dual-complex}
{\em 
Let $E$ be a projective scheme
which is pure dimensional with irreducible components $E_1,\dots, E_r$.
Assume that each $E_i$ is a normal variety and 
for every $J\subset \{1,\dots, r\}$,
if the intersection $\cap_{j\in J} E_j$ is non-empty, 
then every connected component of this intersection is irreducible
of codimension $|J|-1$ in $E$.
We define the {\em dual complex} $\mathcal{D}(E)$ of $E$ as follows.
The vertices $v_1,\dots, v_r$ of $\mathcal{D}(E)$ are in bijection with the irreducible components
$E_1,\dots, E_r$.
To each irreducible component $W$ of $\cap_{j\in J}E_j$,
we associate a cell $v_W$ of dimension $|J|-1$.
Observe that for each $i\in J$, the variety $W$ is contained in a unique
irreducible component of $\cap_{j\in J\setminus \{i\}}E_i$;
this determines the gluing of the cell $v_W$.
$\mathcal{D}(E)$ is a CW complex.
However, in general, this CW complex may be neither regular nor simplicial (see, e.g.,~\cite[\S 2]{dFKX17}).

Let $G$ be a finite group
and $(X,B+M)$ be a $G$-equivariant generalized log canonical pair.
Let $(Y,B_Y+M_Y)$ be a $G$-equviariant generalized dlt modification of $(X,B+M)$.
We define the {\em dual complex} of $(X,B+M)$ to be
\[
\mathcal{D}(X,B+M) := 
\mathcal{D}(\lfloor B_Y\rfloor),
\]
where the dual complex of $\lfloor B_Y\rfloor$ is defined as in the previous paragraph.
By~\cite[Lemma 2.32]{FS20}, we know that up to simple homotopy equivalence
$\mathcal{D}(X,B+M)$ is independent of the generalized dlt modification.
Observe that $G$ acts on $\mathcal{D}(X,B+M)$.
Furthermore, if we choose a simplicial representative $\mathcal{D}(Y,B_Y+M_Y)$
of the homotopy class of $\mathcal{D}(X,B+M)$,
then $G$ acts on $\mathcal{D}(Y,B_Y+M_Y)$ preserving the simplicial complex structure.
In Lemma~\ref{lem:simp-complex}, we show that we can always find such a simplicial representative.
Furthermore, we can assume that all maximal dimensional simplices of
$\mathcal{D}(Y,B_Y+M_Y)$ have the same dimension $r$ (see Lemma~\ref{lem:simp-complex}).
This non-negative integer $r$ is defined to be the regularity of $\mathcal{D}(X,B+M)$,
denoted by ${\rm reg}(X,B+M)$.
In the case that $\mathcal{D}(X,B+M)$ is empty, we set ${\rm reg}(X,B+M):=-\infty$.
From the construction, it follows that $r\in \{-\infty,0,\dots, \dim X -1\}$.

Given a generalized klt singularity $(X,B+M;x)$, 
we define its regularity as the maximum among the regularities
${\rm reg}(X,B+\Gamma+M;x)$ so that $(X,B+\Gamma+M;x)$ has
generalized log canonical singularities at $x\in X$, i.e., 
\[
{\rm reg}(X,B+M;x) := \max
\left\{ 
{\rm reg}(X,B+\Gamma+M;x) \mid 
(X,B+\Gamma+M;x) \text{ is glc }
\right\}. 
\]
Note that the in the case of singularities, we have that 
${\rm reg}(X,B+M;x)\in \{0,\dots,\dim X-1\}$.
Indeed, we can always produce a log canonical center through the point $x$ by adding a multiple of an effective divisor through $x$.

If $(X,B+M;x)$ is a $n$-dimensional generalized klt singularity of regularity $r$, we may say that
$(X,B+M;x)$ is a {\em $n$-dimensional $r$-regular gklt singularity}.
Note that we always have the inequality $r\leq n-1$
and the equality holds, for instance, when the germ is toric.
}
\end{definition}

The following lemma is similar to~\cite[Theorem 3.3]{Lo19}.
The main difference is that we are working in the relative case.

\begin{lemma}\label{lem:dlt-antiample pair}
Let $(X,B)$ be a $n$-dimensional dlt pair and $X\rightarrow Z$ be a projective contraction so that $-(K_X+B)$ is anti-ample over $Z$.
Assume that the prime components $E_1,\dots, E_r \subset \lfloor B\rfloor$ maps onto a subvariety $W\subset Z$.
Then, we have that $r\leq n$.
Furthermore, $\mathcal{D}(X,E_1+\dots+E_r)$ is a $r$-dimensional simplex and $E_1\cap \dots \cap E_r$ dominates $W$.
\end{lemma}

\begin{proof}
We write $B=E_1+\dots+E_r+B'$.
Given $m +1 \leq r+1$ vertices in $\mathcal{D}(X,E_1+\dots+E_r)$, we prove that there exists a unique $m$-simplex with these vertices.
Without loss of generality, we may assume that these vertices correspond to $E_1,\dots,E_{m+1}$.
First, we check that the statement holds for $m=1$.
Indeed, let $E_1$ and $E_2$ be two components of $\lfloor B\rfloor$.
For $\epsilon>0$ small enough, the log pair
\[
\left( 
X, E_1+E_2 + (1-\epsilon)\left(
\sum_{i=3}^r E_i + B'
\right)
\right) 
\]
is dlt and anti-ample over the base.
By~\cite[Theorem 6.50]{KSC04}, we conclude that $E_1\cap E_2\neq \emptyset$
and it dominates $W$.
We claim that $E_1\cap E_2$ is irreducible.
Indeed, for $\epsilon>0$ small enough, the log pair
\[
\left( 
E_1, E_2|_{E_1} + 
(1-\epsilon)\left(
\sum_{i=3}^r E_i + B'
\right)|_{E_1}
\right) 
\]
is dlt and anti-ample over the base.
We conclude that $E_1\cap E_2$ is connected.
The irreducibility of $E_1\cap E_2$ follows from the dlt condition of
$(X,B)$.

Now, we can proceed by induction on $m$.
For any $j\in \{1,\dots, m+1\}$, we define
$W_j:=E_1\cap\dots\cap E_j$.
We also define $W_0:=X$.
By induction on $m$, we have that $W_{m-1}\cap E_m=W_m$ and $W_{m-1}\cap E_{m+1}$ are non-empty
and dominate $W$.
Note that, for $\epsilon>0$ small enough, the log pair
\[
\left( W_{m-1}, E_m|_{W_m-1}+E_{m+1}|_{W_{m-1}} +
(1-\epsilon)\left( \sum_{i=1}^{m-1}E_i + B'\right)|_{W_{m-1}}
\right)
\]
is dlt and anti-ample over the base.
We conclude that $W_{m+1}$ is non-empty and dominates $W$.
Moreover, for $\epsilon>0$ small enough,
the log pair
\[
\left(
W_m , E_{m+1}|_{W_m} + (1-\epsilon)
\left(
\sum_{i=1}^m E_i + B'
\right)|_{W_m}
\right) 
\] 
is anti-ample and dlt over the base.
We conclude that $W_{m+1}$ is connected.
Hence, it is irreducible by the dlt condition of $(X,B)$.
Thus, $W_{m+1}$ is the irreducible intersection of $E_1,\dots,E_{m+1}$ and it dominates the base.
By the dlt condition, we conclude that $m+1\leq n+1$.
\end{proof}

\begin{lemma}\label{lem:triv-act-in-out}
Let $G$ be a finite group.
Let $(X,B+M)$ be a $G$-equivariant generalized dlt pair.
Let $E\subset X$ be a divisorial generalized log canonical center of $(X,B+M)$ which is fixed by $G$.
Let $G_E$ be the quotient group of $G$ which acts on $E$.
Let $(E,B_E+M_E)$ be the $G_E$-equivariant generalized dlt pair induced by equivariant adjunction (see Lemma~\ref{lem:G-equiv-adj}).
Assume that $G_E$ fixes all the generalized log canonical centers of $(E,B_E+M_E)$.
Then $G$ fixes all the generalized log canonical centers
of $(X,B+M)$ intersecting $E$.

Furthermore, if $G_E$ acts as the identity on a minimal generalized log canonical center of $(E,B_E+M_E)$, then $G$ acts as the identity on a minimal generalized log canonical center of $(X,B+M)$.
\end{lemma}

\begin{proof}
Let $W$ be a generalized log canonical center of $(X,B+M)$ of codimension $r$ which intersects $E$ non-trivially.
If $W\subset E$, then we are done.
Hence, we may assume that $W_E:=W\cap E$ is a generalized log canonical center of $(E,B+M)$ of codimension $r$ in $E$.
Note that $W_E$ is a generalized log canonical center of $(X,B+M)$ of codimension $r+1$.
Hence, $W_E$ is contained in exactly $r+1$ generalized log canonical centers of $(X,B+M)$ of codimension $r$.
$W_E$ is a generalized log canonical center of $(E,B_E+M_E)$ of codimension $r$ in $E$. Then $W_E$ is contained in exactly $r$ generalized log canonical centers of $(E,B_E+M_E)$ of codimension $r-1$ in $E$. We denote these centers by
$W'_1,\dots, W'_{r}\subset E$.
Note that $W'_1,\dots, W'_{r}$ are $G$-invariant generalized log canonical centers of $(X,B+M)$ containing $W_E$ which have codimension $r$ in $X$.
We conclude that $W$ must be $G$-fixed, otherwise, we would have at least $r+2$ generalized log canonical centers of $(X,B+M)$ of codimension $r$ containing $W_E$. The last assertion of the lemma is immediate. 
\end{proof}

\begin{lemma}\label{lem:reg-under-quot}
Let $G$ be a finite group.
Let $(X,B+M)$ be a $G$-equivariant generalized log canonical pair.
Let $(X',B'+M')$ be the generalized log canonical pair induced in the quotient $\rho\colon X \rightarrow X'$.
Then, we have an equality
\[
{\rm reg}(X,B+M)={\rm reg}(X',B'+M').
\]
\end{lemma}

\begin{proof}
Let $(Y',B_{Y'}+M_{Y'})$ be a generalized dlt modification of the pair $(X',B'+M')$.
Let $Y$ be the normal closure of $Y'$ in the function field of $X$.
Then, we have a commutative diagram as follows
\[
\xymatrix{
G\acts (X',B'+M')\ar[r]^-{\rho_Y}\ar[d]_-{\phi} & (Y',B_{Y'}+M_{Y'}) \ar[d]^-{\phi'}  \\
G\acts (X,B+M) \ar[r]^-{\rho_X} & (X',B'+M'),
}
\]
where the horizontal maps are Galois and the vertical maps are birational projective morphisms.
We claim that $(X',B'+M')$ is a generalized pair with generalized qdlt singularities.
Indeed, at the generic point of each generalized log canonical center of $(Y',B_{Y'}+M_{Y'})$, the quotient $\rho_Y$ only ramifies along $\lfloor B_{Y'}\rfloor$ which is log smooth.
Thus, we conclude that at the generic point of each generalized log canonical center of $(Y',B_{Y'}+M_{Y'})$, the quotient $\rho_Y$ is a toric quotient.
Hence, the pair $(X',B'+M')$ is generalized qdlt
and the dual complexes $\mathcal{D}(X',B'+M')$ and
$\mathcal{D}(Y',B_{Y'}+M_{Y'})$ have the same dimension. 
\end{proof}

\subsection{Adjunction and canonical bundle formula} 

In this subsection, we recall the generalized pairs obtained by adjunction to divisors
and the canonical bundle formula.
We prove some properties on the boundary and nef part when the canonical bundle formula is performed for Fano type fibrations.

\begin{lemma}\label{lem:G-equiv-adj}
Let $G$ be a finite group.
Let $(X,B+M)$ be a $G$-equivariant generalized log canonical pair.
Let $E\subset \lfloor B\rfloor$ be a prime component which is normal as an algebraic variety.
Assume that $E$ is fixed by the group $G$.
Let $G_E$ be the quotient group of $G$ acting on $E$.
Then, we can write
\[
K_E+B_E+M_E \sim_\qq (K_X+B+M)|_E,
\]
where $(E,B_E+M_E)$ is a $G_E$-equivariant generalized log canonical pair.
Furthermore, if $(X,B+M)$ is generalized dlt, then $(E,B_E+M_E)$ is generalized dlt.
\end{lemma}

\begin{proof}
Let $X\rightarrow X'$ be the quotient by $G$.
Observe that the image $E'$ of $E$ on $X'$ is a normal variety.
Hence, we can define $(E',B_{E'}+M_{E'})$ by generalized adjunction to $E'$ (see, e.g.~\cite[\S 4]{BZ16} or~\cite[Lemma 1.32]{HM20}).
Let $(E,B_E+M_E)$ be the generalized pair obtained by pullback of $(E',B_{E'}+M_{E'})$ to $E$.
By Hurwitz formula, we see that $B_E$ is indeed an effective divisor.
Then, we conclude that $(E,B_E+M_E)$ is a generalized log canonical pair (see Proposition~\ref{prop:quot-gen-pair}).
If $(X,B+M)$ is generalized dlt, then the prime components of $\lfloor B\rfloor$ have simple normal crossing, so the same holds for $\lfloor B_E\rfloor$. This implies that $(E,B_E+M_E)$ is generalized dlt as well.
\end{proof}

\begin{definition}
\label{def:cbf}
{\em  
Let $G$ be a finite group.
Let $(X,B+M)$ be a $G$-equivariant generalized log canonical pair.
Let $f\colon X\rightarrow Z$ be a $G$-equivariant contraction so that
$(X,B+M)$ is generalized log Calabi-Yau over $Z$.
Let $G_Z$ be the quotient group acting on the base.
We let $\pi_X\colon X\rightarrow X'$ be the quotient by $G$ and
$\pi_Z\colon Z\rightarrow Z'$ be the quotient by $G_Z$.
We have an induced contraction $f'\colon X'\rightarrow Z'$ on the quotient
so that there is a commutative diagram:
\begin{equation}\label{eq:egcbf}
\xymatrix{
G\acts X\ar[r]^-{\pi_X}\ar[d]_-{f} & X' \ar[d]^-{f'}  \\
G_Z \acts Z \ar[r]^-{\pi_Z} & Z'.
}
\end{equation} 
By Proposition~\ref{prop:quot-gen-pair}, we have an induced quotient generalized pair
$(X',B'+M')$
on $X'$. 
Note that $(X',B'+M')$ is generalized log canonical and log Calabi-Yau over $Z'$.
Applying the generalized canonical bundle formula~\cite[Theorem 1.4]{Fil18}
for $(X',B'+M')$ over $Z'$, we obtain a glc pair $(Z',B_{Z'}+M_{Z'})$ on $Z'$
so that
\[
K_{X'}+B'+M'\sim_\qq {f'}^*(K_{Z'}+B_{Z'}+M_{Z'}).
\]
We define $(Z,B_Z+M_Z)$ to be the log pull-back of $(Z',B_{Z'}+M_{Z'})$.
By Proposition~\ref{prop:quot-gen-pair}, we conclude that $(Z,B_Z+M_Z)$ is a
$G_Z$-invariant generalized log canonical pair.
By the commutativity of diagram~\eqref{eq:egcbf}, we have that
\[
K_X+B+M\sim_\qq 
f^*(K_Z+B_Z+M_Z).
\] 
The pair $(Z,B_Z+M_Z)$ will be called the generalized pair induced by the 
{\em $G$-equivariant generalized canonical bundle formula}.
If the group is clear from the context, we may simply refer to it as the generalized canonical bundle formula.
}
\end{definition}

In what follows, we will prove that under the assumption that
$X\rightarrow Z$ is a Fano type morphism, we can control 
numerical invariants of the boundary and nef part of the generalized pair
induced by the equivariant canonical bundle formula.

\begin{lemma}\label{lem:G-equiv-cbf}
Let $n$ and $N$ be positive integers.
There exists a constant $M:=M(n,N)$, only depending on $n$ and $N$,
satisfying the following.
Let $G$ be a finite group.
Let $f\colon X\rightarrow Z$ be a $G$-equivariant Fano type morphism where $X$ is a $n$-dimensional variety.
Let $(X,B)$ be a $G$-equivariant log canonical $N$-complement.
Let $G_Z$ be the quotient group acting on the base.
Then, we can write
\[
K_X+B \sim_\qq f^*(K_Z+B_Z+M_Z),
\]
where $(Z,B_Z+M_Z)$ is a $G_Z$-equivariant generalized log canonical $M$-complement.
Here, $(Z,B_Z+M_Z)$ is defined by the canonical bundle formula (see Definition~\ref{def:cbf}).
\end{lemma} 

\begin{proof}
Let $\pi_X\colon X\rightarrow X'$ be the quotient by $G$
and let $\pi_Z\colon Z\rightarrow Z'$ be the quotient by $G_Z$.
Then, we have an induced projective contraction $f\colon X'\rightarrow Z'$ 
which is a Fano type morphism.
Let $(X',B')$ be the generalized $N$-complement 
that pullbacks to $(X,B)$.
Note that $K_{X'}+B'$ is $\qq$-trivial over $Z'$.
Thus, we have a commutative diagram as follows: 
\[
\xymatrix{
X\ar[r]^-{\pi_X} \ar[d]_-{f} & X'\ar[d]^-{f'} \\
Z\ar[r]^-{\pi_Z} & Z'. 
}
\]
By~\cite[Proposition 6.3]{Bir19}, there exists a generalized pair
$(Z',B_{Z'}+M_{Z'})$ so that
\[
K_{X'}+B'\sim_\qq {f'}^*(K_{Z'}+B_{Z'}+M_{Z'}).
\]
Moreover, the following conditions are satisfied:
\begin{enumerate} 
    \item The coefficients of $B_{Z'}$ belongs to a set, that satisfies the descending chain condition, and which only depends on $n$ and $N$, 
    \item $qM_{Z'_0}$ is Cartier, where $Z_0'\rightarrow Z'$ is the model where the b-divisor $M_{Z'}$ descends, and $q$ only depends on $n$ and $N$, and
    \item the linear equivalence $q{f'}^*(K_{Z'}+B_{Z'}+M_{Z'})\sim K_{X'}+B'$ holds.
\end{enumerate}
By~\cite[Proposition 5.3]{FM20}, condition (3),
and the fact that $N(K_{X'}+B)\sim 0$, we conclude that 
\[
Nq(K_{Z'}+B_{Z'}+M_{Z'})\sim 0.
\]
Let $M:=Nq$.
Consider the generalized pair 
\[
K_Z+B_Z+M_Z = \pi_Z^*(K_{Z'}+B_{Z'}+M_{Z'}).
\]
By Proposition~\ref{prop:quot-gen-pair}, the generalized pair $(Z,B_Z+M_Z)$
is a $G_Z$-equivariant $M$-complement.
Here, $M$ only depends on $n$ and $N$.
\end{proof}

\begin{lemma}\label{lem:G-equiv-gklt-rc-cbf}
Let $n$ and $N$ be positive integers.
There exists a constant $M:=M(n,N)$, only depending on $n$ and $N$, satisfying the following.
Let $G$ be a finite group.
Let $f\colon X\rightarrow Z$ be a $G$-equivariant Fano type morphism where $X$ is a $n$-dimensional variety.
Let $(X,B+M)$ be a $G$-equivariant generalized klt $N$-complement.
Let $G_Z$ be the quotient group acting on the base.
Then, we can write
\[
K_X+B+M\sim_\qq f^*(K_Z+B_Z+M_Z),
\]
where $(Z,B_Z+M_Z)$ is a $G_Z$-equivariant generalized klt $M$-complement.
Here, $(Z,B_Z+M_Z)$ is defined by the canonical bundle formula (see Definition~\ref{def:cbf}).
\end{lemma}

\begin{proof}
Let $\pi_X\colon X\rightarrow X'$
be the quotient by $G$
and let $\pi_Z\colon Z\rightarrow Z'$ 
be the quotient by $G_Z$.
Then, we have an induced Fano type morphism
$f\colon X'\rightarrow Z'$.
Let $(X',B'+M')$ be the generalized $N$-complement 
that pullbacks to $(X,B+M)$.
Note that $K_{X'}+B'+M'\sim_\qq 0$.
In particular, it is $\qq$-trivial over $Z'$.
We have a commutative diagram as follows:
\[
\xymatrix{
X\ar[r]^-{\pi_X} \ar[d]_-{f} & X'\ar[d]^-{f'} \\
Z\ar[r]^-{\pi_Z} & Z'. 
}
\]
By~\cite[Theorem 1.5]{FM20}, there exists a generalized pair
$(Z',B_{Z'}+M_{Z'})$ so that
\[
K_{X'}+B'+M' \sim_\qq {f'}^*(K_{Z'}+B_{Z'}+M_{Z'}).
\]
Moreover, the following conditions are satisfied:
\begin{enumerate} 
    \item The coefficients of $B_{Z'}$ belongs to a set, that satisfies the descending chain condition, and which only depends on $n$ and $N$, 
    \item $qM_{Z'}$ is Cartier in the model where it descends, where $q$ only depends on $n$ and $N$, and 
    \item the linear equivalence $q{f'}^*(K_{Z'}+B_{Z'}+M_{Z'})\sim K_{X'}+B'$ holds.
\end{enumerate}
By~\cite[Theorem 5.3]{FM20}, condition (3), and the fact that $N(K_{X'}+B)\sim 0$,
we conclude that
\[
Nq(K_{Z'}+B_{Z'}+M_{Z'})\sim 0.
\]
Let $M:=Nq$.
Consider the generalized pair 
\[
K_Z+B_Z+M_Z = \pi_Z^*(K_{Z'}+B_{Z'}+M_{Z'}).
\]
By Proposition~\ref{prop:quot-gen-pair}, the generalized pair $(Z,B_Z+M_Z)$
is a $G_Z$-equivariant $M$-complement.
Here, $M$ only depends on $n$ and $N$.
Observe that we have the induced $\qq$-linear equivalence 
\[
K_X+B+M\sim_\qq f^*(K_Z+B_Z+M_Z).
\]
\end{proof}

\begin{lemma}\label{lem:reg-fib-or-base-non-trivial}
Let $X\rightarrow Z$ be a contraction.
Let $(X,B+M)$ be a generalized pair
which is log Calabi-Yau over $Z$.
Let $(Z,B_Z+M_Z)$ be the generalized log canonical pair induced
by the canonical bundle formula.
Let $(F,B_F+M_F)$ be the generalized log canonical pair
obtained by adjunction to a general fiber.
Assume that $\mathcal{D}(X,B+M)$ is non-empty.
Then either $\mathcal{D}(F,B_F+M_F)$ or $\mathcal{D}(Z,B_Z+M_Z)$ are non-empty.
\end{lemma}

\begin{proof}
The dual complex $\mathcal{D}(X,B+M)$ is non-empty if and only if 
$(X,B+M)$ has a generalized log canonical center.
If $(X,B+M)$ admits a generalized log canonical center which is horizontal over $Z$,
then $(F,B_F+M_F)$ admits a generalized log canonical center.
Hence, we have that $\mathcal{D}(F,B_F+M_F)$ is non-trivial.
If $(X,B+M)$ admits a generalized log canonical center which is vertical over $Z$,
then $(Z,B_Z+M_Z)$ admits a generalized log canonical center.
Hence, we have that $\mathcal{D}(Z,B_Z+M_Z)$ is non-trivial.
\end{proof}

\subsection{Fibrations} 
In this subsection, we will study Fano type fibrations.
We will prove the existence of certain special modifications
of Fano type fibrations that will be used in Section~\ref{sec:rank-reg}.

\begin{lemma}\label{lem:contr-minimal-rel-dim}
Let $G$ be a finite group.
Let $(X,B+M)$ be $G$-equivariant generalized log Calabi-Yau pair.
Let $X\rightarrow Z$ be a $G$-equivariant Fano type fibration.
Consider the set $\mathcal{C}$ of diagrams of the form: 
\[
\xymatrix{
X  \ar[d] & X' \ar@{-->}[l]^-{\pi} \ar[d]^-{\phi} \\
Z & Z' \ar[l],
}
\]
where $\pi$ is a $G$-equivariant birational map that only extracts
non-canonical places of $(X,B+M)$ and $\phi$ is a $G$-equivariant Fano type fibration.
Then, in the set $\mathcal{C}$ there is a diagram so that the
relative dimension of $\phi$ is minimal.
Furthermore, we can choose a diagram with minimal relative dimension of $\phi\colon X'\rightarrow Z'$ for which $X'$ has canonical singularities, $-K_{X'}$ is ample over $Z'$, and $\rho^G(X'/Z')=1$.
\end{lemma}

\begin{proof}
The existence of $\phi\colon X'\rightarrow Z'$ with minimal relative dimension is clear. Let $\phi$ be one of those minimal fibrations.
Let $X''$ be a $G$-equivariant canonicalization of $X'$.
Note that $X''\rightarrow X'$ only extract non-canonical places of $(X',B'+M')$.
Hence, the composition $X''\dashrightarrow X$ only extract non-canonical places of $(X',B'+M')$.
We run a $G$-equivariant MMP for $K_{X''}$ over $Z'$.
This equivariant minimal model program must terminate with a $G$-equivariant Mori fiber space.
Let $X''\dashrightarrow  X_1$ be the sequence of flips and divisorial contractions
and $\phi_1 \colon X_1\rightarrow Z_1$ be the $G$-equivariant Mori fiber space.
If the relative dimension of $\phi_1$ is less than the one of $\phi$, then we contradict the minimality of the relative dimension of $\phi$.
We conclude that the morphism $Z_1 \rightarrow Z'$ must be birational
and $\phi_1$ has minimal dimension as well.
\end{proof}

\begin{lemma}\label{lem:compatible-dlt-mod}
Let $G$ be a finite group.
Let $X\rightarrow Z$ be a $G$-equivariant Fano type fibration.
Let $G_Z$ be the quotient group acting on $Z$.
Let $(X,B+M)$ be a $G$-equivariant log Calabi-Yau pair.
Let $(Z,B_Z+M_Z)$ be the $G_Z$-equivariant log Calabi-Yau pair
obtained by the generalized canonical bundle formula.
Let $Z'\rightarrow Z$ be a generalized dlt modification of $(Z,B_Z+M_Z)$.
Then, we can find a birational dlt model $(X',B'+M')$ of $(X,B+M)$ which admits
a $G$-equivariant Fano type fibration $X'\rightarrow Z'$ making the following diagram commutative:
\begin{equation}\label{eq:com-diag-dlt}
\xymatrix{
(X,B+M) \ar[d] & (X',B'+M') \ar@{-->}[l]\ar[d] \\
(Z,B_Z+M_Z) & (Z',B_{Z'}+M_{Z'}) \ar[l].
}
\end{equation}
In the diagram, the generalized pairs in the gdlt modifications are induced by log pullback.
Furthermore, we can assume that the preimage of the locus of generalized log canonical centers of $(Z',B_{Z'}+M_{Z'})$ is contained in the support of $\lfloor B'\rfloor$.
\end{lemma}

\begin{proof}
We can find a $G$-equivariant log resolution of singularities $X_0\rightarrow X$
of the generalized pair $(X,B+M)$.
We may assume that there is a $G$-equivariant morphism $X_0\rightarrow Z'$.
Let $(X_0,B_0+M_0)$ be the generalized sub-pair obtained by log pullback of $(X,B+M)$ to $X_0$.
We may assume that for every prime component of $\lfloor B_{Z'}\rfloor$, there is a prime component of $\lfloor B_0\rfloor$ mapping to it.
Let $B'_0$ be the strict transform of $B$ on $X_0$ plus the reduced exceptional divisor of $X_0\rightarrow X$.
Then, we can write 
\[
K_{X_0}+B'_0 + M_0 \sim_{\qq,X} F,
\]
where $F$ is an effective divisor which is exceptional over $X$.
We can write $F=F_{\rm vert}+F_{\rm hor}$ for the vertical and horizontal components over $Z'$.
Note that $F$ is a $G$-invariant divisor of $X_0$.
Applying the negativity lemma to the general fiber of $X_0\rightarrow Z'$, we conclude that there is a containment
\begin{equation}\label{eq:cont} 
{\rm Bs}_{-}(K_{X_0}+B_0'+M_0/Z')\supset F_{\rm hor}.
\end{equation} 
We claim that the divisor $F_{\rm vert}$ is degenerate over $Z'$ (see~\cite[Definition 2.8]{Lai11}).
Indeed, every $G$-prime component of 
$F$ satisfies:
\begin{enumerate}
    \item Its center on $Z'$ has codimension at least two, or
    \item there is a $G$-prime component of $\lfloor B_0\rfloor$ with the same image on $Z'$.
\end{enumerate}
In the first case, the component is degenerate in the usual sense~\cite[Definition 2.8.1]{Lai11}
while in the second case it is of insufficient fiber type~\cite[Definition 2.8.2]{Lai11}.
We run a $G$-equivariant MMP for $(K_{X_0}+B_0'+M_0)$ over $Z'$ with scaling of a $G$-equivariant ample divisor over $Z'$.
By the contaiment~\eqref{eq:cont}, we conclude that after finitely many steps, all the components of $F_{\rm hor}$ are contracted.
Since $F_{\rm vert}$ is degenerate over $Z'$, we conclude that after finitely many steps, 
all the components of $F_{\rm vert}$ are contracted (see, e.g.,~\cite[Proposition 5.3]{Mor20b}).
Let $X\dashrightarrow X'$ be such a minimal model program.
Let $B'$ be the push-forward of $B_0$ on $X'$ and $M'$ be the trace of $M_0$ on $X'$.
Then, the generalized pair $(X',B'+M')$ is generalized dlt and log Calabi-Yau.
Note that $(X',B'+M')$ is $G$-equivariantly log crepant to $(X,B+M)$.
Thus, $(X',B'+M')$ is a birational generalized dlt model of $(X,B+M)$ which admits a Fano type fibration to $Z'$.
Indeed, $X'$ is of Fano type due to Lemma~\ref{lem:extraction-FT}.

Now, we turn to prove the last statement.
For each prime component $P_{Z'}$ of $\lfloor B_{Z'}\rfloor$, 
we can find a prime component $P_{X'}$ of $\lfloor B'\rfloor$ which maps to $P_{Z'}$.
Let $P$ be the $G$-closure of $P_{X'}$ on $X'$.
We run a $-P$ over $Z'$.
This MMP must terminate by the Fano type condition over $Z'$.
This MMP will contract all the effective $G$-invariant divisors (different from $P$)
which maps onto $P_Z$.
Hence, after finitely many steps, we obtain a model in which $P$
equals the reduced pullback of $P_{Z'}$.
We proceed analogously with all the prime components of $\lfloor B_{Z'}\rfloor$.
After finitely many steps, we obtain a model $(X'',B''+M'')$ which admits a $G$-equivariant Fano type morphism to $f''\colon X''\rightarrow Z'$ satisfying:
\begin{enumerate}
    \item The generalized pair $(X'',B''+M'')$ has $G$-equivariant generalized lc singularities, and
    \item the preimage of the locus of generalized log canonical centers of $(Z',B_{Z'}+M_{Z'})$ is contained in the support of $\lfloor B''\rfloor$.
\end{enumerate}
The statement $(1)$ is clear, so we argue $(2)$.
Since $(Z',B_{Z'}+M_{Z'})$ is generalized dlt, we have 
\begin{equation}\label{eq:union-P}
{\rm glcc}(Z',B_{Z'}+M_{Z'}) = 
\bigcup_{\substack{P \subset \lfloor B_{Z'}\rfloor \\ P \text{ prime }}} \supp(P).
\end{equation} 
By construction, for each $P$ as in equality~\eqref{eq:union-P}, we have 
$P_{X''}={f''}^*(P)_{\rm red}$ for some $G$-prime component $P_{X''}$ of $\lfloor B''\rfloor$.
Hence, we have that 
\[
{f''}^{-1}
\left(
{\rm glcc}(Z',B_{Z'}+M_{Z'}) 
\right)  
\subseteq
\supp \lfloor B''\rfloor.
\]
To conclude, it suffices to take a generalized dlt modification of $(X'',B''+M'')$. The dlt modification only extract divisors with generalized log discrepancy zero with respect to $(X'',B''+M'')$. Thus, the property $(2)$ is preserved when passing to a gdlt modification.
\end{proof}

\begin{lemma}\label{lem:special-model} 
Let $G$ be a finite group.
Let $(X,B+M)$ be a $G$-equivariant generalized log Calabi-Yau pair.
Let $\phi\colon X\rightarrow Z$ be a $G$-equivariant Fano type fibration satisfying the following conditions:
\begin{enumerate}
    \item[(a)] $\rho^G(X/Z)=1$, and 
    \item[(b)] $\phi$ has minimal relative dimension in the sense of Lemma~\ref{lem:contr-minimal-rel-dim}.
\end{enumerate}
Let $G_Z$ be the quotient group of $G$ acting on the base.
Let $(Z,B_Z+M_Z)$ be the $G_Z$-equivariant generalized lc pair induced on the base
by the canonical bundle formula.
Let $Z'\rightarrow Z$ be a $G_Z$-equivariant generalized dlt modification of $(Z,B_Z+M_Z)$.
Then, we can find a $G$-equivariant birational generalized lc model $(X',B'+M')$ of $(X,B+M)$
which satisfies the following conditions: 
\begin{enumerate}
\item The variety $X'$ admits a $G$-equivariant Fano type fibration $X'\rightarrow Z$,
\item the generalized pair $(X',B'+M')$ has generalized lc singularities,
\item the preimage of the locus of generalized log canonical centers of $(Z,B_Z+M_Z)$ is contained in the union of the vertical components of $\lfloor B'\rfloor$,
\item the generalized pair
$(X', \lfloor B'\rfloor_{\rm vert})$ is dlt, and
\item the contraction $X\rightarrow Z$ has relative $G$-Picard rank one on the complement of ${\rm glcc}(Z,B_Z+M_Z)$.
\end{enumerate}
\end{lemma}

\begin{proof}
Let $X_1$ be the model provided by Lemma~\ref{lem:compatible-dlt-mod}.
Let $(X_1,B_1+M_1)$ be the $G$-equivariant generalized dlt pair induced on $X_1$.
We may assume that the preimage of the locus of generalized log canonical centers of $(Z',B_{Z'}+M_{Z'})$ is contained in the support of $\lfloor B'\rfloor$.
As in~\eqref{eq:com-diag-dlt}, we have a commutative diagram:
\[
\xymatrix{
(X,B+M) \ar[d] & (X_1,B_1+M_1) \ar@{-->}[l]\ar[d] \\
(Z,B_Z+M_Z) & (Z',B_{Z'}+M_{Z'}). \ar[l]
}
\]
Let $U$ be the complement of the set of generalized log canonical centers of $(Z,B_Z+M_Z)$.
Then, $U$ is a Zariski open dense subset of $Z$.
Furthermore, $Z'\rightarrow Z$ is an isomorphism over $U$
and the induced map $X_1\dashrightarrow U$ is a morphism.
Let $U_X$ be the preimage of $U$ on $X$.
Let $U_{X_1}$ be the preimage of $U$ on $X_1$.
Then, we have an induced $G$-equivariant birational contraction $U_{X_1}\dashrightarrow U_X$.
Let $A$ be a $G$-equivariant ample divisor on $X$.
Let $A_{X_1}$ be its pullback to $X_1$.
We run a $G$-equivariant $A_{X_1}$-MMP over $Z'$.
Let $X_1\dashrightarrow X_2$ be this minimal model program.
Since $U_X\rightarrow U$ is the $G$-equivariant ample model for $A$ relatively to $U$,
we conclude that the preimage $U_{X_2}$ of $U$ on $X_2$ is $G$-equivariantly isomorphic to $U_X$.
In particular, the $G$-equivariant Fano type fibration $X_2\rightarrow Z'$
has relative $G$-Picard rank one on the complement of ${\rm glcc}(Z',B_{Z'}+M_{Z'})$.
Note that the induced $G$-equivariant generalized pair $(X_2,B_2+M_2)$ is generalized log canonical. Moreover, the preimage of the locus of generalized log canonical centers of $(Z',B_{Z'}+M_{Z'})$ is contained in $\lfloor B_2\rfloor$.
By Lemma~\ref{lem:extraction-FT}, we conclude that $X_2$ is of Fano type over $Z'$.

Note that the generalized pair $(X_2,B_2+M_2)$ satisfies $(1)$-$(3)$ and $(5)$.
Thus, it suffices to perform a modification so that it further satisfies $(4)$.
We observe that $X_2$ is klt over $U$.
Hence, the log pair $(X_2,\lfloor B_2\rfloor_{\rm vert})$ is klt over $U$.
Furthermore, every generalized log canonical center of $(X_2,\lfloor B_2\rfloor_{\rm vert})$ maps to the complement of $U$.
We define $(X',B'+M')$ to be the log pull-back of $(X_2,B_2+M_2)$ to a $G$-equivariant generalized dlt modification of $(X_2,\lfloor B_2\rfloor_{\rm vert})$.
We denote by $\psi\colon X'\rightarrow X_2$ the projective birational morphism.
The condition $(2)$ is clear.
This modification is an isomorphism over the preimage of $U$ so $(5)$ is preserved.
Every generalized log canonical place of $(X_2,\lfloor B_2\rfloor_{\rm vert})$ is a generalized log canonical place of $(X_2,B_2+M_2)$.
We conclude that 
\[
K_{X'}+\lfloor B'\rfloor_{\rm vert} = 
\psi^*(K_{X_2}+\lfloor B_2\rfloor_{\rm vert}).
\] 
Hence, the pair $(X_2,\lfloor B_2\rfloor_{\rm vert})$ is dlt and
$\lfloor B_2 \rfloor_{\rm vert}$ consists of all the prime divisors of $X'$ which maps to the complement of $U$. This implies $(4)$ and $(3)$. 
The variety $X'$ is of Fano type of type by Lemma~\ref{lem:extraction-FT}.
This implies condition $(1)$.
\end{proof}

\section{Rationally connected varieties}

In this section, we study finite actions
on rationally connected varieties.
The main result is that finite actions on rationally
connected varieties admitting log Calabi-Yau klt structures must have order bounded by a constant on the dimension of the variety.
We refer the reader to~\cite{KMM92,KSC04} for the literature on rationally connected varieties.

\begin{theorem}\label{introthm:RC-action}
Let $n$ and $N$ be positive integers.
There exists a constant $c(n,N)$, only depending on $n$ and $N$, satisfying the following.
Let $G$ be a finite group.
Let $X$ be a $n$-dimesional rationally connected variety.
Let $(X,B+M)$ be a $G$-equivariant $N$-complement.
Assume that $(X,B+M)$ has generalized klt singularities.
Then the finite group 
$G$ has order at most $c(n,N)$.
\end{theorem}

\begin{proof}
By passing to a $G$-equivariant small $G\qq$-factorialization of $X$, we may assume that 
$X$ is $G\qq$-factorial.
Let $X'\rightarrow X$ be a
$G$-equivariant canonicalization of 
$X$.
Let $(X',B'+M')$ be the generalized pair
obtained by log pullback of
$(X,B+M)$ to $X'$.
Then, $(X',B'+M')$ is a $G$-equivariant generalized klt $N$-complement
and $K_{X'}$ has canonical singularities.
Observe that $X'$ is rationally connected. Hence, $K_{X'}$ is not a pseudo-effective divisor.
We run a $G$-equivariant $K_{X'}$-MMP which terminates with a $G$-equivariant Mori fiber space $f\colon X_1\rightarrow Z_1$.
We denote by $(X_1,B_1+M_1)$ the generalized klt pair induced on $X_1$.
We have an exact sequence
\[
1\rightarrow G_F \rightarrow G\rightarrow G_Z\rightarrow 1, 
\]
where $G_F$ is the largest normal subgroup acting fiberwise.
First, we assume that $\dim Z_1>0$.
By Lemma~\ref{lem:G-equiv-gklt-rc-cbf}, we have an induced 
$G_Z$-equivariant generalized klt 
$M$-complement 
\[
(Z_1, B_{Z_1}+M_{Z_1}), 
\]
so that 
\[
K_{X_1}+B_1+M_1 \sim_\qq 
f^*(K_{Z_1}+B_{Z_1}+M_{Z_1}).
\]
Here, $M$ only depends on $N$ and $n$.
By induction on the dimension, there is a constant $c_0(z,M)$, which only depends on $z=\dim Z_1$ and $M$, so that
$|G_b|< c_0(z,M)$.
On the other hand, 
let $(F,B_F+M_F)$ be the generalized pair obtained by adjunction to the general fiber $F$ of $X_1\rightarrow Z_1$.
Then, the generalized pair
$(F,B_F+M_F)$ is a $G_F$-equivariant 
generalized klt $N$-complement.
We have that $\dim F <\dim X$.
Hence, by induction on the dimension,
there exists a constant $c_0(f,N)$, 
which only depends on $f=\dim F$ and $N$ so that $|G_f|<c_0(f,N)$.
Thus, we conclude that
\begin{equation}
\label{eq:ineq}
|G| < c_0(z,M) + c_0(f,N).
\end{equation} 
Note that the constant on the right-hand side only depend on
$z,f,M$ and $N$.
Furthermore, $z$ and $f$ are integers strictly bounded above by $n$, and $M$ only depends on $N$ and $n$.
We conclude that the upper bound in 
inequality~\eqref{eq:ineq} only depends on $n$ and $N$.

Now, we assume that $Z_1$ is zero dimensional.
This means that $X_1$ is a canonical variety with $\rho^G(X_1)$ equal to one.
Since $K_{X_1}$ is $G$-invariant, we conclude that $K_{X_1}$ is anti-ample,
as it is not pseudo-effective.
Hence, $X_1$ is a $G$-equivariant canonical Fano variety.
By~\cite[Theorem 1.1]{Bir21}, we conclude that $X_1$ belongs to a bounded family.
Let $\pi\colon X_1\rightarrow Y_1$
be the quotient of $X_1$ by $G$.
We can write 
\[
K_{X_1} = \pi^*(K_{Y_1}+B_{Y_1}),
\]
for some boundary $B_{Y_1}$ on $Y_1$ with standard coefficients.
The $N$-complement 
$(X_1,B_1+M_1)$ is the generalized pullback 
of a generalized $N$-complement of $(Y_1,B_{Y_1})$.
Thus, we can find a generalized $N$-complement
$(Y_1,\Gamma_{Y_1}+M_{Y_1})$ so that
\[
K_{X_1}+B_1+M_1 =
\pi^*(K_{Y_1}+\Gamma_{Y_1}+M_{Y_1}).
\]
Since $(X_1,B_1+M_1)$ is generalized klt, 
we conclude that $(Y_1,\Gamma_{Y_1}+M_{Y_1})$ is generalized klt as well.
Since $(Y_1,\Gamma_{Y_1}+M_{Y_1})$ is generalized klt and 
\[
N(K_{Y_1}+\Gamma_{Y_1}+M_{Y_1})\sim 0,
\]
we conclude that 
$a_F(Y_1,\Gamma_{Y_1}+M_{Y_1})\geq \frac{1}{N}$ for every prime divisor $F$ over $Y_1$.
Hence, $(Y_1,\Gamma_{Y_1}+M_{Y_1})$ is generalized $\frac{1}{N}$-log canonical.
As a consequence, the pair
$(Y_1,B_{Y_1})$ is
$\frac{1}{N}$-log canonical 
and $K_{Y_1}+B_{Y_1}$ is anti-ample.
We conclude that $Y_1$ belongs to a bounded family.
We run a $M_{Y_1}$-MMP which terminates in a good minimal model $Y_1\dashrightarrow Y_2$.
Indeed, the divisor $M_{Y_1}$ is pseudo-effective and its diminished base locus is small. So $Y_1\dashrightarrow Y_2$ is a sequence of $M_{Y_1}$-flips.
Let $M_{Y_2}$ be the push-forward of $M_{Y_1}$ to $Y_2$.
Analogously, let $\Gamma_{Y_2}$ be the push-forward of $\Gamma_{Y_1}$ to $Y_2$.
Note that $NM_{Y_2}$ is a Weil divisor.
Since $Y_2$ belongs to a bounded family, 
we can find a constant $N'$, which only depends on $\dim X$, so that
$N'NM_{Y_2}$ is a semiample Cartier divisor.
Let $\Omega_{Y_2}$ be the quotient by $N'N$ of a general section of $N'NM_{Y_2}$. Hence, the log pair
\[
\left( 
Y_2,
\Gamma_{Y_2} 
+ 
\Omega_{Y_2} \right)
\]
is generalized klt and 
\[
K_{Y_2}+\Gamma_{Y_2}+\Omega_{Y_2}\sim_\qq 0.
\]
Replacing $N$ with a multiple which only depends on $n$, we may assume that 
\[
N(K_{Y_2}+\Gamma_{Y_2}+\Omega_{Y_2})\sim 0.
\]
Let $(Y_1,\Gamma_{Y_1}+\Omega_{Y_1})$ be the log pair obtained by pushing-forward
$\Omega_{Y_2}$ to $Y_1$.
Note that 
$(Y_1,\Gamma_{Y_1}+\Omega_{Y_1})$ is a klt pair and
\[
N(K_{Y_1}+\Gamma_{Y_1}+\Omega_{Y_1})\sim 0.
\]
Note that $\Gamma_{Y_1}\geq B_{Y_1}$.
Hence, the pullback of 
$(Y_1,\Gamma_{Y_1}+\Omega_{Y_1})$ to $X_1$
equals $(X_1,\Omega_1)$, where
$\Omega_1 \geq 0$.
The pair $(X_1,\Omega_1)$ is klt,
$G$-equivariant, 
and $N(K_{X_1} +\Omega_1)\sim 0$.

By the proof of~\cite[Theorem 4.1]{Mor20b}, there exists a constant $k(n)$, only depending on $n$, which satisfies the following.
If $|G|>k(n)$, then 
$A<\mathbb{G}_m^r \leqslant {\rm Aut}(X_1,\Omega_1)$ for a positive integer $r$ and a normal subgroup
$A\triangleleft G$ of index at most $k(n)$.
Assume that $|G|>k(n)$.
Then the log pair $(X_1,\Omega_1)$ is $\mathbb{G}_m^r$-invariant.
Let $\widetilde{X}_1\rightarrow X_1$ be a torus equivariant birational morphism 
so that $\widetilde{X}_1$ admits a good quotient for the torus action.
Denote such good quotient by 
$\widetilde{X}_1\rightarrow Z$.
Hence, the general fiber of the morphism
$\widetilde{X}_1\rightarrow Z$ is a projective toric variety $F$.
Let $(\widetilde{X}_1,\widetilde{\Omega}_1)$ be the log sub-pair obtained by log pullback of
$(X_1,\Omega_1)$.
Observe that $(\widetilde{X}_1,\widetilde{\Omega}_1)$ is a log Calabi-Yau sub-pair
which is $\mathbb{G}_m^r$-invariant.
Let $(F,\Omega_F)$ be the restriction of
$(\widetilde{X}_1,\widetilde{\Omega}_1)$ to a general fiber of
$\widetilde{X}_1\rightarrow Z$.
Then, $(F,\Omega_F)$ is a log Calabi-Yau, torus invariant, sub-pair.
This implies that $\Omega_F$ is the reduced toric boundary.
In particular, $\widetilde{\Omega}_1$ must have a prime component of coefficient one, 
which implies that
$(X_1,\Omega_1)$ is strictly log canonical.
This gives a contradiction.
We conclude that $|G|<k(n)$, finishing the proof of the theorem.
\end{proof}

Our next aim is to prove that given a generalized klt log Calabi-Yau pair,
up to a birational contraction,
we can find a projective birational morphism with rationally connected fibers and $K$-trivial base.
The ideas of the following proposition are mostly those in~\cite[Theorem 3.2]{DCS16}. 
We give a proof for the sake of completeness, 
as we are dealing with equivariant generalized pairs.

\begin{proposition}\label{prop:MRC-gklt}
Let $G$ be a finite group.
Let $(X,B+M)$ be a $G$-equivariant generalized klt
log Calabi-Yau pair.
Then, there exists a $G$-equivariant birational contraction 
$X\dashrightarrow X'$ satisfying the following.
Either $X'$ is rationally connected or there exists a $G$-equivariant projective fibration 
$X'\rightarrow Z$ so that
the following conditions are satisfied:
\begin{enumerate}
    \item $Z$ is positive dimensional and has klt singularities, 
    \item the fibers of $X'\rightarrow Z$ are rationally connected, and 
    \item $K_Z\sim_\qq 0$.
\end{enumerate}
\end{proposition}

Before proving the proposition, we will prove two lemmas
using the two ray game.

\begin{lemma}\label{lem:2-ray-1}
Let $G$ be a finite group.
Let $(X,B+M)$ be a $G$-equivariant generalized klt log Calabi-Yau pair.
Assume that $X$ has $G\qq$-factorial singularities.
Assume $X$ is not $\qq$-linearly trivial.
Let $p\colon X\rightarrow Y$ be a $G$-equivariant Mori fiber space 
and let $G_Z$ be the quotient group acting on the base.
Let $q\colon Y\rightarrow Y'$ be a $G_Z$-equivariant divisorial contraction.
Then, there exists:
\begin{enumerate}
\item a $G$-equivariant generalized klt log Calabi-Yau pair
$(X',B'+M')$ so that the canonical divisor $K_{X'}$ is not $\qq$-linearly trivial,
\item a $G$-equivariant Mori fiber space $p'\colon X'\rightarrow Y'$, and 
\item a $G$-equivariant contraction $f\colon X\dashrightarrow X'$ which is a composition of finitely many $G$-equivariant $(K_X+B+M)$-flops followed by a $G$-equivariant divisorial contraction, 
\end{enumerate} 
so that the following diagram is commutative
\[
\xymatrix{
G\acts X \ar@{-->}[r]^-{f} \ar[d]_-{p} & X'\ar[d]^-{p'} \\
G_Z\acts Y \ar[r]^-{q} & Y'.
}
\]
\end{lemma}

\begin{proof}
First, observe that $\rho^G(X/Y')=2$.
The divisor $B+M$ is big over $Y'$ so $X$ is a Mori dream space over $Y'$ (see, e.g.,~\cite[Corollary 1.3.2]{BCHM10}).
Let $F\subset Y$ be the $G_Z$-invariant effective divisor contracted by $q$.
We define $E=p^*(F)$.
Note that $E$ is a $G$-invariant effective divisor whose image on $Y'$ has codimension at least two.
We run a $G$-equivariant MMP for $E$ over $Y'$.
By~\cite[Lemma 2.10]{Lai11}, we conclude that this MMP consists of finitely many
$G$-equivariant $(K_X+B+M)$-flops followed by a divisorial contraction.
Let $X'$ be the model obtained after the contraction.
We have that the morphism $p'\colon X'\rightarrow Y'$ has $G$-invariant Picard rank one.
Observe that $B+M$ is big over $Y'$. 
Then, $B'+M'=f_*(B+M)$ is big over $Y'$.
We conclude that 
$K_{X'}$ is not $\qq$-linearly trivial.
Indeed, $-K_{X'}$ is ample over $Y'$.
\end{proof}

\begin{lemma}\label{lem:2-ray-2}
Let $G$ be a finite group.
Let $(X,B+M)$ be a $G$-equivariant generalized klt 
log Calabi-Yau pair.
Assume that $X$ is $G\qq$-factorial.
Let $p\colon X\rightarrow Y$ be a $G$-equivariant Mori fiber space
Let $G_Z$ be the quotient group of $G$ acting on $Y$.
Let $q\colon Y\dashrightarrow Y'$ be a $G_Z$-equivariant $D$-flip
for some $G_Z$-equivariant divisor $D\subset Y$.
Then, there exists: 
\begin{enumerate}
    \item A $G$-equivariant generalized klt log Calabi-Yau pair
    $(X',B'+M')$, 
    \item a $G$-equivariant Mori fiber space $p'\colon X'\rightarrow Y'$, and 
    \item a $G$-equivariant birational map $f\colon X\dashrightarrow X'$ which is an isomorphism in codimension one,
\end{enumerate}
so that the following diagram is commutative
\[
\xymatrix{
G\acts X \ar@{-->}[r]^-{f} \ar[d]_-{p} & X'\ar[d]^-{p'} \\
G_Z\acts Y \ar@{-->}[r]^-{q} & Y'.
}
\]
\end{lemma}

\begin{proof}
Let $Y\rightarrow W$ be the $G_Z$-equivariant flipping contraction.
$X$ is a $G$-equivariant Mori dream space over $W$.
Indeed, $X\rightarrow W$ is a $G$-equivariant Fano type morphism.
The divisor $p^*D$ is $G$-invariant and movable over $W$.
Indeed, the diminished base locus of $D$ over $W$ is small,
so the diminished base locus of $p^*D$ over $W$ is small as well.
We run a $G$-equivariant $p^*D$-MMP over $W$, 
which terminates with a good minimal model $f\colon X\dashrightarrow X'$.
The $G$-equivariant map $f\colon X\dashrightarrow X'$ is a sequence of $G$-equivariant $(K_X+B+M)$-flops.
The ample model of $f_*p^*D$ over $X'$ is isomorphic to $Y'$.
Since $B+M$ is big over $W$, we conclude that 
$f_*(B+M)=B'+M'$ is big over $W$.
\end{proof}

\begin{proof}[Proof of Proposition~\ref{prop:MRC-gklt}]
Since $(X,B+M)$ is generalized klt, we may replace it with a $G$-equivariant small $G\qq$-factorial modification.
Hence, we may assume that $X$ itself is $G\qq$-factorial.
If $K_X$ is $\qq$-linearly trivial, then we can take $X'$ equal to $X$ and $X\rightarrow Z$ to be the identity morphism.
If $K_X$ is not $\qq$-linearly trivial, then $K_X$ is not pseudo-effective.
We run a $G$-equivariant $K_X$-MMP which terminates with a $G$-equivariant Mori fiber space $X'\dashrightarrow Y_1$.
Let $G_Z$ be the quotient group of $G$ acting on the base $Y_1$.
If $K_{Y_1} \sim_\qq 0$, then we are done as Mori fiber spaces have
rationally connected fibers.
Otherwise, let $(Y_1,B_{Y_1}+M_{Y_1})$ be the $G_Z$-equivariant generalized klt pair obtained by the generalized canonical bundle formula.
If $K_{Y_1}$ is not $\qq$-linearly trivial, then $B_{Y_1}+M_{Y_1}$ is a non-trivial element in the cone of pseudo-effective divisors. Thus, $K_{Y_1}$ is not pseudo-effective.
Then, we can run a $G_Z$-equivariant $K_{Y_1}$-MMP which terminates with a $G_Z$-equivariant Mori fiber space.
Let 
\[
\xymatrix{
Y_1 \ar@{-->}[r]^-{q_0} & 
Y'_1 \ar@{-->}[r]^-{q_1} &
Y'_2 \ar@{-->}[r]^-{q_2} & 
\dots \ar@{-->}[r]^-{q_{m-1}} & 
Y'_m \ar[d]^-{p_1}  \\
& & & & Y_2,
}
\]
be the steps of such $G_Z$-equivariant minimal model program.
This means that $p_1$ is a $G_Z$-equivariant Mori fiber space
and each $q_i$ is either 
\begin{enumerate}
    \item a $G_Z$-equivariant divisorial contraction, or 
    \item a $G_Z$-equivariant flip.
\end{enumerate}
By Lemma~\ref{lem:2-ray-1} and Lemma~\ref{lem:2-ray-2}, 
we can complete a diagram as follows
\[
\xymatrix{
X' \ar@{-->}[r]^-{s_0} \ar[d]^-{t_0} & 
X'_1 \ar@{-->}[r]^-{s_1}  \ar[d]^-{t_1} &
X'_2 \ar@{-->}[r]^-{s_2}  \ar[d]^-{t_2} & 
\dots \ar@{-->}[r]^-{s_{m-1}} & 
X'_m \ar[d]^-{p_1}  \\
Y_1 \ar@{-->}[r]^-{q_0} & 
Y'_1 \ar@{-->}[r]^-{q_1} &
Y'_2 \ar@{-->}[r]^-{q_2} & 
\dots \ar@{-->}[r]^-{q_{m-1}} & 
Y'_m \ar[d]^-{p_1}  \\
& & & & Y_2
}
\]
where each $s_i$ is a birational contraction.
Furthermore, for each $i$, the divisor $B'_i+M'_i$ is big over $Y_i'$, where
$B'_i+M'_i$ is the push-forward of $B'+M'$ to $X'_i$.
We can replace $t_0\colon X'\rightarrow Y_1$ with 
$p_1 \colon X'_m\rightarrow Y'_m$.
If $K_{Y_2}$ is $\qq$-linearly trivial, then we are done. Otherwise,
we can run a $K_{Y_2}$-MMP and repeat the procedure just explained.
Since each equivariant Mori fiber space has positive dimensional fibers, this process eventually stops.
When it stops, we obtain a $G$-equivariant birational contraction $X\dashrightarrow X'$ and a sequence of equivariant Mori fiber spaces
\[
\xymatrix{
X' \ar[r]^-{q_0} & 
Y_1 \ar[r]^-{q_1} &
Y_2 \ar[r]^-{q_2} & 
\dots \ar[r]^-{q_{k-1}} & 
Y_k 
}.
\]
By construction, either $Y_k$ is a point or $Y_k$ is positive dimensional, $Y_k$ has klt singularities, and $K_{Y_k}\sim_\qq 0$.
If $Y_k$ is a point, then $X'$ is rationally connected.
If $Y_k$ is positive dimensional, 
then the general fiber of $X'\rightarrow Y_k$ is rationally connected.
\end{proof}

\begin{lemma}\label{lem:dom-lcc}
Let $X\rightarrow Z$ be a projective contraction 
so that $Z$ has klt singularities and $K_Z\sim_\qq 0$.
Let $(X,\Gamma)$ be a log pair with log canonical singularities so that 
$K_X+\Gamma \sim_{\qq,Z} 0$.
Then, every log canonical center of $(X,\Gamma)$ dominates $W$.
In particular, any minimal log canonical center
of $(X,\Gamma)$ dominates $W$.
\end{lemma}

\begin{proof}
Assume that $(X,\Gamma)$ admits a log canonical center which is vertical over $Z$.
Let $(Z,B_Z+M_Z)$ be the generalized log canonical pair obtained by the canonical bundle formula.
The vertical log canonical centers of $(X,
\Gamma)$ map to generalized log canonical centers of $(Z,B_Z+M_Z)$.
Since $K_Z\sim_\qq 0$, we have that $B_Z=0$ and $M_Z\sim_\qq 0$.
By the negativity lemma, we have that
$M_{Z'}+F = \phi^* M_Z \sim_\qq 0$,
where $\phi\colon Z'\rightarrow Z$ is the model where the b-divisor descends.
Since $M_{Z'}$ is nef and $F$ is effective, we conclude that $F=0$ and $M_{Z'}\sim_\qq 0$.
Hence, the generalized pair $(Z,M_Z)$ is generalized klt.
This contradicts the fact that $(X,\Gamma)$ has a vertical log canonical center.
\end{proof}

\section{\texorpdfstring{$\mathbb{P}^1$-links and minimal log canonical centers}{P1-links and minimal log canonical centers}}

In this section, we study finite actions on $\pp^1$-links and minimal log canonical centers
of log Calabi-Yau pairs. We will use the following definition for $\pp^1$-links. 
Let $(X,B)$ be a log Calabi-Yau dlt pair.
The minimal log canonical centers of $(X,B)$ correspond
to maximal dimensional cells of $\mathcal{D}(X,B)$.
Moreover, the $\pp^1$-links between minimal log canonical centers of $(X,B)$ correspond to facets which are intersections of two maximal dimensional cells.

\begin{definition}\label{def:p1-link}
{\em 
Let $(X,B+M)$ be a $G$-equivariant generalized pair admitting a 
$G$-equivariant fibration $X\rightarrow Z$.
We say that $X\rightarrow Z$ is {\em $\pp^1$-link} for $(X,B+M)$ if the following conditions are satisfied:
\begin{enumerate}
    \item The general fiber of $X\rightarrow Z$ is isomorphic to $\pp^1$, 
    \item the pair $(X,B+M)$ is generalized plt,
    \item $(X,B+M)$ is log Calabi-Yau over $Z$, i.e., 
    $K_X+B+M\sim_{\qq,Z} 0$, and
    \item the restriction of $(X,B+M)$ to a general fiber is isomorphic to
    $(\pp^1,\{0\}+\{\infty\}$). 
\end{enumerate}
Note that, in particular, this means that $\lfloor B\rfloor$ has two prime components $P_1$ and $P_2$ which maps $G$-equivariantly birational to the base $Z$.
}
\end{definition}

\begin{lemma}\label{lem:simp-complex}
Let $G$ be a finite group.
Let $(X,B+M)$ be a $G$-equivariant generalized log Calabi-Yau pair.
Let $(X',B'+M')$ be $G$-equivariant generalized dlt modification of $(X,B+M)$.
We can find a $G$-equivariant gdlt modification $(X'',B''+M'')$ of $(X,B+M)$, which dominates $(X',B'+M')$, so that
$\mathcal{D}(X'',B''+M'')$ is a simplicial complex all whose maximal dimensional simplices have the same dimension.
\end{lemma}

\begin{proof}
The existence of an equidimensional simplicial complex structure on some gdlt modification $(X'',B''+M'')$ of $(X',B'+M')$ follows from the barycenter subdivision (see, e.g.,~\cite[Remark 10]{dFKX17}).
It suffices to show that this barycenter subdivision can be performed $G$-equivariantly.
Indeed, locally at each strata, the stabilizer of $G$ acts as a finite subgroup of ${\rm GL}_n(\kk)$ (see, e.g.,~\cite[Lemma 2.7.(b)]{FZ05}).
Hence, the maximal ideal of the $G$-closure of each strata is fixed by $G$. 
Thus, the blow-up of the orbit of the strata can be performed $G$-equivariantly.
\end{proof}

In the following lemma, we will use the fact that the dual complex is a pseudomanifold (see, e.g.,~\cite[Theorem 1.6]{FS20}). We give a short argument for this fact in the proof. We will use the connectedness of dual complexes~\cite{FS20,Bir20}.

\begin{lemma}\label{lem:fix-max-dim-cell}
Let $(X,B+M)$ be a $G$-invariant generalized log Calabi-Yau pair.
Let $(X',B'+M')$ be a $G$-equivariant dlt modification of $(X,B+M)$.
Assume that $G$ fixes a minimal generalized log canonical center $W$ of $(X',B'+M')$.
Furthermore, assume that $G$ fixes all the generalized log canonical centers of $(X',B'+M')$ containing $W$.
Then, $G$ fixes all the generalized log canonical centers of $(X',B'+M')$.
In particular, $G$ must fix all the generalized log canonical centers of $(X,B+M)$.
\end{lemma}

\begin{proof}
The statement is clear if the dual complex $\mathcal{D}(X,B+M)$ is zero or one-dimensional, so we may assume it has dimension at least two.

At the generic point of $W$ the variety $X'$ is log smooth and $(X',B';\eta_W)$ is a log smooth pair.
Since $G$ if fixing all the divisors of $B'$ containing $W$, then the normal subgroup of $G$ fixing $W$ pointwise must act torically on $(X',B';\eta_W)$.
By Lemma~\ref{lem:simp-complex}, we may replace $(X',B'+M')$ with a higher $G$-equivariant generalized dlt modification and assume that
$\mathcal{D}(X',B'+M')$ is a simplicial complex.
This replacement is attained by toroidal blow-ups and $G$ acts torically on $(X',B';\eta_W)$. Then, the existence of a fixed minimal center is preserved on $(X',B'+M')$.

Let $\mathcal{D}(X',B'+M')$ be the dual complex which we assume of dimension $r\geq 2$.
We introduce an equivalence relation between $r$-simplices in $\mathcal{D}(X',B'+M')$.
Two simplices $S_1$ and $S_k$ are said to be related, if we can find a sequence of $r$-simplices $S_2,\dots, S_{k-1}$ so that for each $i\in \{1,\dots, k-1\}$ the intersection $S_i\cap S_{i+1}$ is a $(r-1)$-dimensional face of both.
We claim that there is a unique equivalence class for this relation.
Assume there is more than one.
Let $C_1,\dots, C_k$ be the union of the $r$-simplices in such equivalence classes.
By the connectedness of log canonical centers~\cite[Theorem 1.6]{FS20},
the intersection $C_1\cap (C_2\cup\dots\cup C_k)$ must be non-empty.
Let $F$ be a maximal dimensional face of such set.
Then, $F$ is contained in some $r$-dimensional simplex $S$ of $C_1$.
In this simplex $S$, the face $F$ is intersection of the $(r-1)$-dimensional faces $\Gamma_1,\dots,\Gamma_s$.
We perform equivariant generalized adjunction to the center corresponding to $F$. 
We obtain two disconnected sets of generalized log canonical centers:
\begin{enumerate} 
\item the former corresponding to $\Gamma_1\cup\dots \cup\Gamma_s$, and 
\item the latter corresponding to the faces of $C_2\cup\dots\cup C_k$ containing $F$.
\end{enumerate}
This imply that $s=1$, which means that $F$ is a $(r-1)$-dimensional face, leading to a contradiction. 

Now, let $S$ be a $r$-dimensional simplex of $\mathcal{D}(X',B'+M')$ on which $G$ acts as the identity.
Let $S'$ be a $r$-dimensional simplex, in such dual complex, for which $S\cap S'$ is a $(r-1)$-fimensional face $F$ of both $S$ and $S'$.
Observe that $G$ acts on $F$ as the identity.
Furthermore, there exists a unique $r$-dimensional simplex in $\mathcal{D}(X',B'+M')$ which intersect $S$ along $F$.
Hence, $G$ must map $S'$ to itself.
In particular, $G$ must map vertices of $S'$ to vertices of $S'$.
Observe that there is a unique vertex of $S'$ which is not contained in $F$.
Thus, such vertex must be fixed by $G$.
Hence, $G$ must act as the identity on the set of vertices of $S'$, so it must act as the identity on $S'$.
By the previous paragraphs, we conclude that $G$ acts on $\mathcal{D}(X',B'+M')$ as the identity.
\end{proof}

\begin{lemma}\label{lem:equiv-P^1-link}
Let $G$ be a finite group.
Let $(X,B+M)$ be a $G$-equviariant generalized dlt log Calabi-Yau pair.
Assume that all the generalized log canonical centers of $(X,B+M)$ are invariant under the action of $G$.
If $G$ acts as the identity on one minimal generalized log canonical center, then it acts as the identity on all minimal generalized log canonical centers
of $(X,B+M)$.
\end{lemma}

\begin{proof}
Let $W$ be the minimal generalized log canonical center on which $G$ acts as the identity.
$W$ corresponds to a maximal dimensional simplex $S_W$ of the dual complex $\mathcal{D}(X,B+M)$ of dimension $r$.
Let $W_1$ be a minimal generalized log canonical center which corresponds to a $r$-dimensional simplex $S_{W_1}$ of $\mathcal{D}(X,B+M)$ so that
$S_W\cap S_{W_1}$ is $(r-1)$-dimensional.
We claim that $G$ acts as the identity on $W_1$.
Let $P$ be the generalized log canonical center of $(X,B+M)$ corresponding to the $(r-1)$-dimensional simplex $S_W\cap S_{W_1}$.
Then, $P$ contains both $W$ and $W_1$ as divisors.
Note that $P$ is $G$-invariant.
Let $(P,B_P+M_P)$ be the generalized dlt pair obtained by adjunction formula to $P$.
We denote by $G_P$ the quotient group of $G$ acting on $P$.
Observe that $W$ and $W_1$ are $G_P$-invariant divisorial generalized log canonical centers of $(P,B_P+M_P)$ and $(P,B_P+M_P)$ admits no other generalized log canonical center. Indeed, its regularity is zero.
We run a $G_P$-equivariant minimal model program for
\[
K_P+B_P -\epsilon W + M_P \sim_\qq -\epsilon W.
\]
This $G_P$-equivariant MMP terminates with a 
$G_P$-equivariant Mori fiber space $P_1\rightarrow Z$.
Note that $P\dashrightarrow P_1$ does not contract $W$ nor $W_1$.
Indeed, if $W$ was contracted, then in such model $W$ would be covered by contracted $W$-negative curves. However, all the curves contracted by this MMP are $W$-positive.
On the other hand, since all the curves contracted by this MMP are $W$-positive and $W_1\cap W$ is empty, we conclude that $W_1$ is not contracted.
By abuse of notation, we denote by $W$ and $W_1$ the push-forward of these divisors on $P_1$.
Note that $W$ is ample over $Z$.
Hence, $W_1$ is horizontal over $Z$.
Otherwise, the intersection $W\cap W_1$ on $P_1$ is non-trivial.
We claim that $Z$ has dimension $\dim W$ and both $W$ and $W_1$ dominate $Z$.
Indeed, assume that the dimension of $Z$ is less than the dimension of $W$.
Then, the general fiber $F$ of $P_1\rightarrow Z$ has dimension at least two. 
The restriction $W_F:=W|_F$ is an ample divisor. 
Since $W_1$ is horizontal over $Z$, then
$W_{F,1}:=W_1|_F$ would intersect $W_F$, leading to a contradiction.
We conclude that $\dim Z=\dim W=\dim W_1$ and the general fiber $F$ of $P_1\rightarrow Z$ is isomorphic to $\pp^1$.
Furthermore, we have that 
\[
(P_1, B_{P_1}+M_{P_1})|_F \simeq (\pp^1, \{0\} + \{\infty\}),
\]
where $B_{P_1}$ is the push-forward of $B_P$ to $P_1$.
This means that $P_1\rightarrow Z$ is an
equivariant $\pp^1$-link in the sense of Definition~\ref{def:p1-link}.
Since the action of $G_P$ on $W$ is trivial,
then the induced action on the base $Z$ is trivial as well.
Indeed, the induced morphism $W\rightarrow Z$ is $G_P$-equivariant.
On the other hand, the induced morphism
$W_1\rightarrow Z$ is $G_P$-equivariant. We conclude that the action of $G_P$ on $W_1$ is trivial. This finishes the proof of the claim.

As we have seen in the proof of Lemma~\ref{lem:fix-max-dim-cell}, given two minimal log canonical centers $W_1$ and $W_k$ corresponding to $r$-dimensional simplices $S_1$ and $S_k$ in $\mathcal{D}(X,B+M)$, we can find a sequence of $r$-simplices $S_2,\dots S_{k_1}$ so that $S_i\cap S_{i+1}$ is a face of dimension $r-1$.
Hence, by the previous paragraph, if $G$ acts as the identity on $W_1$, then it will act as the identity on $W_k$. 
Since there is a unique equivalence class for the relation defined in the previous paragraph,
we conclude that $G$ must act as the identity on
$\mathcal{D}(X,B+M)$.
\end{proof}

\begin{lemma}\label{lem:fixing-min-fix-right-cod}
Let $n$ be a positive integer.
There exists a constant $c(n)$, only depending on $n$, satisfying the following.
Let $G$ be a finite group.
Let $(X,B+M)$ be a $G$-equivariant generalized dlt pair of dimension $n$.
Assume that $G$ acts as the identity on a generalized dlt center $W$
of $(X,B+M)$.
Furthermore, assume that $G$ fixes all generalized log canonical centers passing through $W$.
Then, there exists:
\begin{enumerate} 
\item[(i)] a normal abelian subgroup $A\triangleleft G$ of rank $r$ and index $c(n)$,
\item[(ii)] an $A$-equivariant generalized dlt modification of $(X,B+M)$, and
\item[(iii)] a generalized log canonical center $W'$ of $(X',B'+M')$ of codimension $r$ on which $A$ acts as the identity.
\end{enumerate} 
\end{lemma}

\begin{proof}
By the Jordan property, we may replace $G$ with a normal abelian subgroup $A\triangleleft G$ of rank $r\leq n$ and index $c(n)$.
We may localize at the generic point of $W$ and assume we are working on a log smooth germ.
Furthermore, we can write
$A\simeq \zz_{d_1}\oplus \dots \oplus \zz_{d_r}$ with $d_i \mid d_{i+1}$ and $A<\mathbb{G}_m^r \leqslant {\rm Aut}(X',B';\eta_W)$.
We may assume that $\zz_{d_1}<\mathbb{T}_0:=\mathbb{G}_m\leqslant \mathbb{G}_m^r$.
We can perform a $\mathbb{T}_0$-equivariant blow-up of $Y\rightarrow (X',B';\eta_W)$ so that $\mathbb{T}_0$ acts as the identity on the exceptional $E$.
By performing further $\mathbb{T}_0$-equivariant blow-ups, we may assume that:
\begin{enumerate} 
\item $(Y,E)$ is $A$-invariant,
\item $E$ is a $\mathbb{T}_0$-fixed smooth toric variety, and
\item $\mathbb{T}_0$ acts as the identity on $E$.
\end{enumerate} 
Let $A_E$ be the quotient group of $A$ acting on $E$. Note that $A_E$ is an abelian group of rank $r-1$ and $E$ has dimension $n-1$. Then, the statement follows by induction on the dimension.
\end{proof}

\begin{lemma}\label{lem:fixing-one-dual-complex}
Let $G$ be a finite group.
Let $(X,B+M)$ be a $G$-equivariant generalized log canonical pair.
Let $(X',B'+M')$ be a $G$-equivariant generalized dlt modification of $(X,B+M)$.
Assume that $G$ acts as the identity on $\mathcal{D}(X',B'+M')$.
Let $(X'',B''+M'')$ be a $G$-equivariant generalized dlt modification of $(X,B+M)$.
Then, $G$ acts as the identity on
$\mathcal{D}(X'',B''+M'')$.
\end{lemma}

\begin{proof}
Assume that $G$ does not act as the identity on the generalized dlt modification $(X'',B''+M'')$.
By further blowing-up, we may assume there is a prime divisorial generalized log canonical center $E$ which is not $G$-invariant.
We may replace $X''$ by a $G$-equivariant resolution and assume that $X''\rightarrow X'$ is a $G$-equivariant morphism.
Let $W$ be the image of $E$ on $X'$.
Observe that the $G$-orbit of $E$ is mapped to $W$,
and we are assuming that the $G$-orbit of $E$ contains more than one prime divisor.
Then, $W$ is a $G$-invariant generalized log canonical center of $(X',B'+M')$.
Note that $E$ defines a generalized log canonical center on the log smooth germ
$(X',B';\eta_W)$.
Hence, $E$ can be extracted from $(X',B'+M')$ by a sequence of monoidal transformations centered at $G$-invariant log canonical centers.
Each divisor extracted in this sequence of monoidal transformations is $G$-invariant.
Hence, we conclude that $E$ itself must be $G$-invariant. This leads to a contradiction.
Thus, we conclude that $G$ must act as the identity on $\mathcal{D}(X'',B''+M'')$.
\end{proof}

\begin{lemma}\label{lem:qld-mod-dlt}
Let $X'\rightarrow X$ be a projective birational morphism. 
Let $(X,B+M)$ be a generalized dlt log Calabi-Yau pair.
Let $(X',B'+M')$ be its log pullback to $X'$.
Assume that $(X',B'+M')$ has generalized qdlt singularities.
Then, for every generalized log canonical center $W$ of $(X,B+M)$, there exists a generalized log canonical center $W'$ of $(X',B'+M')$ so that the induced morphism $W'\rightarrow W$ is birational.
\end{lemma}

\begin{proof}
By localizing at the generic point of $(W,B_W)$, we may assume that $W$ is a log smooth closed point.
Hence, we may assume that $\mathcal{D}(X,B+M)$ is a simplicial complex of dimension $r$.
Let $X'\rightarrow X$ be the aforementioned projective birational morphism.
By performing monoidal transformations along generalized log canonical centers of $(X',B'+M')$, we may assume that $(X',B'+M')$ is itself generalized dlt.
Then, by~\cite[Theorem 3]{dFKX17}, we conclude that $\mathcal{D}(X',B'+M')$ is simple-homotopy equivalent to $\mathcal{D}(X,B+M)$.
Hence, $\mathcal{D}(X',B'+M')$ is $r$-dimensional.
We conclude that there is a closed point in $X'$ which is a generalized log canonical center of $(X',B'+M')$ and maps onto $\eta_W$.
This finishes the proof.
\end{proof}

\section{Finite actions on dual complexes}

In this section, we study finite actions on dual complexes, which are induced by finite automorphisms of the log Calabi-Yau pair.
In this section, we will prove Theorem~\ref{introthm:almost-fixed-global} and
Theorem~\ref{introthm:almost-fixed-local}.
The following theorem is a generalization of the former.

\begin{theorem}\label{thm:trivial-action-on-D}
Let $n$ be a positive integer. 
There exists a constant $c(n)$, only depending on $n$, satisfying the following.
Let $(X,B+M)$ be a generalized log Calabi--Yau pair of dimension $n$.
Let $G<{\rm Aut}(X,B+M)$ be a finite subgroup.
Then, there exists a normal subgroup $A\triangleleft G$ of index at most $c(n)$ that acts trivially on $\mathcal{D}(X,B+M)$.
\end{theorem}

\begin{proof}
We proceed by induction on $n$.
If the dimension equals one, then the dual complex $\mathcal{D}(X,B+M)$ is either empty, a point, or two points.
In any case, either $G$ or a normal subgroup of index $2$ of $G$ will act trivially on the dual complex $\mathcal{D}(X,B+M)$.
Thus, we may assume that $\mathcal{D}(X,B+M)$ has dimension at least two.

Let $(X_0,B_0+M_0)$ be a
$G$-equivariant canonicalization of a
$G$-equivariant $G\qq$-factorial dlt modification of $(X,B+M)$.
In particular, $X_0$ has
$G$-equivariant 
canonical
$G\qq$-factorial singularities.
We run a $G$-equivariant 
minimal model program for $K_{X_0}$.
If $K_{X_0}$ is pseudo-effective, then $B_0=M_0=0$.
In particular, we have that the generalized pair $(X,B+M)$ has empty dual complex
$\mathcal{D}(X,B+M)$.
Thus, we may assume that the 
minimal model program terminates with a $G$-equivariant Mori fiber space $\phi_1 \colon X_1\rightarrow Z_1$.
By Lemma~\ref{lem:contr-minimal-rel-dim},
we can find a birational model $(X_2,B_2+M_2)$ of $(X_1,B_1+M_1)$ that admits a $G$-equivariant contraction $\phi_2\colon X_2\rightarrow Z_2$, with minimal relative dimension, so that $X_2$ is canonical and $-K_{X_2}$ is ample over $Z_2$.
We have a commutative diagram as follows:
\[
\label{com-diag}
\xymatrix{
(X_0,B_0+M_0) \ar[d] \ar@{-->}[r] &
(X_1,B_1+M_1) \ar[d]^-{\phi_1} \ar@{-->}[r] & 
(X_2,B_2+M_2) \ar[d]^-{\phi_2}\\
(X,B+M) & 
(Z_1,B_{Z_1}+M_{Z_2}) &
(Z_2,B_{Z_2}+M_{Z_2})\ar[l]
}
\]
where $(Z_2,B_{Z_2}+M_{Z_2})$ is the generalized log canonical pair obtained by the equivariant generalized canonical bundle formula.
We have a short exact sequence
\[
1\rightarrow G_F\rightarrow G\rightarrow G_Z\rightarrow 1,
\]
where $G_F$ acts fiberwise 
and $G_Z$ acts on the base $Z_2$.
With this notation, the generalized log canonical pair
$(Z_2,B_{Z_2}+M_{Z_2})$ is $G_Z$-equivariant.
In what follows, we will say that a generalized log canonical center of $(X_2,B_2+M_2)$ is horizontal 
if it does dominate $Z_2$.
Otherwise, we will say that the generalized log canonical center is vertical.

First, we assume that $\mathcal{D}(F,B_F+M_F)$ is non-trivial, where $F$ is the general fiber of $\phi_2$.
Note that $F$ belongs to a bounded family. 
Indeed, since the relative $G$-Picard rank of $X_2$ over $Z_2$ is one, we conclude that $-K_{X_2}$ is ample over $Z_2$.
In particular, $-K_X$ is an ample divisor.
Therefore, it belongs to a bounded family by~\cite[Theorem 1.1]{Bir21}.
By Lemma~\ref{lem:fix-lcc}, there exists a constant $c(f)$, only depending on $f=\dim F$, satisfying the following.
For every $i \in \{0,\dots, f-1\}$, the generalized log canonical pair $(F,B_F+M_F)$ has at most $c(f)$ 
$i$-dimensional generalized log canonical centers.
Hence, we conclude that
for every $i\in \{0,\dots, f-1\}$, the generalized log canonical pair $(X_2,B_2+M_2)$ has at most 
$c(f)$ $i$-dimensional generalized log canonical centers that are horizontal over $Z_2$.
Note that every element of $G$ must map $i$-dimensional horizontal generalized log canonical centers of $(X_2,B_2+M_2)$ to $i$-dimensional horizontal generalized log canonical centers of $(X_2,B_2+M_2)$.
Hence, for each $i\in \{0,\dots, f-1\}$, we have a homomorphism
$G\rightarrow S_{j(i)}$ where $j(i)\leq c(f)$.
We conclude that there exists a constant $k(f)$, only depending on $f$, and a normal subgroup $H\triangleleft G$, so that $H$ fixes every horizontal generalized log canonical center of $(X_2,B_2+M_2)$. Replacing $G$ with $H$, we may assume that $G$ fixes every horizontal generalized log canonical center of $(X_2,B_2+M_2)$.
Let $W$ be such a minimal horizontal generalized log canonical center.
By Lemma~\ref{lem:existence-g-inv-lcp}, we know there exists a $G$-invariant log canonical place $E$ of $(X_2,B_2+M_2)$ which maps onto $W$.
Hence, we can replace $(X_2,B_2+M_2)$ with a higher $G\qq$-factorial $G$-equivariant dlt modification 
$(X'_2,B'_2+M'_2)$ so that the center
of $E$ on $X'_2$ is a divisor.
 $E$ is a $G$-invariant divisor which appears with coefficient one in $B'_2$. By Lemma~\ref{lem:G-equiv-adj}, we can perform adjunction to it
and obtain a generalized dlt pair
\[
(K_{X'_2}+B'_2+M'_2)|_E \sim_\qq 
K_E+B_E+M_E.
\]
The generalized log Calabi-Yau pair $(E,B_E+M_E)$ admits the action of a cyclic quotient of $G$ (see Lemma~\ref{lem:G-equiv-adj}).
We denote such a cyclic quotient by $G_E$.
By induction on the dimension, there exists a normal subgroup
$A_E\triangleleft G_E$ of index at most
$c(n-1)$ so that $A_E$ fixes
$\mathcal{D}(E,B_E+M_E)$ pointwise.
In particular, $A_E$ fixes all generalized log canonical centers
of $(E,B_E+M_E)$.
We denote by $A$ the preimage of $A_E$ in $G$.
Note that $A\triangleleft G$ is a normal subgroup of index at most $c(n-1)$.
By Lemma~\ref{lem:triv-act-in-out}, we conclude that $A$ fixes all the generalized log canonical centers of 
$(X_2',B'_2+M'_2)$ which intersect $E$.
In particular, $A$ must fix a minimal log canonical center $W$ of $(X_2',B'_2+M'_2)$ and all the generalized log canonical centers containing $W$.
By Lemma~\ref{lem:fix-max-dim-cell}, we conclude that $A$ fixes every generalized log canonical center of
$(X_2',B'_2+M'_2)$.
In particular, the action of $A$ on
$\mathcal{D}(X_2',B'_2+M'_2)$ is trivial.
By Lemma~\ref{lem:fixing-one-dual-complex}, we conclude that $A$ acts trivially on $\mathcal{D}(X_0,B_0+M_0)$.
This finishes the proof in the case that
$\mathcal{D}(F,B_F+M_F)$ is non-trivial.

Now, we assume that $\mathcal{D}(Z_2,B_{Z_2}+M_{Z_2})$ is non-trivial.
By induction, we may assume that a normal subgroup $A_Z\triangleleft G_Z$
acts on $\mathcal{D}(Z_2,M_{Z_2}+B_{Z_2})$ trivially.
The subgroup $A_Z\triangleleft G_Z$ can be chosen to have index bounded above by $c(z)$, where $z=\dim Z_2 $.
Replacing $G$ with the preimage of $A_Z$ on $G$, we may assume that such property is satisfied by $G_Z$ itself.
In particular, $G_Z$ must fix a divisorial log canonical place of $(Z_2,B_{Z_2}+M_{Z_2})$.
By Lemma~\ref{lem:compatible-dlt-mod}, we can find compatible $G$-equivariant dlt modifications
$(X_2',B_2'+M_2')\rightarrow (Z_2',B_{Z_2'}+M_{Z_2'})$.
Let $E_{Z'_2}$ be a prime divisor with coefficient one in $B_{Z'}$ 
that is fixed by $G_Z$.
Let $E'$ be the union of all the prime divisors with coefficient one in $B_2'$ that maps onto $E_{Z'}$.
By Lemma~\ref{lem:compatible-dlt-mod},
We can assume that every prime divisor of $X'_2$ that dominates $E_{Z'}$ appears in $B_2'$ with coefficient one. 
This means that $E'$ is the union of all the prime divisors on $X'$ which maps onto $E_{Z'}$.
Note that $E'$ is $G$-invariant.
In particular, the log pair
$(X_2',E')$ is a dlt $G$-equivariant pair.
We run a $G$-equivariant $(K_{X_2'}+E')$-MMP over $Z'_2$.
Note that $E'$ contains no horizontal divisors over $Z_2'$.
Hence, the divisor $K_{X_2'}+E'$ is not pseudo-effective over the base $Z'_2$.
Thus, the minimal model program terminates with a $G$-equivariant Mori fiber space.
By the minimality assumption of the fibration, this $G$-equivariant Mori fiber space must map to a higher birational model of $Z'_2$.
We denote by $X_2''\rightarrow Z_2''$ such $G$-equivariant Mori fiber space.
We denote by $E''$ the push-forward of $E'$ to $X_2''$.
We claim that $E''$ is a non-trivial effective divisor.
Indeed, since $E'$ is the reduced sum of all prime divisors on $X'_2$ mapping onto $E_{Z'}$, then at least one of the components of $E'$ is not contracted by the minimal model program $X_2'\dashrightarrow X_2''$.
Note that $E''$ is a $G$-invariant divisor
as is the push-forward of an $G$-invariant divisor via a $G$-invariant birational map.
In summary, the log pair
$(X_2'',E'')$ is dlt, 
$E''$ is a non-trivial effective divisor,
$E''$ is $G$-equivariant,
and $K_{X''}+E''$ is anti-ample over $Z_2''$.
Note that $Z_2''\rightarrow Z_2'$ is an isomorphism over the generic point of $E_{Z_2'}$.
By Lemma~\ref{lem:dlt-antiample pair}, we conclude that
$\mathcal{D}(X''_2,E'')$ is a simplex of dimension at most $n$.
Hence, we can replace $G$ with a subgroup of index at most $n!$ so that it fixes one of the prime divisors in the support of $E''$.
After such replacement, $G$ must fix a divisor on the support of $\lfloor B'_2 \rfloor$.
Denote by $E$ the $G$-fixed prime divisor. 
Let $G_E$ be the quotient group of $G$ acting on $E$.
Then, we have an action of $G_E$ on the generalized dlt pair
$(E,B_E+M_E)$ obtained by $G$-equivariant 
generalized adjunction (see Lemma~\ref{lem:G-equiv-adj}).
By induction on the dimension, there exists a constant $c(n-1)$, only depending on $n-1$, and a normal subgroup $A_E$ of $G_E$ of index at most $c(n-1)$ that acts as the identity on the dual complex
$\mathcal{D}(E,B_E+M_E)$.
Let $A$ be the preimage of $A_E$ in $G$.
By Lemma~\ref{lem:triv-act-in-out}, we conclude that $A$ fixes all the generalized log canonical centers of 
$(X_2',B'_2+M'_2)$ which intersect $E$.
By Lemma~\ref{lem:triv-act-in-out}, we conclude that $A$ fixes the CW complex  $\mathcal{D}(X'_2,B'_2+M'_2)$ pointwise.
By Lemma~\ref{lem:fixing-one-dual-complex}, we conclude that $A$ acts on $\mathcal{D}(X,B+M)$ as the identity.
This finishes the proof in the case that $\mathcal{D}(Z_2,B_{Z_2}+M_{Z_2})$ is non-trivial.

By Lemma~\ref{lem:reg-fib-or-base-non-trivial},
provided that
$\mathcal{D}(X_2,B_2+M_2)$ is non-trivial, 
then either
$\mathcal{D}(Z_2,B_{Z_2}+M_{Z_2})$
or 
$\mathcal{D}(F,B_F+M_F)$ are
non-trivial.
Hence, in any case, we can find 
a normal subgroup $A\triangleleft G$ of index at most $c(n)$, so that $A$ acts as the identity on 
$\mathcal{D}(X,B+M)$.
\end{proof}

Now, we proceed to prove the local version of the previous statement. 
In the local case, we need the log canonical singularity to support a klt singularity.
The following theorem is a generalization of Theorem~\ref{introthm:almost-fixed-local}
to the category of generalized pairs.

\begin{theorem}\label{thm:almost-fixed-D-local}
Let $n$ be a positive integer.
There exists a constant $c(n)$, only depending on $n$, satisfying the following.
Let $(X,B+M;x)$ be a generalized log canonical singularity of dimension $n$.
Assume that $(X,B_1+M;x)$
is gklt for some $B_1\leq B$.
Let $G\leqslant {\rm Aut}(X,B+M;x)$ be a finite subgroup.
Then, there exists a normal subgroup $A\triangleleft G$ of index at most $c(n)$ 
so that $A$ acts trivially on $\mathcal{D}(X,B+M;x)$.
\end{theorem}

\begin{proof}
We can find a $G$-invariant $\qq$-Cartier divisor $H$ that passes through $x$, so that
$(X,B+H+M;x)$ is generalized log canonical and $x$ is a generalized log canonical center.
It suffices to prove that there exists a normal subgroup $A\triangleleft G$ of index at most $c(n)$ so that $A$ acts trivially on 
$\mathcal{D}(X,B+H+M;x)$.
Indeed, the CW complex 
$\mathcal{D}(X,B+M;x)$ embeds in 
$\mathcal{D}(X,B+H+M;x)$
equivariantly.

By Lemma~\ref{lem:gen-inv-plt}, we can find a generalized $G$-equivariant plt blow-up $\phi\colon Y\rightarrow X$ for
\[
(X,(1-\epsilon)(B+H)+\epsilon B_1 +M).
\]
Furthermore, we may assume that $E$ is a generalized log canonical center of $(X,B+H+M;x)$ mapping onto $E$.
The divisor $E$ is the unique exceptional divisor of $\phi$ and is $G$-invariant.
Let $(Y,B_Y+H_Y+M_Y)$ be the generalized pair obtained by log pullback of $(X,B+H+M)$ to $Y$.
We denote by 
$(E,B_E+M_E)$ the log Calabi-Yau generalized log canonical pair obtained by generalized adjunction of $(Y,B_Y+H_Y+M_Y)$ to the exceptional divisor $E$. Let $G_E$ be the quotient group of $G$ acting on $E$.
By Theorem~\ref{thm:trivial-action-on-D}, there exists a normal subgroup $A_E\triangleleft G_E$ which acts trivially on $\mathcal{D}(E,B_E+M_E)$ and its index on $G_E$ is at most $c(n-1)$, where $c(n-1)$ is a constant only depending on $n-1$.
We denote by $A$ the preimage of $A_E$ on $G$.

We claim that $A$ is a normal subgroup of $G$ which acts trivially on $\mathcal{D}(X,B+H+M;x)$.
Let $(X',B'+H'+M')$ be a $G\qq$-factorial dlt modification of $(X,B+H+M)$ which dominates $Y$.
Let $E'$ be the strict transform of $E$ on $X'$.
Let $(E',B'+M')$ be the generalized log Calabi-Yau gdlt pair obtained by adjunction of $(X',B'+H'+M')$ to the exceptional divisor $E'$.
Then, the group $A$ fixes $E'$.
Moreover the quotient group $A_E$ acting on $E'$ fixes every log canonical center of $(E',B'+M')$.
By Lemma~\ref{lem:triv-act-in-out}, we conclude that $A$ fixes every log canonical center of $(X',B'+H'+M')$ which intersects $E'$ non-trivially.
In particular, the action of 
$A$ on $\mathcal{D}(X',B'+H'+M')$ must fix a maximal dimensional simplex of the CW complex.
By Lemma~\ref{lem:fix-max-dim-cell}, we conclude that $A$ acts as the identity on 
$\mathcal{D}(X,B+H+M;x)$ and hence on $\mathcal{D}(X,B+M;x)$.
\end{proof}

Now, we are ready to prove that in a fixed dimension, the possible fundamental groups of dual complexes of log Calabi-Yau pairs can only attain finitely many isomorphism classes.

\begin{proof}[Proof of Corollary~\ref{introcor-finiteness-pi1-dual}]
The statement is clear if the dual complex is one-dimensional, since in this case, it is either a segment of a line or a circle.
We assume that the dual complex has dimension at least two.

Due to the Lefschetz principle, we
may assume that we are working over the field of complex numbers
(see, e.g.~\cite[Proposition 2.2]{NXY19}).
Let $(X,B)$ be a complex log Calabai-Yau pair of dimension $n$.
By~\cite[Corollary 58]{KX16}, we can find a projective variety $Y$ with a boundary $B_Y$, satisfying the following conditions:
\begin{enumerate}
\item The variety $Y$ is rationally connected, 
\item the pair $(Y,B_Y)$ is dlt,
\item the boundary $B_Y$ fully supports a nef and big divisor, 
\item $K_Y+B_Y\sim_\qq 0$, and 
\item $\mathcal{D}(X,B)$ is PL-homeomorphic to $\mathcal{D}(Y,B_Y)$.
\end{enumerate}
Let $\Delta_Y$ be the big and nef divisor which is fully supported on $\supp(B_Y)$.
Then, the log pair
$(Y,B_Y-\epsilon \Delta_Y)$ has klt singularities.
Furthermore, the divisor
\[
-(K_Y+B_Y-\epsilon \Delta_Y) \sim_\qq \epsilon \Delta_Y
\]
is a big and nef divisor for $\epsilon>0$ small enough.
Hence, $Y$ is a complex Fano type variety.
By~\cite[Theorem 2.(3)]{KX16}, we have a surjective homomorphism
$\pi_1^{\rm reg}(Y)\rightarrow \pi_1(\mathcal{D}(Y,B_Y))$.
By~\cite[Theorem 2]{Bra20}, we know that $\pi_1^{\rm reg}(Y)$ is finite, hence
$\pi_1(\mathcal{D}(Y,B_Y))$ is finite as well.
In particular, $\pi_1(\mathcal{D}(X,B))$ is finite so is isomorphic to its profinite completion.
By~\cite[(5)]{KX16}, the universal cover
\begin{equation} 
\label{eq:uni-cover}
\widetilde{\mathcal{D}}(X,B)
\rightarrow 
\mathcal{D}(X,B)
\end{equation} 
of $\mathcal{D}(X,B)$ is induced by a finite quasi-\'etale cover $(\widetilde{X},\widetilde{B}) \rightarrow (X,B)$, which we may assume Galois.
Let $G$ be the group acting on $(\widetilde{X},\widetilde{B})$.
By Theorem~\ref{thm:trivial-action-on-D}, there is a normal
subgroup $A\triangleleft G$, so that $A$ acts as the identity on $\widetilde{\mathcal{D}}(X,B)$.
Moreover, the index of $A$ in $G$ is bounded above by $c(n)$. We conclude that 
$\mathcal{D}(X,B)$ is the quotient of $\widetilde{D}(X,B)$ by a group $G/A$ of order at most $c(n)$, i.e., the universal cover~\eqref{eq:uni-cover} has degree at most $c(n)$.
\end{proof}

To conclude this section,
we prove that in a finite group action on a log Calabi-Yau 
pair with arbitrary singularities,
almost every element has a fixed point 
on the dual complex.

\begin{proof}[Proof of Corollary~\ref{introcor:almost-fixed-point-not-lc}]
If the log pair $(X,B)$ is log canonical,
then the statement follows from Theorem~\ref{introthm:almost-fixed-global}.
Hence, we may assume that $(X,B)$ does not have log canonical singularities, i.e., there is at least one non-lc place.
By Theorem~\cite[Theorem 1.7]{FS20}, we conclude that 
$\mathcal{D}(X,B)$ is a collapsible space.
Then, the Lefschetz fixed-point theorem implies that 
every element of $G$ acts with a fixed point.
\end{proof}

\section{Almost fixed points in fibers}

In this section, we prove a theorem regarding the existence of almost fixed points in the fibers of $G$-equivariant morphisms.
The aim of this section is to prove Theorem~\ref{thm:almost-fixed-points}.
This theorem generalizes Theorem~\ref{introthn:almost-fixed-point-morphism} 
to compositions of birational and Fano type morphisms.

\begin{theorem}\label{thm:almost-fixed-points}
Let $n$ be a positive integer.
There exists a constant $c(n)$, only depending only on $n$,
satisfying the following.
Let $X$ be a $G$-equivariant $n$-dimensional variety.
Let $G$ be a finite group.
Let $\phi\colon X \rightarrow Z$ be a $G$-equivariant morphism which is composition of:
\begin{enumerate}
    \item Equivariant Fano type morphisms, and
    \item equivariant birational morphisms.
\end{enumerate}
Let $G_Z$ be the quotient group of $G$ acting on the base.
Let $z\in Z$ be a $G_Z$-fixed point.
Then, there exists a normal subgroup $A\triangleleft G$ so that the fiber over $z$ contains an $A$-fixed point.
\end{theorem}

\begin{proof}
The composition of equivariant birational maps is again an equivariant birational map.
Hence,
we may assume that $\phi\colon X\rightarrow Z$ decomposes as
\[
\label{decomp}
\xymatrix{
X=X_0 \ar[r]^-{\phi_1} & 
Y_0 \ar[r]^-{\psi_1} & 
X_1 \ar[r]^-{\phi_2} &
Y_1 \ar[r]^-{\psi_2} & 
\dots \ar[r]^-{\psi_{j-1}} &
X_j \ar[r]^-{\phi_j} &
Y_j \ar[r]^-{\psi_j} & 
Z,
}
\] 
with $j\leq n$,
where the following conditions are satisfied:
\begin{enumerate}
    \item each $\phi_i$ is an equivariant birational morphism, and
    \item each $\psi_i$ is an equivariant Fano type morphism with positive dimensional general fiber.
\end{enumerate}
Since the source and image of a Fano type morphism have klt type singularities, then it suffices to prove the statement for each $\phi_i$ and $\psi_i$, individually.
We reduce the statement to an equivariant Fano type morphism with positive dimensional fiber  
or an equivariant birational morphism so that the base has klt type singularities.

First, we prove the case of a birational morphism.
Let $\phi\colon X\rightarrow Z$ be a $G$-equivariant birational morphism.
Let $G_Z$ be the quotient group acting on the base.
Let $z\in Z$ be a closed $G_Z$-fixed point. 
There exists a $G_Z$-invariant boundary $B_Z$ with $(Z,B_Z)$ klt at $z$.
Indeed, if the germ $(Z;z)$ has klt type singularities, 
then it has $G$-equivariant klt type singularities.
By Lemma~\ref{lem:gen-inv-plt}, we can find a $G$-equivariant plt blow-up 
$Y\rightarrow Z$ for $(Z,B_Z;z)$ so that the exceptional $E$ is a $G$-invariant Fano type variety.
Since Fano type varieties are rationally connected, we have that $E$ is a $G$-invariant rationally connected divisor.
We can find a $G$-equivariant projective birational morphism $Y'\rightarrow Z$ so that $Y'$ dominates $X$ with a $G$-equivariant projective birational morphism.
Let $E'$ be the strict transform of $E$ on $Y'$.
Since the morphism is $G$-equivariant, then $E'$ is a $G$-invariant rationally connected divisor on $Y'$.
Let $W$ be the image of $E'$ on $X$.
Then, $W$ is a $G$-invariant subvariety of $X$ and $W$ is a rationally connected variety, being the image of the rationally connected variety $E'$.
Note that $W$ has dimension at most $n-1$.
Let $G_W$ the quotient group of $G$ acting on $W$, i.e., 
the kernel of the homomorphism
$G\rightarrow G_W$ is the largest normal subgroup acting as the identity on $W$.
By~\cite[Theorem 4.2]{PS16}, there exists a normal subgroup
$A_W \triangleleft G_W$ of index bounded above by $c(n-1)$, so that $A_W$ has a fixed point $w$ on $W$.
Hence, the normal subgroup $A\triangleleft G$, which is the preimage of $A_W$ on $G$, has a fixed point $w$ on $X$.
By construction, the image of $w$ on $Z$ is the closed point $z$,
so $w$ is a closed point in the fiber $\phi^{-1}(z)$ which is $A$-fixed. 

Now, we prove the case of a Fano type morphism.
Let $\phi\colon X\rightarrow Z$ be a $G$-equivariant Fano type morphism with positive dimensional general fiber.
Let $G_Z$ be the quotient group acting on the base $Z$.
Let $z\in Z$ be a closed $G_Z$-fixed point.
By the Fano type condition, we can find a boundary $B\geq 0$ on $X$ so that
$-(K_X+B)$ is big and nef over $Z$.
By taking the normalized $G$-closure of $B$ on $X$, we may assume that $(X,B)$ is a $G$-equivariant pair.
Recall that in a relative Mori fiber space, a relatively nef divisor is relatively semiample.
Let $\Gamma$ be a $G$-equivariant  klt $\qq$-complement of $(X,B)$ over $Z$, i.e., the log pair
$(X,B+\Gamma)$ is $G$-equivariant, has klt singularities, and 
$K_X+B+\Gamma$ is $\qq$-trivial over the base. 

We define 
\[
C_X := {\rm Spec}_Z 
\left( 
\bigoplus_{m\geq 0}
\phi_* \mathcal{O}_X(-m(K_X+B))
\right). 
\]
On the other hand,
we define $\widetilde{C}_X$ to be the relative spectrum of
\[
\bigoplus_{m\geq 0} \mathcal{O}_X(-m(K_X+B)) 
\]
over $X$.
Let $E$ be the unique prime divisor contracted by
$\widetilde{C}_X\rightarrow C_X$ and $V$ its image on $C_X$.
We have a $G$-equivariant commutative diagram:
\[
\begin{tikzcd}
X \ar[d, "\phi"] \ar[r,"\sim"] & E \ar[d,"\phi'"] \ar[r, hook] & \widetilde{C}_X \ar[d,"\phi_C"] \\
Z \ar[r,"\sim"] & V \ar[r,hook] & C_X.
\end{tikzcd}
\]
Let $v\in V$ be the image of $z$ via the composition $Z\hookrightarrow C_X$ of the bottom horizontal morphisms.
Note that $v$ is a $G$-fixed point on $C_X$.
It suffices to prove that the $G$-equivariant birational map $\phi_C\colon \widetilde{C}_X\rightarrow C_X$
contains an almost fixed point in the fiber 
$\phi_C^{-1}(v) \simeq \phi^{-1}(z)$.
In order to do so, we need to prove that $C_X$ supports a klt singularity around $v$.

We define the divisors $B_{C_X}$ and $\Gamma_{C_X}$ 
to be the cones over the divisors $B$ and $\Gamma$ on $C_X$.
Analogously, 
we define the divisors $B_{\widetilde{C}_X}$ and $\Gamma_{\widetilde{C}_X}$ to be the pullbacks of
$B$ and $\Gamma$ to $\widetilde{C}_X$.
Note that the action of $G$ lifts to $C_X$ and $\widetilde{C}_X$ as $(X,B)$ is a $G$-equivariant log pair.
We have a $G$-equivariant projective birational morphism $\widetilde{C}_X \rightarrow C_X$.
Let $E$ be the unique prime divisor contracted by $\widetilde{C}_X\rightarrow C_X$ and $V$ its image on $C_X$.
Then, $E$ is a generalized log canonical center of
$(C_X,B_{C_X}+\Gamma_{C_X})$.
Indeed, the log pair 
\[
(\widetilde{C}_X, 
B_{\widetilde{C}_X}+\Gamma_{\widetilde{C}_X}+E)
\]
is log canonical and log Calabi-Yau over $C_X$.
Thus, $(C_X;v)$ is a germ of lc type.
Analogously, we have that
\[
(\widetilde{C}_X,
B_{\widetilde{C}_X}+\alpha E)
\]
is sub-klt and log Calabi-Yau over $C_X$ for certain $\alpha<1$.
Thus, $(C_X;B_{C_X};v)$ is klt, 
hence $(C_X;v)$ is a germ of klt type.
Thus, the existence of this almost fixed point in
the fiber $\phi^{-1}(z)$ follows from the birational case previously proved.
This finishes the proof of the theorem.
\end{proof}

\section{Rank and regularity} 
\label{sec:rank-reg}

In this section, we study the relation between the rank and the regularity of Fano type varieties and klt singularities.
First, we prove our second main theorem for Fano type varieties.
We show that a Fano type variety with the action of a large finite group of rank $r$ must have regularity at least $r-1$.
Moreover, in a suitable modification, we can find a log canonical center of codimension $r$ on which such finite group acts as the identity.

\begin{theorem}\label{thm:rank-vs-reg}
Let $n$ and $N$ be two positive integers.
There exists a constant $c:=c(n,N)$, only depending on $n$ and $N$, satifying the following.
Let $X$ be a $n$-dimensional Fano type variety.
Let $G \leqslant {\rm Aut}(X)$ be a finite subgroup with $\zz_m^r \leqslant G$.
Let $(X,B+M)$ be a $G$-equivariant $N$-complement.
If $m\geq c(n,N)$, then there exists:
\begin{enumerate}
    \item[(i)] A normal abelian subgroup $A\triangleleft G$ of index at most $c$, and 
    \item[(ii)] an $A$-equivariant generalized dlt modification $(Y,B_Y+M_Y)$ of $(X,B+M)$.
\end{enumerate}
Furthermore,
the $A$-equivariant modification
satisfies the following properties:
\begin{enumerate} 
    \item There is a generalized log canonical center $W\subset Y$ of $(Y,B_Y+M_Y)$ of codimension $r$ which is fixed pointwise by $A$, 
    \item $A$ fixes all the divisorial log canonical centers $E_1,\dots,E_r$
    of $(Y,B_Y+M_Y)$ containing $W$, and 
    \item there is a monomorphism  $A<\mathbb{G}_m^r \leqslant {\rm Aut}(Y,E_1+\dots+E_r;\eta_W)$.
\end{enumerate}
In particular, $A$ acts as the identity on all the minimal log canonical centers of $(Y,B_Y+M_Y)$.
\end{theorem}

Before we proceed to the proof,
we note that $G$ is almost abelian, with respect to $\dim X$, 
due to~\cite[Proposition 3.1]{BFMS20}.
Hence, the condition $\zz_m^r\leqslant G$ can be imposed in the normal abelian subgroup of $G$.

\begin{proof}
We will prove the theorem in several steps.\\

{\it Step 1:} We will introduce a normal abelian subgroup of $G$ and make some reductions.\\

By~\cite[Proposition 3.1]{BFMS20}, we can find a normal abelian subgroup $A_0$ of $G$ of index bounded by $c(n)$
of rank at least $r$.
By Theorem~\ref{thm:almost-fixed-D-local}, we can find a normal abelian subgroup $A \triangleleft A_0$, so that every generalized log canonical center of $(X,B+M)$ is invariant under the action of $A$.
By Theorem~\ref{thm:almost-fixed-D-local}, we can further assume that every generalized log canonical center of any $A$-equivariant dlt modification of $(X,B+M)$ is invariant under the action of $A$.
Due to Lemma~\ref{lem:fixing-min-fix-right-cod}, it suffices to find an $A$-equivariant generalized dlt modification of $(X,B+M)$ and a generalized lc center which is fixed pointwise under the action of $A$.
If $A$ acts as the identity on a generalized dlt place, then it must act as the identity on all generalized minimal log canonical centers due to Lemma~\ref{lem:equiv-P^1-link}.
Hence, our aim is to prove the existence of an $A$-invariant minimal generalized dlt center which is fixed pointwise by $A$.\\

{\it Step 2:} In this step, we prove that the statement for log pairs of dimension $n$, i.e., when the b-divisor $M$ is trivial, implies the statement for generalized pairs of dimension $n$.\\

Let $(X,B+M)$ be an $A$-equivariant log canonical pair as in the statement.
Let $(Y,B_Y+M_Y)$ be a
$A\qq$-factorial
$A$-equivariant dlt modification.
We know that $A$ fixes all the generalized log canonical centers of $(Y,B_Y+M_Y)$.
By Proposition~\ref{lem:extraction-FT}, we know that $Y$ is a Fano type variety.
In particular, $Y$ is a Mori dream space.
We claim that the log pair $(Y,B_Y)$ is $A$-equivariantly $\qq$-complemented.
Indeed, the divisor $M_Y$ is $A$-equivariant, pseudo-effective, and its diminished base locus is small.
We run an $A$-equivariant $M_Y$-minimal model program which terminates with a good minimal model.
Let $Y\dashrightarrow Y'$ be such
$A$-equivariant minimal model program.
Let $M_{Y'}$ be push-forward of $M_Y$.
Then, the generalized pair
$(Y',B_{Y'}+M_{Y'})$ is $A$-equivariant generalized log canonical and $M_{Y'}$ is semiample.
Hence, we can find an $A$-invariant efective divisor $0\leq \Gamma_{Y'}\sim_\qq M_{Y'}$ so that the log pair
$(Y',B_{Y'}+\Gamma_{Y'})$ is $A$-equivariant and log canonical.
Hence, its push-forward $(Y,B_Y+\Gamma_Y)$ to $Y$ is $A$-equivariant and log canonical.
Thus, the log pair $(Y,B_Y)$ is $A$-equivariantly $\qq$-complemented.
By Lemma~\ref{lem:Q-com-N-com}, up to replacing $N$ with a multiple that only depends on the dimension, we may assume that $(Y,B_Y)$ admits an $A$-equivariant $N$-complement.
Let $(Y,B'_Y)$ be such $A$-equivariant $N$-complement.
This means that $(Y,B'_Y)$ has log canonical singularities, is $A$-invariant, and $N(K_Y+B'_Y)\sim 0$.
Note that every log canonical center of $(Y,B_Y+M_Y)$ is a log canonical center of $(Y,B'_Y)$.
Indeed, the log canonical centers of $(Y,B_Y+M_Y)$ are given by strata of $\lfloor B_Y\rfloor$ and 
$B'_Y\geq B_Y$.

Observe that $(Y,B'_Y)$ is an $A$-equivariant $N$-complement of dimension $n$.
Hence, the statement of the theorem holds for the log pair $(Y,B'_Y)$.
Let $W\subset Y$ be a minimal generalized log canonical center of $(Y,B_Y+M_Y)$.
We claim that a normal subgroup of $A$, of index bounded by a constant only depending on $n$ and $N$, 
acts on $W$ as the identity.
Let $A_W$ be the quotient subgroup of $A$ acting on $W$.
If $W$ is a minimal log canonical center of $(Y,B'_Y)$ then the statement follows.
Indeed, in such a case, a subgroup of $A$ of index bounded by a constant in $n$ and $N$ acts as the identity on $W$.
Hence, we may assume that $W$ is not a minimal log canonical center of $(Y,B'_Y)$.
Since the pair $(Y,B_Y+M_Y)$ is generalized dlt, then by the adjunction formula, we can write 
\[
(K_Y+B_Y+M_Y)|_W \sim_\qq K_W+B_W+M_W.
\]
The generalized pair $(W,B_W+M_W)$ is an $A_W$-equivariant $N$-complement of the variety $W$.
Note that the generalized pair
$(W,B_W+M_W)$ is generalized klt, being $W$ a minimal generalized log canonical center of the generalized pair $(Y,B_Y+M_Y)$.
Furthermore, the log pair defined by adjunction
\[
(K_Y+B'_Y)|_W \sim_\qq K_W+B'_W
\]
is an $A_W$-equivariant generalized $N$-complement.
However, $(W,B'_W)$ is not generalized klt, as $W$ is not a minimal generalized log canonical center of $(Y,B_Y)$.
By Proposition~\ref{prop:MRC-gklt},
we can find a $A_W$-equivariant birational contraction 
$W\dashrightarrow W'$ and an $A_W$-equivariant projective contraction $W'\rightarrow Z$ so that either: 
\begin{enumerate} 
\item $W'$ is rationally connected, or
\item $Z$ is positive dimensional, has klt singularities, $K_Z\sim_\qq 0$, and general fibers of $W'\rightarrow Z$ are rationally connected.
\end{enumerate}
Let $(W',B_{W'}+M_{W'})$ be the generalized pair obtained by push-forward of $(W,B_W+M_W)$ to $W'$.
If $(1)$ holds, then the generalized pair $(W,B_W+M_W)$ is a $A_W$-equivariant generalized klt $N$-complement. 
By Theorem~\ref{introthm:RC-action}, we conclude that the order of $A_W$ is bounded by a constant which only depends on $\dim W\leq n-1$ and $N$.
Thus, we may replace $A$ with the kernel of the surjective homomorphism $A\rightarrow A_W$ and assume that $A$ acts as the identity on the rationally connected variety $W$.
By the first step, we conclude that the statement of the theorem holds for $(X,B+M)$

Hence, we can assume that $(2)$ holds.
Let $(W',B'_{W'})$ be the log pair obtained by push-forward of $(W,B'_W)$ to $W'$.
By Lemma~\ref{lem:dom-lcc}, we conclude that all the log canonical centers of $(W',B'_{W'})$ dominate 
$Z$.
Recall that the action of $A_W$ on a
minimal log canonical center of
$(W',B'_{W'})$ is trivial. 
We conclude that the action 
induced on the base of the $A_W$-equivariant projective contraction $W'\rightarrow Z$ is trivial as well.
Thus, we conclude that the $A_W$-action on $W'$ is fiberwise over $Z$.
Let $F$ be a general fiber of the projective contraction $W'\rightarrow Z$.
By the above considerations, we have that $A_W$ acts on $F$.
Furthermore, the generalized pair
\[
K_F+B_F+M_F \sim_\qq (K_{W'}+B_{W'}+M_{W'})|_F
\]
is an $A_W$-equivariant generalized klt $N$-complement of the rationally connected variety $F$.
By Theorem~\ref{introthm:RC-action}, we conclude that the order of $A_W$ must be bounded above by a constant that only depends on $\dim F\leq n-1$ and $N$.
Thus, we may replace $A$ by the kernel of the surjective homomorphism
$A\rightarrow A_W$ and assume that $A$ acts as the identity on $W$. By the first step, we conclude that the statement of the theorem holds for $(X,B+M)$.\\

Now, we proceed to prove that the statement of the theorem in dimension at most $n-1$, implies the statement of the theorem in dimension $n$ for log pairs. Let $(X,B)$ be a log pair of dimension $n$ as in the statement of the theorem.\\

{\it Step 3:} In this step, we take a higher $A$-equivariant birational model of $(X,B)$ and run an $A$-equivariant minimal model program.\\

Let $(X,B)$ be a $n$-dimensional $A$-equivariant log pair satisfying the assumptions of the theorem. 
Let $X_0 \rightarrow X$ be an $A$-equivariant canonicalization of a $A$-equivariant $A\qq$-factorial dlt modifcation of $(X,B)$.
We denote by $(X_0,B_0)$ the log pair obtained from $(X,B)$ by log pullback to $X_0$.
We run an $A$-equivariant minimal model program for $K_{X_0}$.
Note that $X_0$ has $G\qq$-factorial canonical singularities.
Since $X_0$ is rationally connected, this minimal model program terminates with an $A$-equivariant Mori fiber space 
$\phi_1 \colon X_1\rightarrow Z_1$.
As usual, we denote by $B_1$ the push-forward of $B_0$ to $X_1$.
By Lemma~\ref{lem:contr-minimal-rel-dim}, we can find an $A$-equivariant birational model $(X_2,B_2)$ of $(X_1,B_1)$
that admits an $A$-equivariant contraction $\phi_2\colon X_2\rightarrow Z_2$ with minimal relative dimension.
In particular, we have that $X_2$ is canonical and $-K_{X_2}$ is ample over $Z_2$.\\

{\it Step 4:} In this step, we prove the statement of the theorem for $(X,B)$ when 
$Z_2$ is positive dimensional.\\

We assume that $\dim Z_2>0$.
Note that $Z_2$ is a Fano type variety as it is the image of a Fano type variety via a projective contraction.
We have a short exact sequence
\[
1\rightarrow A_F \rightarrow A\rightarrow A_Z\rightarrow 1, 
\]
where $A_F$ acts fiberwise and $A_Z$ acts on the base $Z_2$.
By Lemma~\ref{lem:G-equiv-cbf}, up to replacing $N$ with a bounded multiple, the canonical bundle formula induces an $A_Z$-equivariant
$N$-complement $(Z_2,B_{Z_2}+M_{Z_2})$.
Without loss of generality, we may assume that
$\zz_m^{k_1} \leqslant A_F$ and 
$\zz_m^{k_2} \leqslant A_Z$ with $k_1+k_2=r$.
By the induction hypothesis, the statement of the theorem holds for $(Z_2,B_{Z_2}+M_{Z_2})$.\\

Let $(Z'_2,B_{Z'_2}+M_{Z'_2})$ be a $A_Z$-equivariant gdlt modification of $(Z_2,B_{Z_2}+M_{Z_2})$.
By Lemma~\ref{lem:special-model}, we may replace $(X_2,B_2)$ with an $A$-equivariant log crepant model of $(X_2',B_2')$ which satisfies the following conditions:
\begin{enumerate}
\item $(X'_2,B'_2)$ has generalized lc singularities,
\item the preimage of ${\rm glcc}(Z'_2,B_{Z'_2}+M_{Z'_2})$ is contained in $E':=\lfloor B'_2 \rfloor_{\rm vert}$, 
\item the generalized pair 
$(X'_2, E')$ is dlt, and
\item the contraction $\phi'\colon X'_2\rightarrow Z'_2$ has relative $A$-Picard rank one on the complement of ${\rm glcc}(Z'_2,B_{Z'_2}+M_{Z'_2})$.
\end{enumerate}
We run an $A$-equivariant
$(K_{X'_2}+E')$-MMP over $Z_2'$
which terminates with an $A$-equivariant Mori fiber space
$\phi''\colon X''_2\rightarrow Z''_2$
to a higher model of $Z'_2$.
We denote by $B''_2$ the push-forward of $B'_2$ to $X''_2$ and by $E''$ the push-forward of $E'$ to $X''_2$.
We denote by 
\[
(Z''_2,B_{Z''_2}+M_{Z''_2})
\]
the log pullback of
$(Z'_2,B_{Z'_2}+M_{Z'_2})$ to $Z''_2$.
We claim that the birational map
$Z''_2\rightarrow Z'_2$ only extract log canonical centers of $(Z'_2,B_{Z'_2}+M_{Z'_2})$.
Indeed, by the condition (4) of the contraction $\phi'$, this MMP is an isomorphism over the complement of 
${\rm glcc}(Z'_2,B_{Z'_2}+M_{Z'_2})$.
Hence, every divisor extracted by $Z_2''\rightarrow Z'_2$ is in the image via $\phi''$ of a divisor contained in $E''$.
This implies that every divisor extracted in $Z_2''$ appears in $B_{Z''_2}$ with coefficient one.
Note that $\phi''$ has $A$-equivariant relative Picard rank one. 
Recall that each prime component of $E''$ is $A$-invariant.
Hence, every divisorial glc center of $(X'_2,E'')$ must map to a divisorial glc center
of $(Z''_2,B_{Z''_2}+M_{Z''_2})$.
Vice-versa, every divisorial glc center of 
$(Z''_2,B_{Z''_2}+M_{Z''_2})$ must be the image of a unique divisorial glc center
of $(X'_2,E'')$.
We conclude that there is an identification of dual complexes
\[
\mathcal{D}(X_2'',E'')
\simeq \mathcal{D}(Z_2'',B_{Z_2''}+M_{Z''_2}).
\]

By~\cite[Proposition 40]{dFKX17}, we conclude that the generalized pair 
$(Z''_2,B_{Z_2''}+M_{Z_2''})$ has generalized qdlt singularities.
By the inductive hypothesis, we know that $(Z_2',B_{Z'_2}+M_{Z'_2})$ contains a generalized log canonical center of codimension $k_2$ which is fixed pointwise by $A_Z$.
Furthermore, $A_Z$ acts formally toric at the generic point of this generalized log canonical center.
We call such generalized log canonical center $W_{Z'_2}$.
By Lemma~\ref{lem:qld-mod-dlt}, we know that 
$(Z_2'',B_{Z''_2}+M_{Z''_2})$ admits a generalized log canonical center $W$ which maps birationally to $W_{Z'_2}$.
In particular, $A_Z$ fixes $W$ pointwise and it acts formally toric at the generic point of $W$.
Let $V$ be the unique generalized log canonical center of $(X''_2,E'')$ that maps onto $W$.
Let 
\[
K_V+E''_V \sim_\qq 
K_{X''_2}+E''|_V,
\]
be the generalized dlt pair obtained by adjunction.
Let $A_V$ be the quotient group of $A$ acting on $V$.
Then, we have that 
$(V,E''_V)$ has $A_V$-equivariant generalized dlt singularities
and $-(K_V+E''_V)$ is ample over $W$.
Furthermore, $(V,E''_V)$ has gklt singularities over the generic point of $W$.
Since $A_Z$ acts as the identity on $W$, we conclude that $A_V$ acts fiberwise on the contraction $W\rightarrow V$.
By the above considerations, we conclude that the general fiber $F$ of $W\rightarrow V$ is a $A_V$-equivariant Fano type variety.

Let 
\[
K_V+B'_V \sim_\qq 
(K_{X'_2}+B'_2)|_V
\]
be the generalized log canonical pair obtained by adjunction.
Note that this pair is a $A_V$-equivariant $N$-complement.
Let 
\[
K_F+B'_F\sim_\qq 
(K_V+B'_V)|_F
\]
be the generalized pair obtained by adjunction to the general fiber.
Then, $(F,B'_F)$ is a $A_V$-equivariant $N$-complement for the Fano type variety $F$.
Note that the dimension of $F$ is at most $n-1$.
Let $f$ be the dimension of $F$.
By induction on the dimension, we know that a subgroup $A_{0,V} \triangleleft A_V$, of index at most $c_0(f,N)$, acts as the identity on a minimal log canonical center of an $A_{0,V}$-equivariant dlt modification of $(F,B'_F)$.
By Lemma~\ref{lem:triv-act-in-out}, we conclude that a subgroup $A_0\triangleleft A$, of index at most $c_0(f,N)$, acts as the identity on a minimal log canonical center of a $A_0$-equivariant dlt modification of $(X,B)$.
By the first step, we conclude that the statement of the theorem holds for $(X,B+M)$.\\

{\it Step 5:} In this step, we prove the statement of the theorem for $(X,B)$ when $Z_2$ is zero dimensional.\\

Since $Z_2$ is zero-dimensional, we have that the log pair $(X_2,B_2)$ is log bounded.
Indeed, $X_2$ is a $A$-equivariant canonical Fano variety and the coefficients of $B_2$ are in a finite set which only depends on $N$ (see, e.g.,~\cite[Theorem 3.3]{FM20}).
Recall that $N(K_{X_2}+B_2)\sim 0$.
By the proof of~\cite[Theorem 4.1]{Mor20b}, we conclude that, up to replacing $A$ with a normal subgroup of index $c(n)$, we have that
$A< \mathbb{G}_m^r \leqslant {\rm Aut}(X,B)$.
Hence, $X$ admits a $r$-dimensional torus action
and $(X,B)$ is a torus invariant pair.
Let $\widetilde{X}\rightarrow X$ be a
projective torus equivariant birational morphism that admits a good quotient for the torus action.
This means that we can find 
a projective morphism $\widetilde{X}\rightarrow Z$ which is a good quotient for the torus action.
The general fiber of $\widetilde{X}\rightarrow Z$ is a projective toric variety $T$.
Write $(\widetilde{X},B_{\widetilde{X}})$ for the log pullback of $(X,B)$ to $\widetilde{X}$.
Note that $(\widetilde{X},B_{\widetilde{X}})$
is a log Calabi-Yau sub-pair.
Let $(T,B_T)$ be the restriction of the log Calbai-Yau sub-pair
$(\widetilde{X},B_{\widetilde{X}})$
to the general fiber of
$\widetilde{X}\rightarrow Z$.
Note that $(T,B_T)$ is
a torus invariant
log Calabi-Yau sub-pair.
Hence, $B_T$ must be the reduced toric boundary.
Note that the dimension of $T$ is $r$.
Hence, the 
log sub-pair $(\widetilde{X},B_{\widetilde{X}})$ admits a qdlt log canonical center $W$ of codimension $r$ which is fixed pointwise by $A$.
By performing toroidal blow-ups at $W$, we obtain a log smooth log Calabi-Yau $A$-invariant sub-pair $(\widetilde{X}',B_{\widetilde{X}'})$ so that: 
\begin{enumerate}
\item There are $A$-invariant prime components $E_1,\dots,E_r$ of $\lfloor B_{\widetilde{X}'}\rfloor_{> 0}$, and
\item $A$ acts as the identity on the irreducible intersection
$W':=E_1\cap\dots\cap E_r$.
\end{enumerate}
In particular, we have that $A<\mathbb{G}_m^r \leqslant {\rm Aut}(\widetilde{X}',E_1+\dots+E_r;\eta_{W'})$.
Let $\widetilde{B}$ be the effective divisor on $\widetilde{X}'$ obtained from $B_{\widetilde{X}'}$ by increasing to one all the coefficients of 
$B_{\widetilde{X}'}$ at exceptional divisor of $\widetilde{X}'\rightarrow X$.
Then, the log pair $(\widetilde{X}',\widetilde{B})$ is $A$-invariant and dlt.
We run an $A$-equivariant MMP for
$K_{\widetilde{X}'}+\widetilde{B}$ over $X$ with scaling of an $A$-invariant ample divisor.
This MMP terminates with an $A$-equivariant $A\qq$-factorial dlt modification of $(X_2,B_2)$
and is an isomorphism at the generic point of $W'$.
Indeed, 
\[
K_{\widetilde{X}'} + \widetilde{B}\sim_{\qq,X} F \geq 0,
\]
where $F$ is an effective divisor that does not contain $W'$.
Since $F$ is contained in the diminished base locus of $K_{\widetilde{X}'}+\widetilde{B}$ over $X$, then it must be contracted after finitely many steps of this minimal model program.
This $A$-equivariant dlt modification satisfies the statement of the theorem.
Indeed, the $A$-equivariant birational map $\widetilde{X}'\rightarrow Y$ is an isomorphism at the generic point of $W'$.
Thus, the properties $(1)$-$(3)$ hold on $Y$.\\

{\it Step 5:} 
In this step, we finish the proof of the theorem.\\

By the third, fourth, and fifth steps, the statement of the theorem in dimension at most $n-1$
implies the statement of the theorem in dimension $n$ for log pairs, i.e.,
for generalized pairs with trivial nef part.
By the second step,
the statement of the theorem in dimension $n$ for log pairs implies the statement of the theorem in dimension $n$.
The case of dimension one is covered by the fourth step.
The proof of the theorem follows.
\end{proof}

Now, we turn to prove Theorem~\ref{introthm:rank-vs-reg-global} about the rank and the regularity of Fano type varieties.

\begin{proof}[Proof of Theorem~\ref{introthm:rank-vs-reg-global}]
By Theorem~\ref{thm:rank-vs-reg}, we can find
a dlt modification $(Y,B_Y)$ of $(X,B)$
so that $(Y,B_Y)$ admits a log canonical center
$W\subset Y$ of codimension $r$.
This implies that $\mathcal{D}(Y,B_Y)$ has dimension at least $r-1$.
\end{proof}

To conclude this section, we prove Theorem~\ref{introthm:rank-vs-reg-local} 
about the rank and the regularity of a klt singularity.

\begin{proof}[Proof of Theorem~\ref{introthm:rank-vs-reg-local}]
Let $(X;x)$ be a $n$-dimensional klt singularity.
Assume that $\zz_m^r \leqslant {\rm Aut}(X;x)$ for some $m\geq c(n)$.
By Lemma~\ref{lem:G-equiv-N-comp}, there exists a $\zz_m^r$-equivariant boundary $B\geq 0$
so that $(X,B;x)$ is log canonical
and $N(K_X+B)~0$ around $x\in X$.
Let $(Y,B_Y)$ be a $G$-equivariant dlt modification of $(X,B)$.
By Theorem~\ref{introthm:almost-fixed-local}, up to replacing $\zz_m^r$ by a normal subgroup which only depends on the dimension of $X$, 
we may assume that $\zz_m^r$ fixes all the log canonical centers of $(Y,B_Y)$.
Let $E$ be a divisorial log canonical center of $(Y,B_Y)$.
Let $(E,B_E)$ be the pair obtained by adjunction.
Let $G_E$ be the quotient group of $G$ acting non-trivially on $E$.
Without loss of generality, we may assume that
$\zz_m^{r-1}\leqslant G_E$.
By Theorem~\ref{introthm:rank-vs-reg-global}, we conclude that $\mathcal{D}(E,B_E)$ is at least $(r-2)$-dimensional.
Hence, we conclude that
$\mathcal{D}(Y,B_Y)$ is at least $(r-1)$-dimensional.
This finishes the proof.
\end{proof}

\section{Toroidalization of finite actions}

In this section, we prove the main theorem of the paper. The following theorem is a generalization of Theorem~\ref{introthm:group-toroidal}.
The following theorem states that we can perform a toroidalization of finite actions on generalized klt singularities.
This fact is a generalization of the Jordan property for finite subgroups of ${\rm GL}_n(\kk)$ (see, e.g.,~\cite{Jor1873}).

\begin{theorem}\label{thm:gen-group-toroidal}
Let $n$ and $p$ be positive integers
and $\Lambda$ be a set of rational numbers with only rational accumulation points.
There exists a constant $c:=c(n,p,\Lambda)$,
only depending on $n,p$ and $\Lambda$, 
satisfying the following.
Let $G$ be a finite group.
Let $(X,\Delta+M;x)$ be a $G$-equivariant $n$-dimensional generalized klt singularity satisfying the following:
\begin{enumerate}
    \item[(a)] The coefficients of $\Delta$ belong to $\Lambda$, and
    \item[(b)] $pM$ is Cartier in the quotient where it descends.
\end{enumerate}
Then, there exists:
\begin{enumerate}
    \item[(i)] A normal abelian subgroup
    $A\triangleleft G$ of index at most $c$ and rank $r\leq n$,
    \item[(ii)] a $G$-equivariant $c$-complement
    $(X,B+M)$ of $(X,\Delta+M)$, and 
    \item[(iii)] an $A$-equivariant projective birational morphism
    $\pi\colon Y\rightarrow X$, 
\end{enumerate}
so that the following statements hold:
\begin{enumerate}
    \item the exceptional locus of $\pi$ is divisorial and consists of $r$ prime divisors $E_1,\dots,E_r$, 
    \item each $E_i$ is $A$-invariant and a generalized log canonical place of $(X,B+M)$,
    \item the intersection $E_1\cap\dots\cap E_r$ is non-trivial,
    \item $A$ fixes a component $Z$ of $E_1\cap\dots\cap E_r$,
    \item $(Y,E_1+\dots+E_r)$ is formally toric at the generic point $\eta_Z$ of $Z$, 
    \item $A$ acts as the identity on $Z$, and 
    \item $A<\mathbb{G}_m^r \leqslant {\rm Aut}(Y,E_1+\dots+E_r,\eta_Z)$.
\end{enumerate}
In particular, the action of $A$ on $Y$ is toroidal at the generic point of $Z$.
\end{theorem}

\begin{proof} 
By Lemma~\ref{lem:G-equiv-N-comp}, 
we can produce a $G$-equivariant $N$-complement
$(X,B+M;x)$ of $(X,\Delta+M;x)$.
This $N$ only depends on $n,p$ and $\Lambda$.
By Theorem~\ref{thm:almost-fixed-D-local}, we may replace $G$ by a subgroup whose index only depends on the dimension, so that this subgroup acts trivially on the dual complex of $(X,B+M;x)$.
Thus, we may assume that $G$ acts trivially on $\mathcal{D}(X,B+M;x)$.
Let $(Y_0,B_{Y_0}+M_{Y_0})$ be a $G$-equivariant generalized dlt modification of $(X,B+M;x)$.
Let $(E,B_E+M_E)$ be the pair obtained by adjunction to a divisorial generalized log canonical center of $(Y_0,B_{Y_0}+M_{Y_0})$.
Let $G_E$ be the quotient group of $G$ acting non-trivially on $E$.
By~\cite[Theorem 3.1]{BFMS20},
we know that there exists a normal abelian subgroup
$A_0\triangleleft G_E$ of index at most $c_0(n-1)$.
Hence, replacing $G$ with the preimage of $A_0$ in $G$,
we may assume that $G_E$ is an abelian group.
Hence, we may write
$G_E \simeq \zz_{m_1} \oplus \dots \oplus \zz_{m_{r-1}}$
with $m_1\mid \dots \mid  m_{r-1}$ and $r-1\leq n-1$.
By Theorem~\ref{thm:rank-vs-reg}, there exists a constant
$c_1:=c_1(n-1,N)$, only depending on $n-1$ and $N$, satisfying the following.
If $m_1 \geq c_1$, then there exists:
\begin{enumerate}
    \item[(i)] A normal abelian subgroup $A_E\triangleleft G_E$ of index at most $c_1$, and
    \item[(ii)] an $A_E$-equivariant generalized gdlt modification $(E',B_{E'}+M_{E'})$ of $(E,B_E+M_E)$.
\end{enumerate}
Furthermore, the $A_E$-equivariant gdlt modification satisfies the following conditions: 
\begin{enumerate}
    \item There exists a generalized log canonical center $W\subset E'$ of $(E',B_{E'}+M_{E'})$ of codimension $r-1$ which is fixed pointwise by $A_E$,
    \item $A_E$ fixes all the divisorial log canonical centers $E_1,\dots,E_{r-1}$ of $(E',B_{E'}+M_{E'})$ containing $W$, and
    \item there is a monomorphism
    $A_E < \mathbb{G}_m^{r-1}\leqslant {\rm Aut}(E',E_1+\dots+E_{r-1};\eta_W)$.
\end{enumerate}
We may assume that $m_1\geq c_1$.
Otherwise, we can replace $G_E$ with
$\zz_{m_2}\oplus\dots\oplus \zz_{m_r}$,
replace $r-1$ with $r-2$, 
and proceed inductively.
Let $A$ be the preimage of $A_E$ in $G$.
Note that $A$ has index at most $c_1$ in $G$.
The group $A$ is an extension of a cyclic group
and an abelian group of rank $r$.\\

\noindent
\textit{Claim:} There exists a constant $c_2(r)$, 
only depending on $r$, satisfying the following.
There exists a normal abelian subgroup of $A$
of index at most $c(r)$.

\begin{proof}[Proof of the Claim]
By the conditions $(1)$-$(3)$, there exists
an $A$-equivariant gdlt modification $(Y_1,B_{Y_1}+M_{Y_1})$ of $(Y_0,B_{Y_0}+M_{Y_0})$ that satisfies the following conditions:
\begin{enumerate}
    \item[(a)] There exists a generalized log canonical center $W\subset Y_1$ of $(Y_1,B_{Y_1}+M_{Y_1})$ of codimension $r$ which is fixed pointwise by $A$, and
    \item[(b)] $A$ fixes all the divisorial log canonical centers $E_1,\dots,E_r$ of $(Y_1,B_{Y_1}+M_{Y_1})$ containing $W$.
\end{enumerate}
Indeed, the conditions $(1)$-$(3)$ above are preserved by passing to higher $A_E$-equivariant gdlt modifications of $(E',B_{E'}+M_{E'})$.
Hence, it suffices to find a sufficiently high $A$-equivariant gdlt modification of $(Y_0,B_{Y_0}+M_{Y_0})$ on which the center of $E'$ is a divisor.
For the generalized log canonical center,
we can take $W\subset E'$ provided by $(1)$.
Note that $W$ has codimension $r$ on $Y$.
For condition (b), it suffices to set $E_r=E'$.
Localizing at the generic point of $W$, 
we have that $A<{\rm GL}_r(\kk)$.
By~\cite{Jor1873}, we know that there exists a normal abelian subgroup $A'\triangleleft A$ of index bounded above by $c_2(r)$
and rank at most $r$.
This finishes the proof of the claim.
\end{proof} 

Replacing $A$ with the subgroup $A'$ constructed in the claim,
we may assume that $A$ is a normal abelian group of rank $r$.
By condition (b) above, we have a monomorphism 
$A<\mathbb{G}_m^r \leqslant {\rm Aut}(Y,E_1+\dots+E_r;\eta_W)$.
Observe that $A\triangleleft G$ has index at most $c_1c_2(r)$.
We can set $c(n,p,\Lambda):=c_1(n-1,N)c_2(r)$.
Indeed, $N$ only depends on $n,p$ and $\Lambda$.
On the other hand, 
the $A$-equivariant gdlt modification 
$(Y_1,B_{Y_1}+M_{Y_1})$ satisfies $(2)$-$(7)$ as in the statement of the theorem.
Thus, it suffices to run a suitable minimal model program to achieve condition $(1)$.

Let $F$ be a prime exceptional divisor of $Y_1\rightarrow X$ which is not in the support of $E_1+\dots+E_r$.
The divisor $F$ is $A$-invariant.
Since $Y_1\rightarrow X$ is a Fano type morphism, then we can run an $A$-invariant $F$-MMP over $X$.
This minimal model program will contract $F$ and it does not contract other prime divisors over $X$.
Note that each such prime divisor is $A$-invariant.
We can proceed inductively with all the exceptional prime divisors of $Y\rightarrow X$ which are not in the support of $E_1+\dots+E_r$.
Hence, we will obtain $A$-equivariant projective morphism $Y_2\rightarrow X$ which only extracts the strict transform of the divisors $E_1,\dots, E_r$.
By abuse of notation, we denote these strict transforms $E_1,\dots, E_r$, respectively.
The $A$-equivariant birational map
$Y_1\dashrightarrow Y_2$ is an isomorphism at the generic point of $W$, hence the conditions $(1)$-$(3)$ are still satisfied in the model $Y_2$.
Furthermore, since the model $Y_2$ is $A\qq$-factorial and the divisors $E_1,\dots,E_r$ are $A$-fixed, then they are $\qq$-Cartier.
Write $\phi_2\colon Y_2\rightarrow X$ for the induced projective morphism.
Write
\[
\phi_2^*(B-\Delta) =
\Gamma + \alpha_1 E_1 +\dots + \alpha_r E_r,
\]
for positive rational numbers
$\alpha_1,\dots, \alpha_r$.
The $A$-invariant effective divisor $\Gamma$ does not contain the generic point of $W$ in its base locus over $X$.
We run an $A$-equivariant MMP for $\Gamma$ over $X$, which terminates with a good minimal model and then we take its
$A$-equivariant ample model $Y\rightarrow X$.
By abuse of notation, for each $i$, we will denote by $E_i$, the push forward on $Y$ of $E_i\subset Y_2$.
Observe that on $Y$ the divisor 
\[
-(\alpha_1 E_1+ \dots + \alpha_r E_r)
\]
is ample over the base.
Hence, the exceptional locus of $\pi\colon Y\rightarrow X$ is divisorial and consists of the $r$ prime divisors $E_1,\dots, E_r$. This shows $(1)$.
Each prime divisor $E_i$ is $A$-invariant.
By construction, each $E_i$ is a generalized log canonical place of $(X,B+M)$.
This shows $(2)$.
The conditions (4)-(7) are satisfied by $(Y,E_1+\dots+E_r;W)$ as they are satisfied in the model $Y_1$ and 
$Y_1\dashrightarrow Y$ is an
$A$-equivariant isomorphism on a neighborhood of a general point of $W$. This concludes the proof.
\end{proof}

\begin{remark}{\em 
In Theorem~\ref{thm:gen-group-toroidal},
if we assume that the germ $X$ is $A\qq$-factorial,
then we can further prove that the divisors $E_i$ are $\qq$-Cartier.
Indeed, we can take the small $A\qq$-factorialization $Y_0$ of the model $Y$.
Since $X$ is $A\qq$-factorial, then $Y_0\rightarrow Y$ is an isomorphism over the complement of $\supp(E_1+\dots+E_r)$.
In this new model, the exceptional divisor is still divisorial and it consists of the $r$ divisors $E_1,\dots, E_r$.
Indeed, the pull-back of $-(\alpha_1E_1+\dots+\alpha_rE_r)$ to $Y_0$ is semiample, hence every curve extracted by $Y_0\rightarrow X$ must be contained in the support of $E_1+\dots+E_r$.
The variety $Y$ is $\qq$-factorial at the generic point of $E_1\cap\dots \cap E_r$, so all the properties $(1)$-$(7)$ are preserved by the small $A\qq$-factorialization.
Since each $E_i$ is $A$-invariant, then they are $\qq$-Cartier in this higher model $Y_0$.
}
\end{remark} 

\begin{remark}{\em 
We remark that the proof of Theorem~\ref{thm:gen-group-toroidal} implies that~\cite[Theorem 2]{BFMS20} holds over any algebraically closed field of characteristic zero, when working with the \'etale fundamental group.
In the proof of~\cite[Theorem 2]{BFMS20}, we use the Whitney retraction to an exceptional divisor, which is a complex analytic technique.
On the other hand, Theorem~\ref{thm:rank-vs-reg} and Theorem~\ref{thm:gen-group-toroidal} only use~\cite[Proposition 3.1]{BFMS20}, which holds over any field of characteristic zero.
}
\end{remark}

\section{Fundamental groups of klt singularities}

In this section, we discuss the implications of the main theorem on the fundamental group of klt singularities.
Passing to the universal cover, 
we are in the situation of finite actions on klt singularities, so Theorem~\ref{thm:gen-group-toroidal} can be applied.
We introduce versions of Corollary~\ref{introcor:p1-toroidal}
and Theorem~\ref{introthm:Jordan-reg} for generalized klt pairs. 

\begin{corollary}
Let $n$ be a positive integer. There exists a constant $c(n)$, 
only depending on $n$, that satisfies the following conditions.
Let $(X,B+M;x)$ be a $n$-dimensional generalized klt singularity. Then, there exists:
\begin{enumerate}
    \item A projective birational morphism $\phi \colon Y\rightarrow X$, 
    \item an effective divisor $B_Y$ on $Y$, and
    \item a klt toric singularity $(Y,B_Y;y)$,
\end{enumerate}
so that the induced homomorphism
\[
\phi_* \colon
\pi_1^{\rm reg}(Y,B_Y;y) \rightarrow 
\pi_1^{\rm reg}(X,B+M;x)
\]
has cokernel 
of order at most $c(n)$.
In particular, $\pi_1^{\rm reg}(X,B+M;x)$ admits 
a normal abelian subgroup of rank at most ${\rm codim}(Y,y)$ and index at most $c(n)$.
\end{corollary} 

\begin{proof}
Let $(X,B+M;x)$ be a $n$-dimensional klt singularity.
Let $(X',B'+M';x')$ be the universal cover of $(X,B+M;x)$ and 
$G$ be the finite group acting on $(X',B'+M')$.
Here, $x'$ is the pre-image of $x$.
We know that $x'$ is a $G$-fixed point on $X'$.
Note that the singularity $(X',B'+M';x)$ is 
generalized klt due to Proposition~\ref{prop:quot-gen-pair}.
Let $\phi' \colon X''\rightarrow X'$
be a $G$-equivariant small $G\qq$-factorialization of $X'$.
We know that the $G$-equivariant birational map
$X''\rightarrow X'$ is of Fano type.
By Theorem~\ref{thm:almost-fixed-points}, there exists a constant $c_0(n)$, only depending on $n$, and a normal subgroup $A_0\triangleleft G$, of index at most $c_0(n)$, so that $A_0$ admits a fixed point on ${\phi'}^{-1}(x')$.
Let $x''$ be such a fixed point.
By Theorem~\ref{thm:gen-group-toroidal},
there exists a constant $c_1(n)$, only depending on $n$,
satisfying the following:
\begin{enumerate}
    \item There exists a normal subgroup $A_1\triangleleft A_0$ of index at most $c_1(n)$, 
    \item there exists an $A_1$-equivariant projective birational morphism $X'''\rightarrow X''$, and 
    \item $X'''$ admits a formally toric point $x'''\in X'''$ on which $A_1$ acts formally toric. 
\end{enumerate}
Here, $x'''$ is the generic point of the center $Z$ constructed in Theorem~\ref{thm:gen-group-toroidal}.
Let $(Y,B_Y)$ be the log quotient of $X'''$ by $A_1$ 
(see Proposition~\ref{prop:quot-gen-pair}). Then, we have a commutative diagram as follows:
\[
\xymatrix{
A_1 \acts  (X''',x''')\ar[rr] \ar[d] & & (Y,B_Y) \ar[dd]^-{\phi} \\
A_0 \acts (X'',x'') \ar[d]^-{\phi'} & & \\
G \acts (X',\Delta';x') \ar[rr] & & (X,\Delta;x).
}
\]
Let $y$ be the image of $x'''$ on $Y$.
By construction, the point $(Y,B_Y;y)$ is a toric singularity.
Furthermore, the natural homomorphism
\begin{equation}\label{eq:almost-surj}
\phi_* \colon \pi_1^{\rm reg}(Y,B_Y;y)
\rightarrow \pi_1^{\rm reg}(X,\Delta;x)\simeq G,
\end{equation} 
surjects onto the normal subgroup $A_1\leqslant G$.
Note that $A_1$ is a normal abelian subgroup of rank at most ${\rm codim}(Y,y)$.
Moreover, the cokernel of the homomorphism~\eqref{eq:almost-surj} has order at most $c_0(n)c_1(n)$.
This finishes the proof of the corollary.
\end{proof}

\begin{theorem}
\label{thm:Jordan-reg-gen}
Let $n$ and $r$ be a positive integer with $r\leq n-1$.
Then, there exists a constant $c=c(n)$, only depending on $n$, satisfying the following property. Let $x\in (X,B+M)$ be a $n$-dimensional generalized klt singularity
of regularity $r$.
Then, there is an exact sequence
\begin{equation}
\label{eq:reg-fund}
\xymatrix{
1\ar[r] &
A\ar[r] &
\pi_1^{\rm reg}(X,B+M,x) \ar[r]&
N\ar[r] &
1,
}
\end{equation}
where $A$ is a finite 
abelian group of rank at 
most $r+1$ 
and index at most $c(n)$.
Furthermore, since the regional fundamental group surjects onto 
the local 
fundamental group of the singularity, 
we obtain an exact sequence 
\begin{equation}
\label{eq:loc-fund}
\xymatrix{
1\ar[r]&
A'\ar[r]&
\pi_1^{\rm loc}(X,x) \ar[r]&
N'\ar[r]&
1,
}
\end{equation}
where $A'$ is a finite abelian group of rank at most $r+1$ and index at most $c(n)$.
\end{theorem}

\begin{proof}
Let $(X',B'+M';x')\rightarrow (X,B+M;x)$ be the universal cover of the generalized klt singularity.
By Proposition~\ref{prop:quot-gen-pair}, we know that
$(X',B'+M';x')$ is a generalized klt singularity.
Let $G$ be the finite group acting on $(X',B'+M';x')$.
Then, $x'$ is a $G$-fixed point.
Let $\Delta \geq 0$ be an effective divisor so that $(X,B+\Delta+M;x)$ is a $\qq$-complement computing the regularity of the generalized pair $(X,B+M)$ at $x$, i.e., 
\[
{\rm reg}(X,B+\Delta+M;x) =
{\rm reg}(X,B+M).
\]
Let $B^s\leq B$ be the standard approximation of $B$.
We can find a $N$-complement $(X,\Gamma;x)$ of $(X,B^s;x)$ so that
\[
{\rm reg}(X,\Gamma;x)= {\rm reg}(X,B+\Delta+M;x).
\]
Here, $N$ only depends on the dimension $n$ and the set of standard coefficients, which is independent of the dimension.
Let $(X',\Gamma')$ be the log pullback of $(X,\Gamma)$ to $X'$.
Since $\Gamma\geq B^s$, then 
$(X',\Gamma')$ is a log pair.
Since $(X,\Gamma)$ is log canonical, then
$(X',\Gamma')$ is log canonical due to
Proposition~\ref{prop:quot-gen-pair}.
By Lemma~\ref{lem:reg-under-quot}, we have that
\[
{\rm reg}(X',\Gamma';x') =
{\rm reg}(X,\Gamma;x) =
{\rm reg}(X,B+\Delta+M;x) =
{\rm reg}(X,B+M;x) = r.
\]
By Theorem~\ref{thm:gen-group-toroidal}, we
can find a normal subgroup
$A\triangleleft G$ of index at most $c(n)$,
so that $A$ has rank at most $r$.
This proves the existence of the exact sequence~\eqref{eq:reg-fund}.
Since we have a natural surjection
$\pi_1^{\rm reg}(X,\Delta;x) \rightarrow \pi_1^{\rm loc}(X;x)$, then the existence of
the exact sequence~\eqref{eq:loc-fund} follows.
\end{proof}

We continue this subsection by discussing the fundamental group of the smooth locus of Fano type varieties.
Given a Fano type variety $X$, we know that
$X^{\rm reg}$ has finite fundamental group (see, e.g.,~\cite[Theorem 2]{Bra20}).
Furthermore, by~\cite[Proof of Theorem 2]{BFMS20}, we know that there exists a closed point $x\in X$ for which
\[
\pi_1^{\rm reg}(X;x) \rightarrow \pi_1(X^{\rm reg})
\]
is almost surjective with respect to the dimension.
This means that almost every loop of $X^{\rm sm}$ is
homotopic to a loop arbitrarily close to the point $x$.
It is natural to try to understand which closed points $x\in X$ have this property.
The following theorem shows that we can choose such $x\in X$ to be a general closed point in a log canonical center of some $\qq$-complement of $X$.

\begin{theorem}
Let $n$ be a positive integer.
There exists a constant $c(n)$, only depending on $n$, 
satisfying the following.
Let $X$ be a $n$-dimensional Fano type variety.
Then, there exists 
\begin{enumerate}
    \item A normal abelian subgroup $A\triangleleft \pi (X^{\rm reg})$ of indext at most $c(n)$ and rank at most $r\leq n$, 
    \item a boundary $B$ on $X$ so that $(X,B)$ is log Calabi-Yau, and 
    \item a log canonical center $W$ of $(X,B)$ of codimension at least $r$ so that
    \[
    \pi_1^{\rm reg}(X;x)\rightarrow \pi_1(X^{\rm reg}),
    \]
    has cokernel of order at most $c(n)$
    for $x\in W$ general. 
\end{enumerate}
\end{theorem}

\begin{proof}
By the Jordan property for klt singularities,
we can find a normal abelian subgroup 
$A\leqslant \pi_1(X^{\rm reg})$ of index at most $c'(n)$ and rank at most $n$.
Let $r$ be the rank of $A$.
Let $X'\rightarrow X$ be the normal closure of $X$ in the function field of the universal cover of $X^{\rm reg}$.
Then $X'$ admits a $G$-action,
with $G:=\pi_1(X^{\rm reg})$.
By Theorem~\ref{thm:rank-vs-reg}, there exists a constant $c''(n)$, only depending on $n$, satisfying the following:
\begin{enumerate}
    \item A normal abelian subgroup $A_0\triangleleft A$ of rank $r$ and index at most $c''(n)$, 
    \item an $A_0$-equivariant $c''(n)$-complement $(X',B')$, 
    \item an $A_0$-equivariant dlt modification $(Y,B_Y)$ of $(X',B')$, and 
    \item a log canonical center $W_Y$ of $(Y,B_Y)$ of codimension $r$ in $Y$ on which $A_0$ acts as the identity.
\end{enumerate}
Let $W$ be the image of $W_Y$ on $X$.
$W$ is a log canonical center of $(X,B)$, the log quotient of $(X',B')$ on $X$.
Furthermore, $W$ has codimension at least $r$.
Then, by construction, we have that for $x\in W$ general, the image of the homomorphism 
\[
\pi_1^{\rm reg}(X;x)\rightarrow \pi_1(X^{\rm reg})
\]
contains the normal abelian subgroup $A_0\triangleleft G$.
Note that $A_0$ has index at most $c(n):=c'(n)c''(n)$ in $G$.
This finishes the proof.
\end{proof}

We finish this subsection by stating some applications of Theorem~\ref{thm:Jordan-reg-gen}. The proofs of the following corollaries are analogous to those in~\cite{BFMS20}, by replacing~\cite[Theorem 1]{BFMS20} with
Theorem~\ref{thm:Jordan-reg-gen}.
The generalizations of the results of~\cite{BFMS20} to the category of generalized pairs are straightforward.

The following corollary allows us to control the number of generators of the regional fundamental group of a generalized klt singularity.
Furthermore, it says that the number of generators of large order is bounded above by the regularity.

\begin{corollary}
Let $n$ and $r$ be two positive integers with $r\leq n-1$.
There exists a constant $b(n)$, only depending on $n$, satisfying the following.
Let $x\in (X,B+M;x)$ be a $n$-dimensional $r$-regular gklt singularity.
Then, the fundamental group $\pi_1^{\rm reg}(X,B+M;x)$ is generated by at most $b(n)$ elements.
Furthermore, there are at most $r$ elements with order larger than $b(n)$.
\end{corollary}

Analogously,
the following corollary
allows us to control the number of generators of the local Class group.
Furthermore, we conclude that the number of elements in the local Class group which have large order is bounded by the regularity.

\begin{corollary}
Let $n$ and $r$ be two positive integers with $r\leq n-1$.
There exists a constant $b(n)$, only depending on $n$, satisfying the following.
Let $x\in (X,B+M;x)$ be a $n$-dimensional $r$-regular gklt singularity.
Then, the torsion subgroup
${\rm Cl}(X;x)_{\rm tor}$ is generated by at most $b(n)$ elements.
Furthermore, there are at most $r$ elements with order larger than $b(n)$.
\end{corollary}

In~\cite[Theorem 1.10]{GKP16}, 
the authors proved the existence of a simultaneous index one cover for $n$-dimensional klt singularities.
In~\cite[Theorem 4]{BFMS20}, 
the authors proved that such simultaneous index one cover can be achieved by $n$ cyclic covers followed by a cover of degree bounded by a constant on $n$.
The following corollary states that for a $n$-dimensional $r$-regular gklt singularity, 
such simultaneous index one cover can be achieved by $r$ cyclic covers followed by a cover of degree at most $b(n)$.

\begin{corollary} 
Let $n$ and $r$ be two positive integers with $r\leq n-1$.
There exists a constant $b(n)$,
which only depends on $n$,
that satisfies the following.
Let $(X,B+M;x)$ be a $n$-dimensional $r$-regular gklt singularity.
Then, there exists a sequence of finite quasi-\'etale covers
\[
\xymatrix{
X &
\ar[l] X_1 &
\ar[l] X_2 &
\ar[l] \dots &
\ar[l] X_{j},
}
\]
with $j\leq r+1$, 
that satisfies the following conditions:
\begin{enumerate}
    \item For any $1 \leq i \leq j$, $x\in X$ has a unique preimage $x_i\in X_i$;
    \item every integral 
    $\qq$-Cartier divisor on $X_j$ is a Cartier divisor;
    \item the quotient $X\longleftarrow X_1$ has degree at most $b(n)$; and 
    \item for any $1 \leq i \leq j-1$, the quotient $X_i \longleftarrow X_{i+1}$ is cyclic.
\end{enumerate}
In particular, $x_{j}\in X_{j}$ is a canonical singularity.
\end{corollary}

\section{Examples and Questions}

In this section, 
we provide some examples showing that the hypotheses of our theorems cannot be weakened.
Furthermore, we propose some questions about finite actions on rationally connected varieties.

\begin{example}
\label{ex:elliptic}
{\em 
In this example, we show that the Fano type assumption in Theorem~\ref{introthm:rank-vs-reg-global} is indeed necessary.
Analogously, the statement of Theorem~\ref{introthm:rank-vs-reg-local} does not hold for non-lc singularities or for lc singularities that do not support klt singularities.

First, we consider an elliptic curve $E$.
The curve $E$ is Calabi-Yau and $K_E\sim 0$.
Furthermore, for each $n$, we can find $\zz_n\leqslant {\rm Aut}(E)$ by the multiplication of a $n$-torsion point. 
The regularity of $E$ equals $-\infty$, being the boundary empty.
We conclude that the statement of theorem~\ref{introthm:rank-vs-reg-global} does not hold if we drop the Fano type assumption.

Analogously, we can consider the cone $C$ over an elliptic curve $E$.
The germ $C$ has a strictly log canonical singularity. 
For each $n$, we still have 
$\zz_n\leqslant {\rm Aut}(C)$ by considering the action of $n$-torsion points in the cone.
However, the dual complex $\mathcal{D}(C)$ is empty, so the statement of Theorem~\ref{introthm:rank-vs-reg-local} does not hold for lc pairs which do not support klt singularities.

Finally, we consider the Fermat curve
\[
X_n:=\{x^n+y^n+z^n=0\}\subset \pp^2.
\]
Let $C_n$ be the cone over the Fermat curve.
Then, the curve $C_n$ is not log canonical, as the blow-up of the vertex extracts a unique non-lc place.
Observe that we have a monomorphism
$\zz_n^2\leqslant {\rm Aut}(C_n)$ by multiplication of roots of unity.
Thus, the germ $C_n$ has regularity one, while it admits a large finite abelian group of rank three.
This shows that the statement of Theorem~\ref{introthm:rank-vs-reg-local} does not hold if we allow non-lc singularities. 
}
\end{example}

\begin{example}\label{ex:no-control-B}
{\em 
In this example, we show that the control of the Cartier index in Theorem~\ref{introthm:rank-vs-reg-global} is necessary.
Consider $\pp^1$ with the action of $\zz_m$ by the multiplication by roots of unity.
Let $B$ be the reduced divisor defined as the orbit of $\{1\}$ under this action, i.e., 
the support of $B$ is the union of all $m$-roots of unity.
Consider $B_m:=\frac{2}{m}B$.
Then, the pair $(\pp^1,B_m)$ is log Calabi-Yau.
Furthermore, we have that 
\[
m(K_{\pp^1}+B_m) \sim 0.
\]
On the other hand, the pair $(\pp^1,B_m)$ is klt and admits a $\zz_m$-action.
In particular, its regularity equals $-\infty$.
So the statement of Theorem~\ref{introthm:rank-vs-reg-global} does not hold if we don't control the Cartier index of the log Calabi-Yau pair.
}
\end{example}

\begin{example}
\label{ex:no-control-M}
{\em 
In this example, 
we stress the importance of controlling the Cartier index of an equivariant generalized pair in the quotient where it descends. 
Let $\zz_m \leqslant \mathbb{G}_{m,t} < {\rm Aut}(\pp^1_t)$ and
$t_0,t_\infty \in \pp^1_t$ be the two torus invariant points.
Consider the quotient
\[
\phi\colon 
\pp^1_t \rightarrow \pp^1_s
\] 
by the cyclic action $\zz_m$.
Let $s_0$ (resp. $s_\infty)$
be the image of
$t_0$ (resp. $t_\infty)$.
Consider the $\zz_m$-equivariant generalized klt pair 
\begin{equation}
\label{eq:gen-pair-ex}
\left(\pp^1_t,
\frac{1}{2}(t_0+t_\infty)+
\frac{1}{2}(t_0+t_\infty) 
\right), 
\end{equation} 
where \[B_t=M_t=\frac{1}{2}(t_0+t_\infty).\]
Note that the previous generalized pair is log Calabi-Yau and gklt.
We have that 
\[
K_{\pp^1_t}+B_t+M_t = 
\phi^*\left( 
K_{\pp^1_s}+\frac{2m-1}{2m}(s_0+s_\infty) + \frac{1}{2m}(s_0+s_\infty)
\right) 
\]
where 
\[
B_{s,m}= \frac{2m-1}{2m}(s_0+s_\infty)\text{ and }
M_{s,m}=\frac{1}{2m}(s_0+s_\infty).
\]
Note that the following conditions are satisfied:
\begin{enumerate}
    \item The coefficients of $B_t$ belongs to $\{\frac{1}{2}\}$,
    \item the Cartier index of $M_t$ is $2$, and
    \item the coefficients of $B_{s,m}$ belongs to a set of rational numbers satisfying the descending chain condition.
\end{enumerate}
However, the Cartier index of $M_{s,m}$ equals $2m$.
Indeed, the generalized pair~\eqref{eq:gen-pair-ex} admits $\zz_m$ actions for $m$ arbitrarily large, although its dual complex is trivial. Thus, this example shows that 
the statement Theorem~\ref{introthm:rank-vs-reg-global} fails for generalized pairs, even if we control the Cartier index of the moduli part in the model where it descends.
To remedy this situation,
we need to control the Cartier index of the moduli part in the quotient where it descends.
}
\end{example}

We finish this section by proposing two questions about finite actions on rationally connected varieties.

\begin{question}
\label{question-1}
{\em 
Let $X$ be a $n$-dimensional rationally connected variety.
Assume that $X$ is $\qq$-complemented, i.e., 
there exists a boundary $B$ on $X$ so that
$(X,B)$ is log canonical and $K_X+B\sim_\qq 0$.
Assume that $\zz_m^r\leqslant {\rm Aut}(X)$ for $m$ large enough compared with the dimension. 
Can we find $\Gamma\geq 0$ on $X$ so that $(X,\Gamma)$ is log Calabi-Yau of regularity at least $r-1$?
}
\end{question}

The previous question is essentially Theorem~\ref{introthm:rank-vs-reg-global} for rationally connected varieties.
It is natural to expect a positive answer for the previous question. Indeed, log Calabi-Yau rationally connected varieties behave similar to Fano type varieties (see, e.g.,~\cite{BDCS20}).
The main obstruction in giving an affirmative answer to the previous question, is the existence of bounded complements for rationally connected varieties.
In~\cite{FMX19}, the authors prove the existence of bounded complements for log Calabi-Yau pairs of dimension $3$.
Combining the techniques of this article and~\cite[Theorem 1]{FMX19}, we expect a positive answer
to Question~\ref{question-1} in the case of low dimensional rationally connected varieties.

\begin{question}
\label{question-2} 
Let $(X,B)$ be a log Calabi-Yau pair of dimension $n$.
Assume that $X$ is of Fano type.
Let $A:=\zz_m^r \leqslant {\rm Aut}(X,B)$ for $m$ large compared with the dimension.
Can we find an $A$-equivariant dlt modification $(Y,B_Y)$ of $(X,B)$ so that:
\begin{enumerate}
    \item There exists an $A$-invariant log canonical center $W$ of $(Y,B_Y)$,
    \item $A$ acts as the identity on $W$, and 
    \item $W$ is rationally connected.
\end{enumerate}
\end{question}

The previous question is Theorem~\ref{thm:rank-vs-reg}
with the extra outcome that the fixed log canonical center $W$ is rationally connected.
Following the philosophy of Question~\ref{question-1}, 
one could further ask Question~\ref{question-2} for rationally connected varieties.
As shown in Theorem~\ref{thm:rank-vs-reg} and its proof, 
finite actions on Fano type varieties resemble torus actions on Fano type varieties.
Likely, we expect that finite actions on rationally connected varieties resemble torus actions on rationally connected varieties.
In the case of a torus action $\mathbb{T}$ on a rationally connected variety $X$, we can find a 
$\mathbb{T}$-equivariant birational morphism
$\widetilde{X}\rightarrow X$ so that $\widetilde{X}$ admits a good quotient for the $\mathbb{T}$-action.
We denote by $\widetilde{X}\rightarrow Z$ such good quotient.
Note that $Z$ is again rationally connected,
being the image of a rationally connected variety.
Furthermore, the morphism $\widetilde{X}\rightarrow Z$ admits a section. 
This section is fixed pointwise by $\mathbb{T}$.
Moreover, it is a log canonical center of every
$\mathbb{T}$-invariant log Calabi-Yau structure on $X$.
The fixed log canonical center $W$ would play the role of this section in the statement of Question~\ref{question-2}.

To conclude this section, we explain how a positive answer to Question~\ref{question-2} would have an application to the study of klt singularities.
Let $(X,\Delta;x)$ be a $n$-dimensional klt singularity of regularity $r$ so that $A:=\zz_m^{r+1}\leqslant \pi_1^{\rm reg}(X,\Delta;x)$ for $m$ large enough.
Due to Theorem~\ref{introthm:Jordan-reg}, this is the largest fundamental group that we can have among $n$-dimensional $r$-regular klt singularities.
By Theorem~\ref{thm:gen-group-toroidal}, we could find a projective birational morphism $Y\rightarrow X$ satisfying:
\begin{enumerate}
    \item the exceptional locus of $\pi$ is divisorial and consists of $r$ prime divisors $E_1,\dots, E_r$,
    \item the intersection $E_1\cap \dots \cap E_r=Z$ is irreducible, and
    \item there exists a boundary $B_Y$ supported on $E_1\cup\dots\cup E_r$ so that the homomorphism
    \[
    \pi_1^{\rm reg}(Y,B_Y+\Delta_Y;\eta_Z)\rightarrow \pi_1^{\rm reg}(X,B;x)
    \]
    is almost surjective.
\end{enumerate}
In particular, in this case, $Z$ is a minimal log canonical center.
Furthermore, by Theorem~\ref{thm:gen-group-toroidal}, we know that there exists a boundary $\Gamma_Y$ on $Y$ so that
$(Y,\Gamma_Y+E_1+\dots+E_r)$ is log canonical, $Z$ is a minimal log canonical center and 
\[
N(K_Y+\Gamma_Y+E_1+\dots+E_r)\sim 0.
\]
Assuming a positive answer to Question~\ref{question-2}, we could assume that $Z$ is rationally connected.
In particular, the pair
\[
K_Z+\Gamma_Z \sim_\qq (K_Y+\Gamma_Y+E_1+\dots+E_r)|_Z
\]
is klt and $N$-complemented.
By~\cite[Conjecture 1.5]{BDCS20}, we expect that such $Z$ belongs to a bounded family.
Hence, the blow-up $Y\rightarrow X$ would generalize the extraction of the unique potential log canonical place of an exceptional klt singularity.
It is natural to expect the singularity $(X,\Delta;x)$ to degenerate to a multi-graded orbifold cones over $Z$ (see, e.g.,~\cite{Mor18b,HLM20}).

\bibliographystyle{habbrv}
\bibliography{bib}

\end{document}